\documentclass[12pt]{amsart}
\usepackage[hidelinks]{hyperref}
\usepackage[T1]{fontenc}

\usepackage[english]{babel}

\usepackage{tikz-cd}
\usepackage{geometry}

\usepackage{amsmath}
\usepackage{amsfonts}
\usepackage{amssymb}
\usepackage{amsthm}
\usepackage{graphicx}
\usepackage{dsfont}

\usepackage{tikz}
\usepackage{chngcntr}
\counterwithin{figure}{section}

\usepackage{float}
\usetikzlibrary{calc,decorations.pathmorphing,shapes}
\usepackage{soul} 
\usepackage{url}

\usepackage{multicol}
\usepackage{lineno}
\usepackage[linewidth=1pt]{mdframed}
\usepackage{xcolor}
\usepackage{todonotes}
\usepackage{menukeys}
\usepackage{enumitem}

\usepackage[capitalize]{cleveref}
\newcommand{\myref}[1]{\hyperref[#1]{\namecref{#1} \ref{#1}}}

\theoremstyle{definition}
\newtheorem{thm}{Theorem}[section]

\newtheorem{defn}[thm]{Definition}
\newtheorem{lemma}[thm]{Lemma}
\newtheorem{prop}[thm]{Proposition}
\newtheorem{ex}[thm]{Example}
\newtheorem*{thm211}{Theorem \ref{THM_CombinatorialCharacterization_SpreaoutMap}}
\newtheorem*{thm315}{Theorem \ref{THM_TreeCollection}}
\newtheorem*{cor316}{Corollary \ref{COR_RaySPaceChar}}
\newtheorem*{thm318}{Theorem \ref{COR_EndSpaceChar}}
\newtheorem*{thm323}{Theorem \ref{THM_CompletelyUltrametrizableChar}}
\newtheorem*{thm48}{Theorem \ref{COR_QuasiRaySPaceChar}}
\newtheorem*{thm510}{Theorem \ref{THM_ScatteredIsEnd}}
\newtheorem*{cor511}{Corollary \ref{THM_CountSpecRaySubSpaceChar}}
\newtheorem*{thm64}{Theorem \ref{COR_ConvSeq&UncountTopsNotProd}}
\newtheorem*{thm69}{Theorem \ref{COR_ConvSeq&Omega1NotProd}}
\newtheorem*{question612}{Question \ref{QUESTION_FirsCountRayProd0}}

\newtheorem{cor}[thm]{Corollary}
\newtheorem{claim}[thm]{Claim}
\newtheorem{rmk}[thm]{Remark}
\newtheorem{question}[thm]{Question}

\newtheorem*{Counter-Examples from Scales}{Counter-Examples from Scales}
\newtheorem*{Pseudo-Power}{Pseudo-Power Theorem}

\numberwithin{equation}{thm}

\newcommand{\set}[1]{\left\{ {#1} \right\}}
\newcommand{\seq}[1]{{\left( {#1} \right)}}
\newcommand{\com}[1]{}
\newcommand{\interior}[1]{%
  {\kern0pt#1}^{\mathrm{o}}%
}

\DeclareMathOperator{\succs}{succ}
\DeclareMathOperator{\Rtop}{top}
\newcommand{\downward}[1]{\left\lceil {#1} \right\rceil}

\newcommand{\sdom}[1]{\succs(\dom({#1}))}

\DeclareMathOperator{\dom}{dom}

\DeclareMathOperator{\diam}{diam}

\newcommand{\ali}{\textsc{Alice}}
\newcommand{\bob}{\textsc{Bob}}

\newcommand{\TreeGame}{\mathrm{RayGame}}
\newcommand{\PTreeGame}{\mathrm{QuasiRayGame}}

\newcommand{\RR}{{\mathbb{R}}}

\newcommand{\ZZ}{{\mathbb{Z}}}

\newcommand{\bbS}{{\mathbb{S}}}
\newcommand{\PP}{{\mathbb{P}}}

\newcommand{\mB}{{\mathcal{B}}}
\newcommand{\mC}{{\mathcal{C}}}

\newcommand{\mF}{{\mathcal{F}}}

\newcommand{\mP}{{\mathcal{P}}}

\newcommand{\mR}{{\mathcal{R}}}
\newcommand{\mS}{{\mathcal{S}}}

\newcommand{\mU}{{\mathcal{U}}}

\newtheorem*{thm*}{Pseudo-Power Theorem}
\makeatletter
\AtBeginEnvironment{thm*}{\def\@currentlabel{Pseudo-Power Theorem}}
\makeatother

\newtheorem*{thm**}{Counter-Examples from Scales}
\makeatletter
\AtBeginEnvironment{thm**}{\def\@currentlabel{Counter-Examples from Scales}}
\makeatother

\DeclareFontFamily{U}{rcjhbltx}{}
\DeclareFontShape{U}{rcjhbltx}{m}{n}{<->rcjhbltx}{}
\DeclareSymbolFont{hebrewletters}{U}{rcjhbltx}{m}{n}

\newtheorem*{theorem*}{Theorem}
\newtheorem*{conjecture*}{Conje
cture}
\DeclareMathSymbol{\lamed}{\mathord}{hebrewletters}{108}
\DeclareMathSymbol{\mem}{\mathord}{hebrewletters}{109}
\DeclareMathSymbol{\ayin}{\mathord}{hebrewletters}{96}
\DeclareMathSymbol{\tsadi}{\mathord}{hebrewletters}{118}
\DeclareMathSymbol{\qof}{\mathord}{hebrewletters}{113}
\DeclareMathSymbol{\shin}{\mathord}{hebrewletters}{152}
\DeclareMathSymbol{\resh}{\mathord}{hebrewletters}{114}

\usepackage{lineno}
\title[Ray and End Spaces]{Ray and End Spaces: Characterizations and Classification up to Homeomorphism}

\author{Matheus Duzi}
\address{Instituto de Ciências Matemáticas e de Computação, Universidade de São Paulo}
\email{matheus.duzi.costa@usp.br}
\author{Gabriel Fernandes $^\ast$ }
\address{Instituto de Ciências Matemáticas e de Computação, Universidade de São Paulo}
\email{fernandes@icmc.usp.br}
\thanks{$^\ast$ The second named author acknowledges the support of Fundação de Amparo à Pesquisa do Estado
de São Paulo (FAPESP) through grant no. 2025/09425-1.}
\author{Paulo Magalhães Júnior}
\address{Instituto Federal do Rio Grande do Norte, Campus Currais Novos}
\email{paulo.magalhaes@ifrn.edu.br}

\date{}

\begin{document}
\begin{abstract}
    We provide a combinatorial characterization for pairs of order-theoretic trees with homeomorphic ray spaces, answering an open problem proposed by Kurkofka and Pitz. 
    This solution is inspired by the introduction of a transfinite topological game, which allows us to characterize not only ray spaces through the existence of winning strategies for one of the players, but also their homeomorphism classes. 
    As applications of these results, we obtain a new topological characterization for graph-theoretic end spaces (thus obtaining yet another solution to a recently solved problem of Diestel), as well as for edge-end spaces and completely ultrametrizable spaces. 
    We also introduce a generalization of the class of ray spaces (which is strict, as witnessed by the Sorgenfrey line). 
    Furthermore, we establish that, for subspaces with cardinality less than continuum of end spaces, the scattered property is equivalent to the property of being, itself, an end space. 
    At last, we determine that ray spaces in a couple of classes fail to have their product with any non-discrete space as a ray space. 

\end{abstract}
\maketitle
\section{Introduction}
The so-called \emph{ends} of graphs, as classes of rays which cannot be separated by finitely many vertices, were first introduced in \cite{halin} and have since been used to study important problems in infinite combinatorics, as can be seen, e.g., in: 

\begin{itemize}
    \item The network flow theory is extended in \cite{maxflowmincut} to infinite graphs by using the concept of ends to understand and control flows that ``escape to infinity'';
    \item It is shown in \cite{arboricityandtree-packing} that ends are essential for extending the Nash-Williams tree-packing theorem to locally finite graphs;
    \item Ends are demonstrated in \cite{cycle-cocycle} to be a necessary component for the Faithful Cycle Cover Conjecture to hold in locally finite graphs.
\end{itemize}

The study of infinite trees and their associated topological spaces, on the other hand, has been a central theme in infinite combinatorics and set-theoretic topology --
recall that an \emph{order tree} is a partially ordered set $(T, \le)$ with a minimum $r$ (the root) such that, for every node $t \in T$, the set of predecessors 
\[
    \lceil t \rceil := \{s \in T : s \le t\}
\]
is well-ordered.
 A primary object of interest in this field is the \emph{ray space} $\mR(T)$, which consists of the set of all \emph{rays} (i.e., downward closed chains in $T$ with no maximal element) of a tree $T$, equipped with the subspace topology of the power set $\mP(T)\simeq 2^T$. 
As established in recent literature (such as in  \cite{pitz2023characterisingpathraybranch} and \cite{carvalho2026coveringpropertiesendray}), these spaces provide a robust framework for investigating the interplay between discrete combinatorial structures and continuous topological properties.
Indeed, we highlight the following example, which determines an intimate relation between the class of ray spaces of trees and the class of end spaces of graphs:
\begin{thm}[Theorem 1 in \cite{representationtheorem}]\label{THM_PitzRepresentation}
    A space $X$ is homeomorphic to the end space of some graph if and only if it is homeomorphic to the ray space of some special tree.
\end{thm}
Here, a tree $T$ is called \emph{special} if it can be obtained as a countable union of antichains (i.e., sets of pairwise incomparable elements of $T$) and the set of ends $\Omega(G)$ of a graph $G$ is equipped with the topology generated by the basic open sets of the form, for $\varepsilon \in \Omega(G)$ and finite set of vertices $F$ in $G$,
\[
    \Omega(\varepsilon, F) = \set{\eta\in \Omega(G): \text{ all $R\in \varepsilon$ and $R'\in \eta$ cannot be separated by $F$}}.
\]

One of the fundamental problems concerning ray spaces is their classification: Kurkofka and Pitz explicitly ask, in Open problem 5.5 of \cite{representationtheorem}, how to combinatorially determine when two trees $T$ and $T'$ yield homeomorphic ray spaces. 
A significant contribution of this paper is providing a complete answer to this question, for which we dedicate \myref{SEC_CombinatorialCharacterization}.  

Indeed, for a tree $T$, $t\in T$, finite $F\subseteq T$ and a ray $R\subseteq T$, consider
\begin{align*}
    \succs(t) &:= \set{s>t: \downward{s} = \downward{t}\cup \{s\} },\\
    \succs(T) &:= \{r\}\cup \left(\bigcup_{t\in T}\succs(t)\right),\\
    \Rtop(R) &:= \set{s\in T: \downward{s} = R\cup\{s\}},\\
    [t,F] &:= \set{R\in \mR(T): t\in R \text{ and } \forall s\in F(s\notin R)},
\end{align*}
where $r\in T$ is the root of $T$.
We then introduce the partially ordered set $\bbS(T)$ associated to a tree $T$, whose elements are certain collections of pairs $(t,F)$, with $t\in T$ and finite $F\subset T$. 
Thus, by considering 
\[
    \bigvee \mS := \bigcup_{(t,F)\in \mS}[t,F] \subseteq \mR(T)
\]
for each $\mS\in \bbS(T)$, we are able to provide:

\begin{thm211}
    Suppose $T$ and $T'$ are pruned trees. Then $\mR(T)\simeq \mR(T')$ if and only if there is a map $\psi\colon T\to \bbS(T')$ satisfying all of the following properties:
    \begin{enumerate}[label=(\roman*)]
        \item If $r\in T$ is the root of $T$, then $\psi(r) = \{(r',\emptyset)\}$, where $r'\in T'$ is the root of $T'$.
        \item If $t\le s$ in $T$, then $\psi(t)\le \psi(s)$ in $\bbS(T')$.
        \item If $t,s\in T$ are incompatible, then $\psi(t)$ and $\psi(s)$ are ray-disjoint.
        
        \item If $t\in T$, then
        \[\bigvee \psi(t) = \bigcup_{s\in \succs(t)}\bigvee\psi(s)\]
        
        \item If $R$ is a ray in $T$, then there is a unique ray $R'$ in $T'$ such that 
        \begin{equation*}
            \bigcap_{t\in R}\left(\bigvee\psi(t)\right)\setminus \bigcup_{s\in \Rtop(R)}\left(\bigvee\psi(s)\right)  = \{R'\}
        \end{equation*}
        and, for each $t'\in R'$ and each finite $F'\subseteq \Rtop(R')$, there are $t\in R$ and finite $F\subseteq \Rtop(R)$ such that
        \begin{equation*}
            \bigvee\psi(t)\setminus \bigcup_{s\in F}\left(\bigvee\psi(s)\right) \subseteq [t',F'].
        \end{equation*}
    \end{enumerate}
\end{thm211}

Let us say that a space $X$ is a \emph{ray space} if there is some tree $T$ for which $X$ is homeomorphic to $\mR(T)$. 
It was then shown in \cite{pitz2023characterisingpathraybranch} (Theorem 3.6) that a Hausdorff space $X$ is a ray space if and only if $X$ admits a subbase satisfying a certain list of properties. 

On the other hand, the initial idea for our characterization in \myref{THM_CombinatorialCharacterization_SpreaoutMap} comes from a topological characterization for ray spaces of our own, presented in \myref{SEC_TopologicalCharacterization}: we present, in \myref{DEF_TreeGame}, a transfinite topological game $\TreeGame(X)$ to be played over the space $X$.
Thus, letting $\Sigma(X)$ be the (possibly empty) collection of all winning strategies for $\ali$ in said game, we prove the following theorem.
\begin{thm315}
    For a given space $X$, the collection 
    \[\set{\sdom{\sigma}: \sigma\in \Sigma(X)}\]
    comprises, up to isomorphism, the class of all pruned trees whose ray space is homeomorphic to $X$.
\end{thm315}
Here, $\dom(\sigma)$ denotes the tree consisting of every sequence of $\bob$'s moves in the domain of $\sigma$, ordered by inclusion.
With this, we thus obtain:
\begin{cor316}
    For a non-empty topological space $X$, the following are equivalent:
    \begin{enumerate}[label=(\alph*)]
        \item $X$ is a ray space.
        \item $\ali$ has a winning strategy in $\TreeGame(X)$.
    \end{enumerate}
\end{cor316}

We are further able to obtain a handful of applications from this equivalence, the main of which we briefly elaborate in this section. 

For starters, let us say that a space $X$ is an \emph{end space} if it is homeomorphic to the end space of some graph $G$.
Recall that Diestel asked in Problem 5.1 of \cite{Diestelquestion} for topological properties over a space $X$ which could characterize $X$ being an end space.
An answer to this question came in \cite{pitz2023characterisingpathraybranch} (Theorem 3.6) with the existence of a subbase satisfying, again, a certain list of properties.
We can also find an alternative answer in
\cite{PauloGustavoDavide} (Theorem 2.1), where one of the properties of the first characterization is replaced by the existence of a winning strategy for one of the players in a certain topological game over the space $X$.
With our \myref{THM_TreeCollection}, on the other hand, we are able to provide a completely different answer to Diestel's problem:
\begin{thm318}
    A topological space $X$ is an end space if and only if there is a winning strategy $\sigma$ for $\ali$ in $\TreeGame(X)$ such that $\sdom{\sigma}$ is a special tree.
\end{thm318}

As completely ultrametrizable spaces comprise another class of spaces which are homeomorphic to specific kinds of ray spaces, we also obtain:

\begin{thm323}
    A topological space $X$ is completely ultrametrizable if and only if there is some strategy $\sigma$ for $\ali$ in $\TreeGame(X)$ which always wins at the $\omega$th inning.
\end{thm323}

\myref{COR_RaySPaceChar} further allows us to generalize the concept of ray spaces in \myref{SEC_QuasiRay}: let us say that a topological space $(X,\tau)$ is a \emph{quasi-ray space} if there is some coarser topology $\tau'\subseteq \tau$ such that $(X,\tau')$ is a ray space.
Then, by defining a transfinite topological game $\PTreeGame(X)$ which is ``easier'' for $\ali$ in comparison to $\TreeGame(X)$, we show that
\begin{thm48}
    For a non-empty topological space $X$, the following are equivalent:
    \begin{itemize}
        \item[(a)] $X$ is a quasi-ray space.
        \item[(b)] $\ali$ has a winning strategy in $\PTreeGame(X)$.
    \end{itemize}
\end{thm48}
We are thus able to determine that a few spaces, such as the Sorgenfrey line, $2^{\omega_1}$ or subspaces of ordinals, are quasi-ray spaces, despite not satisfying the strict requirements to be ray spaces.

Finding the class of subspaces of end spaces that are themselves end spaces is a problem that has been tackled over the last few years. 
It was shown in \cite{approximating} that every open subset of an end space is also an end space, while it was demonstrated in \cite{pitz2023characterisingpathraybranch} that every closed subspace is an end space as well. Most recently, it was proven in \cite{PauloGustavoDavide} that these results can be generalized to  $G_\delta$ subspaces, where
this problem was recently formalized in Question 4.6.
The applicability of our characterizations is then further highlighted in \myref{SEC_RaySubspaces}, where we partially answer this open problem with 
\begin{thm510}
    Let $T$ be a special tree and suppose $X\subseteq \mR(T)$ is scattered. Then $X$ is an end space. 
\end{thm510}

\begin{cor511}
    Let $G$ be a graph and suppose $X \subseteq \Omega(G)$ satisfies $|X| < 2^{\aleph_0}$. Then the following are equivalent:
    \begin{enumerate}[label=(\alph*)]
        \item $X$ is scattered (i.e., every non-empty $Y\subseteq X$ has an isolated point).
        \item $X$ is an end space.
    \end{enumerate}
\end{cor511}

Furthermore, continuing the trend started in \cite{PauloGustavoDavide}, we determine that the finite product of ray spaces within certain classes is never a ray space. Namely, we show that
\begin{thm64}
    Suppose $T$ is a pruned tree with a ray containing uncountably many tops. Then the following are equivalent for a tree $T'$: 
    \begin{enumerate}[label=(\alph*)]
        \item $\mR(T')$ is discrete.
        \item $\mR(T)\times \mR(T')$ is a ray space.
    \end{enumerate}
\end{thm64}
\begin{thm69}
    Suppose $T$ is a pruned tree with an uncountable ray. Then the following are equivalent for a tree $T'$: 
    \begin{enumerate}[label=(\alph*)]
        \item $\mR(T')$ is discrete.
        \item $\mR(T)\times \mR(T')$ is a ray space.
    \end{enumerate}
\end{thm69}

Hence, we narrow down the search for the maximal class of ray spaces which is closed under finite products to the question:

\begin{question612}
    Suppose $T$ and $T'$ are trees in which:
    \begin{enumerate}[label=(\roman*)]
        \item every branch is countable;
        \item every ray has at most countably many tops.
    \end{enumerate}
    Is $\mR(T)\times \mR(T')$ then a ray space?
\end{question612}

Finally, we end this paper with \myref{SEC_FinalRmk} by presenting some concluding remarks and open questions regarding the determinacy of the games introduced.

\subsection{Preliminaries and notation}

 A down-closed chain in $T$ is a path, and a path without a maximal element is a ray. The height of a node $t \in T$, denoted $\text{ht}(t)$, is the order type of its set of strict predecessors $\lceil \mathring{t}\rceil := \lceil t \rceil \setminus \{t\}$.

If $t < t'$ and no element $x \in T$ satisfies $t < x < t'$, then $t'$ is an immediate successor of $t$. The set of immediate successors of $t$ is $\text{succ}(t)$, and the successor set of the tree $T$ with root $r$ is defined as $\text{succ}(T) = \{r\} \cup \bigcup_{t \in T} \text{succ}(t)$. This collection contains the root along with all non-limit elements of $T$.

For a ray $R \subseteq T$, a node $t \in T$ is a top of $R$ if $\lceil \mathring{t}\rceil = R$. The set of all tops of $R$ is denoted $\text{top}(R) = \{ t \in T : \lceil \mathring{t}\rceil = R \}$, or equivalently $\lceil t\rceil=R\cup\{t\}$. In a general tree order, a ray may have zero, one, or multiple tops.

For $s\in T$, we write $\lfloor s\rfloor:=\{t\in T:s\le t\}$ for the upper cone of $s$.

For any ordinal $\alpha$, the restriction of $T$ to heights strictly less than $\alpha$ is denoted $T \upharpoonright \alpha = \{ t \in T : \text{ht}(t) < \alpha \}$, which forms a rooted subtree under the induced ordering. Similarly, the $\alpha$-th level of the tree is denoted $T(\alpha) = \{ t \in T : \text{ht}(t) = \alpha \}$.

We shall write $T\cong T'$ when $T$ and $T'$ are order-isomorphic. 

By a \emph{space}, unless otherwise specified, we mean a \emph{topological space}. When $X$ and $Y$ are homeomorphic spaces, we shall write $X\simeq Y$.
\section{A combinatorial characterization for homeomorphic ray spaces}\label{SEC_CombinatorialCharacterization}
To provide a combinatorial characterization of when two trees have homeomorphic ray spaces, we rely on a poset constructed from their partial orders. We begin by establishing the preliminary definitions required for this construction.

\begin{defn}\label{DEFN_Spreadout}
    For a tree $T$, we denote by $\PP(T)$  the set of all pairs $(t,F)$, such that $t\in T$ and $F\subset T$ is a finite set of tops in $T$ which are above $t$. 

    \sloppy We say that a collection $\mC\subseteq \PP(T)$ is \emph{pairwise disjoint} if, for all distinct $(t,F),(t',F')\in \mC$, 
    \[
       t'\le t \implies \exists s'\in F' \, (s'\le t).
    \]
    We then say that two pairwise disjoint collections $\mC,\mC'\subseteq\PP(T)$ are \emph{ray-disjoint} if $\mC\cap \mC' = \emptyset$ and $\mC\cup\mC'$ is pairwise disjoint as well.

    For a collection $\mC\subseteq \PP(T)$, we write
    \[
        \bigvee \mC = \bigcup_{(t,F)\in \mC}[t,F]\subseteq \mR(T).
    \]

    We say that a non-empty family $\mS \subseteq \PP(T)$ is \emph{spread out} if it is pairwise disjoint and for all $R\in \mR(T)$:

    \begin{itemize}
        \item[(a)] $(\bigcup\{\lceil t \rceil : \exists F ( ( t,F) \in \mS) \wedge  R \not\subseteq \lceil t \rceil  \})\cap R$ is not cofinal in $R$.
    \end{itemize}
    \begin{itemize}
        \item[(b)] $(\bigcup\{\lceil t \rceil : \exists F\ ( t,F) \in \mS\})\cap \Rtop(R)$ is finite.
       
    \end{itemize}
\end{defn}

To motivate the choice of vocabulary, we point out the following characterizations (the proofs for the trivial ones are left for the reader):

\begin{prop}\label{PROP_PairwiseDisjointPairCollection}
    Suppose $T$ is a tree and $(t,F),(t',F')\in \PP(T)$. Then $[t,F]\cap[t',F'] = \emptyset$ if and only if, one of the following holds:\begin{itemize}
        \item $t$ and $t'$ are incompatible.
        \item $t'\le t $ and there is $ s'\in F'$ such that  $ s'\le t$. 
        \item $t \le t'$ and there is $s \in F$ such that $s \le t'$.
    \end{itemize}  
    
    \sloppy In particular, $\mC\subseteq \PP(T)$ is pairwise disjoint if and only if the elements of $\set{[t,F]:(t,F)\in \mC}$ are pairwise disjoint.
\end{prop}

\begin{cor}\label{PROP_RayDisjointCollection}
    Two pairwise disjoint collections $\mC$ and $\mC'$ in a tree $T$ are ray-disjoint if and only if 
    \[\left(\bigvee\mC\right) \cap \left(\bigvee\mC'\right) = \emptyset.\]
\end{cor}

\begin{prop}\label{PROP_PairwiseDisjoint_InwardSpreadout}
    Suppose $T$ is a tree and $\mC\subseteq \PP(T)$ is pairwise disjoint. If $R\in \bigvee\mC$, then $R$  satisfies (a) and (b) of \myref{DEFN_Spreadout} for $\mC$.
\end{prop}
\begin{proof}
    Let $(t_R,F_R)\in \mC$ be such that $R\in [t_R,F_R]$. 
    Because $\mC$ is pairwise disjoint, it follows from \myref{PROP_PairwiseDisjointPairCollection} that such $(t_R,F_R)$ is unique.
    Furthermore, every $(t,F)\in \mC$ distinct from $(t_R,F_R)$ is such that either
    $t_R \not\le t$ or there is some $s'\in F_R$ with $s'\le t$.
    Hence:
    \begin{itemize}
        \item (a) in \myref{DEFN_Spreadout} holds for $R$, since every $(t,F)\in \mC$ with $t$ not being above $R$ must be such that $t \not\ge t_R\in R$.
        \item (b) in \myref{DEFN_Spreadout}  holds for $R$, since every $(t,F)\in \mC$ such that $t$ is above $R$ (and, thus, above $t_R$) must be such that $t$ is above some member of the finite set $F_R$.
       
    \end{itemize}
\end{proof}

\begin{prop}\label{PROP_spreadout_clopen_eqv}
    Suppose $T$ is a tree and $\mS\subseteq \PP(T)$ is pairwise disjoint. Then $\mS$ is spread out in $T$ if and only if $\bigvee \mS$ is clopen in $\mR(T)$.
\end{prop}
\begin{proof}
    Suppose $\mS$ is spread out. Then, being the union of standard basic sets, $\bigvee \mS$ is open. 
    In order to show that $\bigvee\mS$ is also closed, let $R\in \mR(T)\setminus \bigvee\mS$. By item (a) in \myref{DEFN_Spreadout}, there must be an $s\in R$ such that 
    \[
        \lfloor s\rfloor \cap \set{t:(t,F)\in \mS} = \emptyset.
    \]
    Furthermore, item (b) in \myref{DEFN_Spreadout} tells us that there must be a finite $G\subseteq \Rtop(R)$ such that every element of $\set{t:(t,F)\in \mS}$ which is above $R$ is, in fact, above $G$.
    Thus, $R\in [s,G]\subseteq \mR(T)\setminus \bigvee\mS$ (so $\bigvee\mS$ is closed as well).

    \sloppy Now suppose $\bigvee\mS$ is clopen and let $R\in \mR(T)$ be given. 
    In view of \myref{PROP_PairwiseDisjoint_InwardSpreadout}, $R\in\bigvee\mS$ already tells us that $R$ satisfies (a) and (b) of \myref{DEFN_Spreadout} for $\mS$.
    
    So, without loss of generality, we may assume that $R\in \mR(T)\setminus \bigvee\mS$. 
    Since $\bigvee\mS$ is closed, there must then be a $t_R\in R$ and a finite $F_R\subseteq \Rtop(R)$ such that $[t_R,F_R]\subseteq \mR(T)\setminus \bigvee\mS$.
    In view of \myref{PROP_PairwiseDisjointPairCollection}, this means that $\mS' = \mS\cup\{(t_R,F_R)\}$ is pairwise disjoint.
    Hence, since $R\in \bigvee\mS'$, it follows again from \myref{PROP_PairwiseDisjoint_InwardSpreadout} that $R$ satisfies (a) and (b) of \myref{DEFN_Spreadout} for $\mS'$.
    But, because $\mS\subset \mS'$, this means that $R$ also satisfies (a) and (b) of \myref{DEFN_Spreadout} for $\mS$, as desired.
\end{proof}

\begin{cor}\label{COR_RayDisjointSpreadoutFamily}
    Suppose $T$ is a tree and $\mS,\mS'$ are spread out families in $T$. Then $\mS,\mS'$ are ray-disjoint if and only if $\mS\cap\mS' = \emptyset$ and $\mS\cup\mS'$ is also spread out.
\end{cor}
\begin{proof}
    It is clear that $\mS\cap\mS' = \emptyset$ and $\mS\cup\mS'$ being spread out implies that $\mS,\mS'$ are ray-disjoint.

    On the other hand, by \myref{PROP_spreadout_clopen_eqv} applied to $\mS$ and $\mS'$, we have that 
    \[
        \bigvee\left(\mS\cup\mS'\right) = \left(\bigvee\mS\right)\cup\left(\bigvee\mS'\right)
    \]
    is clopen. 
    
    Thus, if $\mS$ and $\mS'$ are ray-disjoint (meaning that $\mS\cap\mS' = \emptyset$ and $\mS\cup\mS'$ is pairwise disjoint), it follows from \myref{PROP_spreadout_clopen_eqv} that $\mS\cup \mS'$ is spread out.
\end{proof}

\begin{defn}\label{DEF_ShorterPair}
    Suppose $T$ is a tree and $(t,F),(t',F')\in \PP(T)$. We say that $(t,F)$ is \emph{shorter} than $(t',F')$ (or, equivalently, that $(t',F')$ is \emph{longer} than $(t,F)$), writing $(t,F)\succeq (t',F')$, if $t'\le t$ and $F'$ is above $F$, i.e., every member of $F'$ is above or equal to some member of $F$ (see \myref{FIG_ShorterPair}).
\begin{figure}
    \centering

\tikzset{every picture/.style={line width=0.75pt}} 

\begin{tikzpicture}[x=0.70pt,y=0.70pt,yscale=-1,xscale=1]

\draw  [dash pattern={on 0.84pt off 2.51pt}] (164.51,272.55) -- (24.27,12.56) -- (303.92,12.12) -- cycle ;
\draw [color={rgb, 255:red, 208; green, 2; blue, 27 }  ,draw opacity=1 ] [dash pattern={on 0.84pt off 2.51pt}]  (148.33,177.67) .. controls (127.67,225) and (185.67,188.33) .. (164.51,272.55) ;
\draw [shift={(148.33,177.67)}, rotate = 113.59] [color={rgb, 255:red, 208; green, 2; blue, 27 }  ,draw opacity=1 ][fill={rgb, 255:red, 208; green, 2; blue, 27 }  ,fill opacity=1 ][line width=0.75]      (0, 0) circle [x radius= 3.35, y radius= 3.35]   ;
\draw [color={rgb, 255:red, 74; green, 144; blue, 226 }  ,draw opacity=1 ] [dash pattern={on 0.84pt off 2.51pt}]  (175,139) .. controls (179.67,175.67) and (149.67,168.33) .. (148.33,177.67) ;
\draw [shift={(175,139)}, rotate = 82.75] [color={rgb, 255:red, 74; green, 144; blue, 226 }  ,draw opacity=1 ][fill={rgb, 255:red, 74; green, 144; blue, 226 }  ,fill opacity=1 ][line width=0.75]      (0, 0) circle [x radius= 3.35, y radius= 3.35]   ;
\draw [color={rgb, 255:red, 74; green, 144; blue, 226 }  ,draw opacity=1 ] [dash pattern={on 0.84pt off 2.51pt}]  (120.33,95) .. controls (127.14,105.55) and (148.01,123.73) .. (158.14,127.84) .. controls (168.27,131.95) and (175.65,134.43) .. (175,139) ;
\draw [shift={(120.33,95)}, rotate = 57.17] [color={rgb, 255:red, 74; green, 144; blue, 226 }  ,draw opacity=1 ][fill={rgb, 255:red, 74; green, 144; blue, 226 }  ,fill opacity=1 ][line width=0.75]      (0, 0) circle [x radius= 3.35, y radius= 3.35]   ;
\draw [color={rgb, 255:red, 74; green, 144; blue, 226 }  ,draw opacity=1 ] [dash pattern={on 0.84pt off 2.51pt}]  (171.67,87) .. controls (173,107) and (176.33,129.67) .. (175,139) ;
\draw [shift={(171.67,87)}, rotate = 86.19] [color={rgb, 255:red, 74; green, 144; blue, 226 }  ,draw opacity=1 ][fill={rgb, 255:red, 74; green, 144; blue, 226 }  ,fill opacity=1 ][line width=0.75]      (0, 0) circle [x radius= 3.35, y radius= 3.35]   ;
\draw [color={rgb, 255:red, 74; green, 144; blue, 226 }  ,draw opacity=1 ] [dash pattern={on 0.84pt off 2.51pt}]  (217.67,88.33) .. controls (219,108.33) and (176.33,110.33) .. (175,119.67) ;
\draw [shift={(217.67,88.33)}, rotate = 86.19] [color={rgb, 255:red, 74; green, 144; blue, 226 }  ,draw opacity=1 ][fill={rgb, 255:red, 74; green, 144; blue, 226 }  ,fill opacity=1 ][line width=0.75]      (0, 0) circle [x radius= 3.35, y radius= 3.35]   ;
\draw [color={rgb, 255:red, 74; green, 144; blue, 226 }  ,draw opacity=1 ] [dash pattern={on 0.84pt off 2.51pt}]  (125.12,101) .. controls (135.56,112.74) and (148.46,123.92) .. (158.14,127.84) .. controls (168.27,131.95) and (175.65,134.43) .. (175,139) ;
\draw [shift={(123.67,99.33)}, rotate = 49.14] [color={rgb, 255:red, 74; green, 144; blue, 226 }  ,draw opacity=1 ][line width=0.75]    (10.93,-4.9) .. controls (6.95,-2.3) and (3.31,-0.67) .. (0,0) .. controls (3.31,0.67) and (6.95,2.3) .. (10.93,4.9)   ;
\draw [color={rgb, 255:red, 74; green, 144; blue, 226 }  ,draw opacity=1 ] [dash pattern={on 0.84pt off 2.51pt}]  (172.15,93.76) .. controls (173.54,112.92) and (176.29,129.99) .. (175,139) ;
\draw [shift={(172,91.67)}, rotate = 86.19] [color={rgb, 255:red, 74; green, 144; blue, 226 }  ,draw opacity=1 ][line width=0.75]    (10.93,-4.9) .. controls (6.95,-2.3) and (3.31,-0.67) .. (0,0) .. controls (3.31,0.67) and (6.95,2.3) .. (10.93,4.9)   ;
\draw [color={rgb, 255:red, 74; green, 144; blue, 226 }  ,draw opacity=1 ] [dash pattern={on 0.84pt off 2.51pt}]  (215.41,95.43) .. controls (201.66,113.64) and (180.17,108.03) .. (175,119.67) ;
\draw [shift={(216.67,93.67)}, rotate = 123.91] [color={rgb, 255:red, 74; green, 144; blue, 226 }  ,draw opacity=1 ][line width=0.75]    (10.93,-4.9) .. controls (6.95,-2.3) and (3.31,-0.67) .. (0,0) .. controls (3.31,0.67) and (6.95,2.3) .. (10.93,4.9)   ;
\draw [color={rgb, 255:red, 208; green, 2; blue, 27 }  ,draw opacity=1 ] [dash pattern={on 0.84pt off 2.51pt}]  (85,48.33) .. controls (110.33,78.33) and (103.67,78.33) .. (120.33,95) ;
\draw [shift={(85,48.33)}, rotate = 49.82] [color={rgb, 255:red, 208; green, 2; blue, 27 }  ,draw opacity=1 ][fill={rgb, 255:red, 208; green, 2; blue, 27 }  ,fill opacity=1 ][line width=0.75]      (0, 0) circle [x radius= 3.35, y radius= 3.35]   ;
\draw [color={rgb, 255:red, 208; green, 2; blue, 27 }  ,draw opacity=1 ] [dash pattern={on 0.84pt off 2.51pt}]  (101.67,32.33) .. controls (114.33,63) and (120.33,77) .. (120.33,95) ;
\draw [shift={(101.67,32.33)}, rotate = 67.56] [color={rgb, 255:red, 208; green, 2; blue, 27 }  ,draw opacity=1 ][fill={rgb, 255:red, 208; green, 2; blue, 27 }  ,fill opacity=1 ][line width=0.75]      (0, 0) circle [x radius= 3.35, y radius= 3.35]   ;
\draw [color={rgb, 255:red, 208; green, 2; blue, 27 }  ,draw opacity=1 ] [dash pattern={on 0.84pt off 2.51pt}]  (136.33,28.33) .. controls (133,59) and (116.33,51) .. (116.33,69) ;
\draw [shift={(136.33,28.33)}, rotate = 96.2] [color={rgb, 255:red, 208; green, 2; blue, 27 }  ,draw opacity=1 ][fill={rgb, 255:red, 208; green, 2; blue, 27 }  ,fill opacity=1 ][line width=0.75]      (0, 0) circle [x radius= 3.35, y radius= 3.35]   ;
\draw [color={rgb, 255:red, 208; green, 2; blue, 27 }  ,draw opacity=1 ] [dash pattern={on 0.84pt off 2.51pt}]  (176.33,26.33) .. controls (166.33,71.67) and (223,46.33) .. (171.67,87) ;
\draw [shift={(176.33,26.33)}, rotate = 102.44] [color={rgb, 255:red, 208; green, 2; blue, 27 }  ,draw opacity=1 ][fill={rgb, 255:red, 208; green, 2; blue, 27 }  ,fill opacity=1 ][line width=0.75]      (0, 0) circle [x radius= 3.35, y radius= 3.35]   ;
\draw [color={rgb, 255:red, 208; green, 2; blue, 27 }  ,draw opacity=1 ] [dash pattern={on 0.84pt off 2.51pt}]  (104.6,39.93) .. controls (116.53,68.77) and (120.33,77.45) .. (120.33,95) ;
\draw [shift={(103.67,37.67)}, rotate = 67.56] [color={rgb, 255:red, 208; green, 2; blue, 27 }  ,draw opacity=1 ][line width=0.75]    (10.93,-4.9) .. controls (6.95,-2.3) and (3.31,-0.67) .. (0,0) .. controls (3.31,0.67) and (6.95,2.3) .. (10.93,4.9)   ;
\draw [color={rgb, 255:red, 208; green, 2; blue, 27 }  ,draw opacity=1 ] [dash pattern={on 0.84pt off 2.51pt}]  (135.23,35.64) .. controls (129.44,60.31) and (118.61,47.23) .. (116.33,69) ;
\draw [shift={(135.67,33.67)}, rotate = 101.71] [color={rgb, 255:red, 208; green, 2; blue, 27 }  ,draw opacity=1 ][line width=0.75]    (10.93,-4.9) .. controls (6.95,-2.3) and (3.31,-0.67) .. (0,0) .. controls (3.31,0.67) and (6.95,2.3) .. (10.93,4.9)   ;
\draw [color={rgb, 255:red, 208; green, 2; blue, 27 }  ,draw opacity=1 ] [dash pattern={on 0.84pt off 2.51pt}]  (175.52,33.12) .. controls (171.99,66.69) and (219.69,51.39) .. (171.67,87) ;
\draw [shift={(175.8,31)}, rotate = 99.07] [color={rgb, 255:red, 208; green, 2; blue, 27 }  ,draw opacity=1 ][line width=0.75]    (10.93,-4.9) .. controls (6.95,-2.3) and (3.31,-0.67) .. (0,0) .. controls (3.31,0.67) and (6.95,2.3) .. (10.93,4.9)   ;
\draw  [color={rgb, 255:red, 74; green, 144; blue, 226 }  ,draw opacity=1 ][dash pattern={on 0.84pt off 2.51pt}] (109.67,77) -- (231.67,77) -- (231.67,110.33) -- (109.67,110.33) -- cycle ;
\draw  [color={rgb, 255:red, 208; green, 2; blue, 27 }  ,draw opacity=1 ][dash pattern={on 0.84pt off 2.51pt}] (61.67,15.67) -- (207,15.67) -- (207,71.67) -- (61.67,71.67) -- cycle ;
\draw [color={rgb, 255:red, 208; green, 2; blue, 27 }  ,draw opacity=1 ] [dash pattern={on 0.84pt off 2.51pt}]  (89.6,53.31) .. controls (109.76,79.62) and (104,78.67) .. (120.33,95) ;
\draw [shift={(88.33,51.67)}, rotate = 52.27] [color={rgb, 255:red, 208; green, 2; blue, 27 }  ,draw opacity=1 ][line width=0.75]    (10.93,-4.9) .. controls (6.95,-2.3) and (3.31,-0.67) .. (0,0) .. controls (3.31,0.67) and (6.95,2.3) .. (10.93,4.9)   ;
\draw [color={rgb, 255:red, 208; green, 2; blue, 27 }  ,draw opacity=1 ] [dash pattern={on 0.84pt off 2.51pt}]  (376.8,32.71) .. controls (384.8,118.85) and (426,118.85) .. (459.6,121.65) ;
\draw [shift={(376.8,32.71)}, rotate = 84.69] [color={rgb, 255:red, 208; green, 2; blue, 27 }  ,draw opacity=1 ][fill={rgb, 255:red, 208; green, 2; blue, 27 }  ,fill opacity=1 ][line width=0.75]      (0, 0) circle [x radius= 3.35, y radius= 3.35]   ;
\draw [color={rgb, 255:red, 208; green, 2; blue, 27 }  ,draw opacity=1 ] [dash pattern={on 0.84pt off 2.51pt}]  (377.7,40.35) .. controls (383.63,79.3) and (390.99,93.04) .. (404.06,104.78) .. controls (417.39,116.76) and (430.13,119.53) .. (459.6,121.65) ;
\draw [shift={(377.33,37.93)}, rotate = 81.63] [color={rgb, 255:red, 208; green, 2; blue, 27 }  ,draw opacity=1 ][line width=0.75]    (10.93,-4.9) .. controls (6.95,-2.3) and (3.31,-0.67) .. (0,0) .. controls (3.31,0.67) and (6.95,2.3) .. (10.93,4.9)   ;
\draw  [dash pattern={on 0.84pt off 2.51pt}] (475.51,272) -- (335.27,12) -- (614.92,11.57) -- cycle ;
\draw [color={rgb, 255:red, 208; green, 2; blue, 27 }  ,draw opacity=1 ] [dash pattern={on 0.84pt off 2.51pt}]  (459.33,177.11) .. controls (438.67,224.45) and (496.67,187.78) .. (475.51,272) ;
\draw [shift={(459.33,177.11)}, rotate = 113.59] [color={rgb, 255:red, 208; green, 2; blue, 27 }  ,draw opacity=1 ][fill={rgb, 255:red, 208; green, 2; blue, 27 }  ,fill opacity=1 ][line width=0.75]      (0, 0) circle [x radius= 3.35, y radius= 3.35]   ;
\draw [color={rgb, 255:red, 74; green, 144; blue, 226 }  ,draw opacity=1 ] [dash pattern={on 0.84pt off 2.51pt}]  (486,138.45) .. controls (490.67,175.11) and (460.67,167.78) .. (459.33,177.11) ;
\draw [shift={(486,138.45)}, rotate = 82.75] [color={rgb, 255:red, 74; green, 144; blue, 226 }  ,draw opacity=1 ][fill={rgb, 255:red, 74; green, 144; blue, 226 }  ,fill opacity=1 ][line width=0.75]      (0, 0) circle [x radius= 3.35, y radius= 3.35]   ;
\draw [color={rgb, 255:red, 74; green, 144; blue, 226 }  ,draw opacity=1 ] [dash pattern={on 0.84pt off 2.51pt}]  (431.33,94.45) .. controls (438.14,105) and (459.01,123.18) .. (469.14,127.29) .. controls (479.27,131.4) and (486.65,133.88) .. (486,138.45) ;
\draw [shift={(431.33,94.45)}, rotate = 57.17] [color={rgb, 255:red, 74; green, 144; blue, 226 }  ,draw opacity=1 ][fill={rgb, 255:red, 74; green, 144; blue, 226 }  ,fill opacity=1 ][line width=0.75]      (0, 0) circle [x radius= 3.35, y radius= 3.35]   ;
\draw [color={rgb, 255:red, 74; green, 144; blue, 226 }  ,draw opacity=1 ] [dash pattern={on 0.84pt off 2.51pt}]  (482.67,86.45) .. controls (484,106.45) and (487.33,129.11) .. (486,138.45) ;
\draw [shift={(482.67,86.45)}, rotate = 86.19] [color={rgb, 255:red, 74; green, 144; blue, 226 }  ,draw opacity=1 ][fill={rgb, 255:red, 74; green, 144; blue, 226 }  ,fill opacity=1 ][line width=0.75]      (0, 0) circle [x radius= 3.35, y radius= 3.35]   ;
\draw [color={rgb, 255:red, 74; green, 144; blue, 226 }  ,draw opacity=1 ] [dash pattern={on 0.84pt off 2.51pt}]  (528.67,87.78) .. controls (530,107.78) and (487.33,109.78) .. (486,119.11) ;
\draw [shift={(528.67,87.78)}, rotate = 86.19] [color={rgb, 255:red, 74; green, 144; blue, 226 }  ,draw opacity=1 ][fill={rgb, 255:red, 74; green, 144; blue, 226 }  ,fill opacity=1 ][line width=0.75]      (0, 0) circle [x radius= 3.35, y radius= 3.35]   ;
\draw [color={rgb, 255:red, 74; green, 144; blue, 226 }  ,draw opacity=1 ] [dash pattern={on 0.84pt off 2.51pt}]  (436.12,100.44) .. controls (446.56,112.18) and (459.46,123.36) .. (469.14,127.29) .. controls (479.27,131.4) and (486.65,133.88) .. (486,138.45) ;
\draw [shift={(434.67,98.78)}, rotate = 49.14] [color={rgb, 255:red, 74; green, 144; blue, 226 }  ,draw opacity=1 ][line width=0.75]    (10.93,-4.9) .. controls (6.95,-2.3) and (3.31,-0.67) .. (0,0) .. controls (3.31,0.67) and (6.95,2.3) .. (10.93,4.9)   ;
\draw [color={rgb, 255:red, 74; green, 144; blue, 226 }  ,draw opacity=1 ] [dash pattern={on 0.84pt off 2.51pt}]  (483.15,93.21) .. controls (484.54,112.36) and (487.29,129.44) .. (486,138.45) ;
\draw [shift={(483,91.11)}, rotate = 86.19] [color={rgb, 255:red, 74; green, 144; blue, 226 }  ,draw opacity=1 ][line width=0.75]    (10.93,-4.9) .. controls (6.95,-2.3) and (3.31,-0.67) .. (0,0) .. controls (3.31,0.67) and (6.95,2.3) .. (10.93,4.9)   ;
\draw [color={rgb, 255:red, 74; green, 144; blue, 226 }  ,draw opacity=1 ] [dash pattern={on 0.84pt off 2.51pt}]  (526.41,94.87) .. controls (512.66,113.08) and (491.17,107.47) .. (486,119.11) ;
\draw [shift={(527.67,93.11)}, rotate = 123.91] [color={rgb, 255:red, 74; green, 144; blue, 226 }  ,draw opacity=1 ][line width=0.75]    (10.93,-4.9) .. controls (6.95,-2.3) and (3.31,-0.67) .. (0,0) .. controls (3.31,0.67) and (6.95,2.3) .. (10.93,4.9)   ;
\draw [color={rgb, 255:red, 208; green, 2; blue, 27 }  ,draw opacity=1 ] [dash pattern={on 0.84pt off 2.51pt}]  (396,47.78) .. controls (421.33,77.78) and (414.67,77.78) .. (431.33,94.45) ;
\draw [shift={(396,47.78)}, rotate = 49.82] [color={rgb, 255:red, 208; green, 2; blue, 27 }  ,draw opacity=1 ][fill={rgb, 255:red, 208; green, 2; blue, 27 }  ,fill opacity=1 ][line width=0.75]      (0, 0) circle [x radius= 3.35, y radius= 3.35]   ;
\draw [color={rgb, 255:red, 208; green, 2; blue, 27 }  ,draw opacity=1 ] [dash pattern={on 0.84pt off 2.51pt}]  (412.67,31.78) .. controls (425.33,62.45) and (431.33,76.45) .. (431.33,94.45) ;
\draw [shift={(412.67,31.78)}, rotate = 67.56] [color={rgb, 255:red, 208; green, 2; blue, 27 }  ,draw opacity=1 ][fill={rgb, 255:red, 208; green, 2; blue, 27 }  ,fill opacity=1 ][line width=0.75]      (0, 0) circle [x radius= 3.35, y radius= 3.35]   ;
\draw [color={rgb, 255:red, 208; green, 2; blue, 27 }  ,draw opacity=1 ] [dash pattern={on 0.84pt off 2.51pt}]  (447.33,27.78) .. controls (444,58.45) and (427.33,50.45) .. (427.33,68.45) ;
\draw [shift={(447.33,27.78)}, rotate = 96.2] [color={rgb, 255:red, 208; green, 2; blue, 27 }  ,draw opacity=1 ][fill={rgb, 255:red, 208; green, 2; blue, 27 }  ,fill opacity=1 ][line width=0.75]      (0, 0) circle [x radius= 3.35, y radius= 3.35]   ;
\draw [color={rgb, 255:red, 208; green, 2; blue, 27 }  ,draw opacity=1 ] [dash pattern={on 0.84pt off 2.51pt}]  (487.33,25.78) .. controls (477.33,71.11) and (534,45.78) .. (482.67,86.45) ;
\draw [shift={(487.33,25.78)}, rotate = 102.44] [color={rgb, 255:red, 208; green, 2; blue, 27 }  ,draw opacity=1 ][fill={rgb, 255:red, 208; green, 2; blue, 27 }  ,fill opacity=1 ][line width=0.75]      (0, 0) circle [x radius= 3.35, y radius= 3.35]   ;
\draw [color={rgb, 255:red, 208; green, 2; blue, 27 }  ,draw opacity=1 ] [dash pattern={on 0.84pt off 2.51pt}]  (415.6,39.37) .. controls (427.53,68.22) and (431.33,76.9) .. (431.33,94.45) ;
\draw [shift={(414.67,37.11)}, rotate = 67.56] [color={rgb, 255:red, 208; green, 2; blue, 27 }  ,draw opacity=1 ][line width=0.75]    (10.93,-4.9) .. controls (6.95,-2.3) and (3.31,-0.67) .. (0,0) .. controls (3.31,0.67) and (6.95,2.3) .. (10.93,4.9)   ;
\draw [color={rgb, 255:red, 208; green, 2; blue, 27 }  ,draw opacity=1 ] [dash pattern={on 0.84pt off 2.51pt}]  (446.23,35.09) .. controls (440.44,59.75) and (429.61,46.67) .. (427.33,68.45) ;
\draw [shift={(446.67,33.11)}, rotate = 101.71] [color={rgb, 255:red, 208; green, 2; blue, 27 }  ,draw opacity=1 ][line width=0.75]    (10.93,-4.9) .. controls (6.95,-2.3) and (3.31,-0.67) .. (0,0) .. controls (3.31,0.67) and (6.95,2.3) .. (10.93,4.9)   ;
\draw [color={rgb, 255:red, 208; green, 2; blue, 27 }  ,draw opacity=1 ] [dash pattern={on 0.84pt off 2.51pt}]  (486.52,32.56) .. controls (482.99,66.13) and (530.69,50.84) .. (482.67,86.45) ;
\draw [shift={(486.8,30.45)}, rotate = 99.07] [color={rgb, 255:red, 208; green, 2; blue, 27 }  ,draw opacity=1 ][line width=0.75]    (10.93,-4.9) .. controls (6.95,-2.3) and (3.31,-0.67) .. (0,0) .. controls (3.31,0.67) and (6.95,2.3) .. (10.93,4.9)   ;
\draw  [color={rgb, 255:red, 74; green, 144; blue, 226 }  ,draw opacity=1 ][dash pattern={on 0.84pt off 2.51pt}] (420.67,76.45) -- (542.67,76.45) -- (542.67,109.78) -- (420.67,109.78) -- cycle ;
\draw  [color={rgb, 255:red, 208; green, 2; blue, 27 }  ,draw opacity=1 ][dash pattern={on 0.84pt off 2.51pt}] (372.67,15.11) -- (518,15.11) -- (518,71.11) -- (372.67,71.11) -- cycle ;
\draw [color={rgb, 255:red, 208; green, 2; blue, 27 }  ,draw opacity=1 ] [dash pattern={on 0.84pt off 2.51pt}]  (400.6,52.76) .. controls (420.76,79.07) and (415,78.11) .. (431.33,94.45) ;
\draw [shift={(399.33,51.11)}, rotate = 52.27] [color={rgb, 255:red, 208; green, 2; blue, 27 }  ,draw opacity=1 ][line width=0.75]    (10.93,-4.9) .. controls (6.95,-2.3) and (3.31,-0.67) .. (0,0) .. controls (3.31,0.67) and (6.95,2.3) .. (10.93,4.9)   ;

\draw (159.33,169.07) node [anchor=north west][inner sep=0.75pt]  [color={rgb, 255:red, 208; green, 2; blue, 27 }  ,opacity=1 ]  {$t'$};
\draw (184.67,129.07) node [anchor=north west][inner sep=0.75pt]  [color={rgb, 255:red, 74; green, 144; blue, 226 }  ,opacity=1 ]  {$t$};
\draw (238.67,81.73) node [anchor=north west][inner sep=0.75pt]  [color={rgb, 255:red, 74; green, 144; blue, 226 }  ,opacity=1 ]  {$F$};
\draw (222,30.4) node [anchor=north west][inner sep=0.75pt]  [color={rgb, 255:red, 208; green, 2; blue, 27 }  ,opacity=1 ]  {$F'$};
\draw (270,23.07) node [anchor=north west][inner sep=0.75pt]    {$T$};
\draw (470.33,168.51) node [anchor=north west][inner sep=0.75pt]  [color={rgb, 255:red, 208; green, 2; blue, 27 }  ,opacity=1 ]  {$t'$};
\draw (495.67,128.51) node [anchor=north west][inner sep=0.75pt]  [color={rgb, 255:red, 74; green, 144; blue, 226 }  ,opacity=1 ]  {$t$};
\draw (549.67,81.18) node [anchor=north west][inner sep=0.75pt]  [color={rgb, 255:red, 74; green, 144; blue, 226 }  ,opacity=1 ]  {$F$};
\draw (533,29.85) node [anchor=north west][inner sep=0.75pt]  [color={rgb, 255:red, 208; green, 2; blue, 27 }  ,opacity=1 ]  {$F'$};
\draw (582,23.51) node [anchor=north west][inner sep=0.75pt]    {$T$};

\end{tikzpicture}
    \caption{On the left, we have an example of a pair $(t,F)$ which is shorter than another $(t',F')$, according to \myref{DEF_ShorterPair}. On the right, however, $(t,F)$ is not considered shorter than $(t',F')$, since an element of $F'$ is not above any element of $F$.}
    \label{FIG_ShorterPair}
\end{figure}
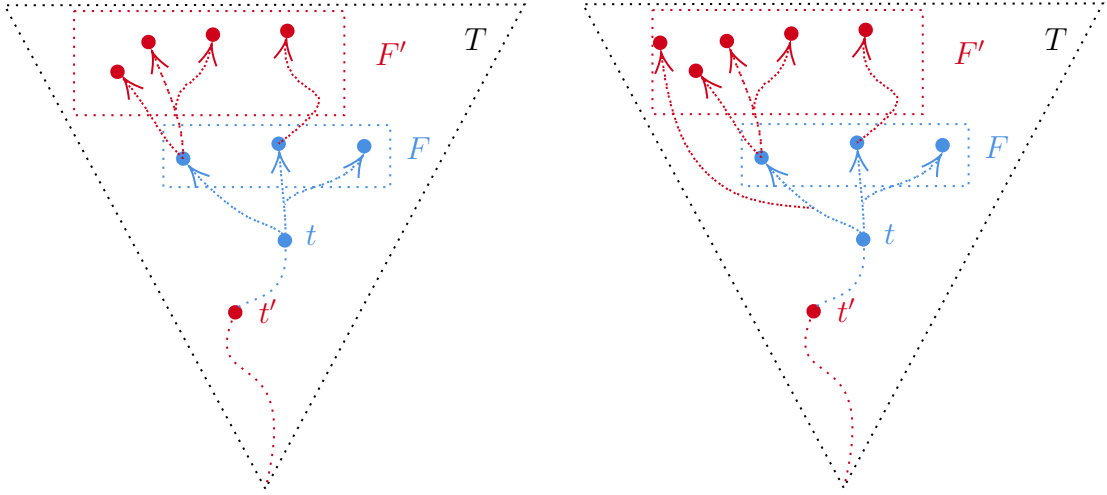

    We then consider $\bbS(T)$ as the partially ordered collection of all spread out families in $T$, in which $\mS'\ge\mS$ (reading ``\emph{$\mS'$ refines $\mS$}'') if and only if for every $(t',F')\in \mS'$ there is some $(t,F)\in \mS$ such that $(t',F')\succeq (t,F)$.
\end{defn}

\begin{ex}
    Suppose $T$ is a tree with $r\in T$ being its root. Since we do not consider the empty family as spread out, note that $\{(r,\emptyset)\}$ is the \emph{least refined} spread out family in $T$ (meaning that $\mS\ge \{(r,\emptyset)\}$ for every $\mS\in \bbS(T)$).
\end{ex}

The following couple of technical results, of utmost importance for the main theorem of this section, tell us that we can always find the most \emph{optimal} partition of a clopen set comprised of standard basic open sets in a ray space.

\begin{lemma}\label{LEMMA_BroadestPair}
    Suppose $T$ is a tree, $U\subseteq \mR(T)$ is clopen and $R\in U$. If $\mF\subseteq \PP(T)$ is the set of pairs $(t,F)$ such that $R\in [t,F]\subseteq U$ and $F$ is minimal in this regard (meaning that $[t,F\setminus \{s\}]$ is not contained in $U$ for any $s\in F$), then $\mF$ contains a longest member $(t_R,F_R)$ (i.e.,  a minimum with respect to $\succeq$).
\end{lemma}
\begin{proof}
    Let us begin, using the well order in $R$, by fixing the least $t_R\in R$ for which there exists some finite set of tops $F'\subset T$ such that $R\in [t_R,F']\subseteq U$.

    Let $\mF(t_R)=\set{(t_R,F)\in \PP(T):(t_R,F)\in \mF}$ and note that $(t_R,F)\succ (t_R,F')$ if and only if $F'$ is above $F$.

    In order to show that we can apply Kuratowski-Zorn's Lemma to $(\mF(t_R),\succeq)$, suppose $\mC$ is a chain in $\mF(t_R)$. Minimality of the finite sets which appear in $\mF$ tells us that, if $(t_R,F), (t_R,F')\in \mC$ are such that $(t_R,F')\preceq(t_R,F)$, then for every $s\in F$ there is some $s'\in F'$ such that $s'\ge s$.

    We claim that there is some $F_0$ such that $(t_R,F_0)\in \mC$ and, for every $(t_R,F')\in \mC$ such that $(t_R,F')\preceq (t_R,F_0)$ and $s\in F_0$, there is a unique $s'\in F'$ with $s'\ge s$. 
    Indeed, suppose not. Then we can find a $\succeq$-decreasing sequence $((t_R,F_n):n\in\omega)$ in $\mC$, together with a sequence of pairs $(s^0_n,s^1_n:n\in\omega)$ such that, for every $n\in\omega$, $s^0_{n+1},s^1_{n+1}$ are distinct elements of $F_{n+1}$ and $s^0_{n+1},s^1_{n+1} \ge s^0_n$.
    
    Once again, minimality of each $F_n$ tells us that, in fact, $s^0_{n+1},s^1_{n+1}> s^0_n$. 
    In this case, consider the ray $R'\in\mR(T)$ in which $(s^0_n:n\in\omega)$ is cofinal. For each $n\in\omega$, choose a ray $R_n\in\mR(T)$ of which $s^0_n$ is a top. The minimality of $F_n$ implies that its members are pairwise incomparable, so $R_n\in[t_R,F_n]\subseteq U$. Moreover, $R_n\to R'$: every basic neighborhood $[t,F]$ of $R'$ contains $R_n$ for all sufficiently large $n$, because $t<s^0_n$ eventually and every member of $F\subseteq\Rtop(R')$ lies strictly above $s^0_n$. Since $U$ is closed, it follows that $R'\in U$. 
    However, since $U$ is open, this would tell us that there must be some $t\in R'$ and a finite set $F\subset T$ of tops of $R'$ such that $[t,F]\subseteq U$, which would contradict the minimality of the $F_n$'s above such $t$ (since, in this case, $[t,F_n\setminus{\set{s^1_n}}]$ would still be contained in $U$).

    Thus, we may fix such an $F_0$ and note that, for each $s\in F_0$, the set 
    \[
        \mC(s) = \set{s'\in \lfloor s \rfloor: \exists F'\subset T ( s'\in F'\wedge  (t_R,F')\in \mC \wedge (t_R,F')\preceq (t_R,F_0))}
    \]
    is a chain in $T$. 
    If $\mC(s)$ contains a maximal element $s'$, then we let $F(s) = \{s'\}$. 
    Otherwise, let us consider the ray $R_s\in \mR(T)$ in which $\mC(s)$ is cofinal. Then, since $U$ is closed and  $R_s$ is the accumulation point of the set of rays 
    \[
        \set{R''\in \mR(T): \text{ there is some $s''\in \mC(s)$ which is a top of $R''$}}\subseteq U,
    \]
    $R_s\in U$. Note $R_s$ cannot be a branch of $T$: otherwise, the fact that $U$ is open would tell us that there would be some $t\in R_s$ with $[t,\emptyset]\subseteq U$, contradicting the minimality of any $F$ such that $(t_R,F)\preceq (t_R,F_0)$ and $F\cap \mC(s)$ is above $t$. There must then be some $t\in R_s$ and a finite set $F(s)\subset T$ of tops of $R_s$ such that $[t,F(s)]\subseteq U$.

    In this case, let $F = \bigcup_{s\in F_0}F(s)$ (see \myref{FIG_LongestPair}). 
    We claim that
    \begin{equation}\label{EQ_LongestPair}
        [t_R,F] = \set{R_s:s\in F_0 \text{  and $\mC(s)$ does not attain a maximum}}\cup\left(\bigcup_{\substack{(t_R,F')\in \mC \\ (t_R,F')\preceq (t_R,F_0)}}[t_R,F']\right).
    \end{equation}
    Indeed, 
    suppose $R'\in [t_R,F]$. 
    Note that $[t_R,F_0]\subseteq [t_R,F']$ for every $(t_R,F')\in \mC$ such that $(t_R,F')\preceq (t_R,F_0)$, so we may assume that $R'\notin [t_R,F_0]$, without loss of generality.
    In this case, let $s\in F_0\cap R'$. 
    If $R'\neq R_s$, then there is a $(t_R,F')\in \mF(t_R)$  with $(t_R,F')\prec (t_R,F_0)$ and an $s'\in F'$ such that $s'$ is above $R'$. 
    Hence, $R'\in [t_R,F']$, as desired.

\begin{figure}
    \centering

\tikzset{every picture/.style={line width=0.75pt}} 

\begin{tikzpicture}[x=0.75pt,y=0.75pt,yscale=-1,xscale=1]

\draw  [dash pattern={on 0.84pt off 2.51pt}] (329.51,275) -- (189.27,15) -- (468.92,14.57) -- cycle ;
\draw [color={rgb, 255:red, 0; green, 0; blue, 0 }  ,draw opacity=1 ] [dash pattern={on 0.84pt off 2.51pt}]  (313.33,180.11) .. controls (292.67,227.45) and (350.67,190.78) .. (329.51,275) ;
\draw [shift={(313.33,180.11)}, rotate = 113.59] [color={rgb, 255:red, 0; green, 0; blue, 0 }  ,draw opacity=1 ][fill={rgb, 255:red, 0; green, 0; blue, 0 }  ,fill opacity=1 ][line width=0.75]      (0, 0) circle [x radius= 3.35, y radius= 3.35]   ;
\draw [color={rgb, 255:red, 0; green, 0; blue, 0 }  ,draw opacity=1 ] [dash pattern={on 0.84pt off 2.51pt}]  (276.5,124) .. controls (283.31,134.55) and (292.03,160.54) .. (299,166) .. controls (305.97,171.46) and (313.99,175.55) .. (313.33,180.11) ;
\draw [shift={(276.5,124)}, rotate = 57.17] [color={rgb, 255:red, 0; green, 0; blue, 0 }  ,draw opacity=1 ][fill={rgb, 255:red, 0; green, 0; blue, 0 }  ,fill opacity=1 ][line width=0.75]      (0, 0) circle [x radius= 3.35, y radius= 3.35]   ;
\draw [color={rgb, 255:red, 0; green, 0; blue, 0 }  ,draw opacity=1 ] [dash pattern={on 0.84pt off 2.51pt}]  (310,128.11) .. controls (311.33,148.11) and (314.67,170.78) .. (313.33,180.11) ;
\draw [shift={(310,128.11)}, rotate = 86.19] [color={rgb, 255:red, 0; green, 0; blue, 0 }  ,draw opacity=1 ][fill={rgb, 255:red, 0; green, 0; blue, 0 }  ,fill opacity=1 ][line width=0.75]      (0, 0) circle [x radius= 3.35, y radius= 3.35]   ;
\draw [color={rgb, 255:red, 0; green, 0; blue, 0 }  ,draw opacity=1 ] [dash pattern={on 0.84pt off 2.51pt}]  (355,114) .. controls (349.5,133) and (314.67,147.28) .. (313.33,156.61) ;
\draw [shift={(355,114)}, rotate = 106.14] [color={rgb, 255:red, 0; green, 0; blue, 0 }  ,draw opacity=1 ][fill={rgb, 255:red, 0; green, 0; blue, 0 }  ,fill opacity=1 ][line width=0.75]      (0, 0) circle [x radius= 3.35, y radius= 3.35]   ;
\draw [color={rgb, 255:red, 0; green, 0; blue, 0 }  ,draw opacity=1 ] [dash pattern={on 0.84pt off 2.51pt}]  (278.94,129.9) .. controls (285.58,143.33) and (291.81,157.36) .. (299,166) .. controls (306.53,175.04) and (313.99,175.55) .. (313.33,180.11) ;
\draw [shift={(278,128)}, rotate = 63.43] [color={rgb, 255:red, 0; green, 0; blue, 0 }  ,draw opacity=1 ][line width=0.75]    (10.93,-4.9) .. controls (6.95,-2.3) and (3.31,-0.67) .. (0,0) .. controls (3.31,0.67) and (6.95,2.3) .. (10.93,4.9)   ;
\draw [color={rgb, 255:red, 0; green, 0; blue, 0 }  ,draw opacity=1 ] [dash pattern={on 0.84pt off 2.51pt}]  (310.48,134.87) .. controls (311.87,154.03) and (314.62,171.11) .. (313.33,180.11) ;
\draw [shift={(310.33,132.78)}, rotate = 86.19] [color={rgb, 255:red, 0; green, 0; blue, 0 }  ,draw opacity=1 ][line width=0.75]    (10.93,-4.9) .. controls (6.95,-2.3) and (3.31,-0.67) .. (0,0) .. controls (3.31,0.67) and (6.95,2.3) .. (10.93,4.9)   ;
\draw [color={rgb, 255:red, 0; green, 0; blue, 0 }  ,draw opacity=1 ] [dash pattern={on 0.84pt off 2.51pt}]  (351.75,119.79) .. controls (338.11,138.72) and (318.51,144.97) .. (313.33,156.61) ;
\draw [shift={(353,118)}, rotate = 123.91] [color={rgb, 255:red, 0; green, 0; blue, 0 }  ,draw opacity=1 ][line width=0.75]    (10.93,-4.9) .. controls (6.95,-2.3) and (3.31,-0.67) .. (0,0) .. controls (3.31,0.67) and (6.95,2.3) .. (10.93,4.9)   ;
\draw [color={rgb, 255:red, 208; green, 2; blue, 27 }  ,draw opacity=1 ] [dash pattern={on 0.84pt off 2.51pt}]  (268.5,72) .. controls (275,102.5) and (271,104.5) .. (276.5,124) ;
\draw [shift={(268.5,72)}, rotate = 77.97] [color={rgb, 255:red, 208; green, 2; blue, 27 }  ,draw opacity=1 ][fill={rgb, 255:red, 208; green, 2; blue, 27 }  ,fill opacity=1 ][line width=0.75]      (0, 0) circle [x radius= 3.35, y radius= 3.35]   ;
\draw [color={rgb, 255:red, 208; green, 2; blue, 27 }  ,draw opacity=1 ] [dash pattern={on 0.84pt off 2.51pt}]  (308,63.33) .. controls (309,92.33) and (310,110.11) .. (310,128.11) ;
\draw [shift={(308,63.33)}, rotate = 88.03] [color={rgb, 255:red, 208; green, 2; blue, 27 }  ,draw opacity=1 ][fill={rgb, 255:red, 208; green, 2; blue, 27 }  ,fill opacity=1 ][line width=0.75]      (0, 0) circle [x radius= 3.35, y radius= 3.35]   ;
\draw [color={rgb, 255:red, 208; green, 2; blue, 27 }  ,draw opacity=1 ] [dash pattern={on 0.84pt off 2.51pt}]  (347.67,39.11) .. controls (347.67,39) and (347.33,35.33) .. (351,42.67) ;
\draw [shift={(347.67,39.11)}, rotate = 270] [color={rgb, 255:red, 208; green, 2; blue, 27 }  ,draw opacity=1 ][fill={rgb, 255:red, 208; green, 2; blue, 27 }  ,fill opacity=1 ][line width=0.75]      (0, 0) circle [x radius= 3.35, y radius= 3.35]   ;
\draw [color={rgb, 255:red, 208; green, 2; blue, 27 }  ,draw opacity=1 ] [dash pattern={on 0.84pt off 2.51pt}]  (308,68) .. controls (309.33,94) and (310,110.11) .. (310,128.11) ;
\draw [color={rgb, 255:red, 208; green, 2; blue, 27 }  ,draw opacity=1 ] [dash pattern={on 0.84pt off 2.51pt}]  (357.59,44.83) .. controls (356.31,79.44) and (355.65,82.31) .. (355,114) ;
\draw [shift={(357.67,42.67)}, rotate = 92.08] [color={rgb, 255:red, 208; green, 2; blue, 27 }  ,draw opacity=1 ][line width=0.75]    (10.93,-4.9) .. controls (6.95,-2.3) and (3.31,-0.67) .. (0,0) .. controls (3.31,0.67) and (6.95,2.3) .. (10.93,4.9)   ;
\draw  [color={rgb, 255:red, 0; green, 0; blue, 0 }  ,draw opacity=1 ][dash pattern={on 0.84pt off 2.51pt}] (267,106.5) -- (362.5,106.5) -- (362.5,145.89) -- (267,145.89) -- cycle ;
\draw  [color={rgb, 255:red, 208; green, 2; blue, 27 }  ,draw opacity=1 ][dash pattern={on 0.84pt off 2.51pt}] (260.33,29.78) -- (377.33,29.78) -- (377.33,85.78) -- (260.33,85.78) -- cycle ;
\draw [color={rgb, 255:red, 208; green, 2; blue, 27 }  ,draw opacity=1 ] [dash pattern={on 0.84pt off 2.51pt}]  (269.5,77.5) .. controls (273.33,99) and (272.67,106.33) .. (276.5,124) ;
\draw [color={rgb, 255:red, 208; green, 2; blue, 27 }  ,draw opacity=1 ] [dash pattern={on 0.84pt off 2.51pt}]  (358,37.45) .. controls (358,37.33) and (357.67,33.67) .. (361.33,41) ;
\draw [shift={(358,37.45)}, rotate = 270] [color={rgb, 255:red, 208; green, 2; blue, 27 }  ,draw opacity=1 ][fill={rgb, 255:red, 208; green, 2; blue, 27 }  ,fill opacity=1 ][line width=0.75]      (0, 0) circle [x radius= 3.35, y radius= 3.35]   ;
\draw [color={rgb, 255:red, 208; green, 2; blue, 27 }  ,draw opacity=1 ] [dash pattern={on 0.84pt off 2.51pt}]  (368.33,39.11) .. controls (368.33,39) and (368,35.33) .. (371.67,42.67) ;
\draw [shift={(368.33,39.11)}, rotate = 270] [color={rgb, 255:red, 208; green, 2; blue, 27 }  ,draw opacity=1 ][fill={rgb, 255:red, 208; green, 2; blue, 27 }  ,fill opacity=1 ][line width=0.75]      (0, 0) circle [x radius= 3.35, y radius= 3.35]   ;

\draw (324.33,171.51) node [anchor=north west][inner sep=0.75pt]  [color={rgb, 255:red, 0; green, 0; blue, 0 }  ,opacity=1 ]  {$t_{R}$};
\draw (369.67,117.68) node [anchor=north west][inner sep=0.75pt]  [color={rgb, 255:red, 0; green, 0; blue, 0 }  ,opacity=1 ]  {$F_{0}$};
\draw (387,32.85) node [anchor=north west][inner sep=0.75pt]  [color={rgb, 255:red, 208; green, 2; blue, 27 }  ,opacity=1 ]  {$F$};
\draw (434,24.51) node [anchor=north west][inner sep=0.75pt]    {$T$};

\end{tikzpicture}
    \caption{Representation of the pair $(t_R,F)$, obtained as the infimum of the chain $\mC$ in the proof of \myref{LEMMA_BroadestPair}.}
    \label{FIG_LongestPair}
\end{figure}

    On the other hand, by construction of $F$, $R_s\in [t_R,F]$ for every $s\in F_0$ such that $\mC(s)$ does not attain a maximum and $(t_R,F)\preceq (t_R,F')$ for every $(t_R,F')\in \mC$ with $(t_R,F')\preceq (t_R,F_0)$. 
    Hence, equality \ref{EQ_LongestPair} has been demonstrated. 
    
    In particular, $(t_R,F)\in \mF(t_R)$ and $(t_R,F)\preceq (t_R, F')$ for every $(t_R,F')\in \mC$, so $(t_R,F)$ is, indeed, an infimum of $\mC$ in $\mF(t_R)$. 

    Now, using Kuratowski-Zorn's Lemma, let $(t_R,F_R)\in \mF(t_R)$ be a $\preceq$-minimal element in $\mF(t_R)$. We claim that $(t_R,F_R)$ must, in fact, be a $\preceq$-minimum element in $\mF(t_R)$, which concludes the proof. 
    Indeed, fix $(t_R,F')\in\mF(t_R)$ and let
    \[
        H=\set{\max\{s,s'\}:s\in F_R,\ s'\in F',\text{ and $s,s'$ are comparable}}.
    \]
    Since the nodes in every ray form a chain, we have
    \[
        [t_R,H]=[t_R,F_R]\cup[t_R,F']\subseteq U.
    \]
    Choose an inclusion-minimal $H_0\subseteq H$ such that $[t_R,H_0]\subseteq U$. Then $(t_R,H_0)\in\mF(t_R)$. Every member of $H_0$ is above a member of $F_R$ and a member of $F'$, and hence
    \[
        (t_R,H_0)\preceq(t_R,F_R)
        \qquad\text{and}\qquad
        (t_R,H_0)\preceq(t_R,F').
    \]
    The $\preceq$-minimality of $(t_R,F_R)$ now gives $(t_R,H_0)=(t_R,F_R)$, so $(t_R,F_R)\preceq(t_R,F')$. Thus $(t_R,F_R)$ is a $\preceq$-minimum element of $\mF(t_R)$, as desired.
\end{proof}

\begin{cor}\label{COR_Clopen_SpreadoutPartition}
    Suppose $T$ is a tree and $U\subseteq \mR(T)$ is clopen. If, for each $R\in U$, $(t_R,F_R)\in \PP(T)$ is as in \myref{LEMMA_BroadestPair}, then $\set{[t_R,F_R]:R\in U}$ is a partition of $U$.

    In particular, $\mS_U = \set{(t_R,F_R):R\in U}$ is spread out in $T$.
\end{cor}
\begin{proof}
    Suppose $R,R'\in U$ are such that $[t_{R},F_{R}]\cap [t_{R'},F_{R'}]\neq\emptyset$ and fix $\tilde{R}\in [t_{R},F_{R}]\cap [t_{R'},F_{R'}]$. 
    Then minimality of $(t_{\tilde{R}},F_{\tilde{R}})$ tells us that $(t_{\tilde{R}},F_{\tilde{R}})$ is longer than or equal to both $(t_{R},F_{R})$ and $ (t_{R'},F_{R'})$. However, this entails $[t_{\tilde{R}},F_{\tilde{R}}]$ contains both $[t_{R},F_{R}]$ and $ [t_{R'},F_{R'}]$ (thus, in particular, it also contains $R$ and $R'$). Hence, minimality of $(t_{R},F_{R})$ and $ (t_{R'},F_{R'})$ entails
    \[
        (t_{R},F_{R}) = (t_{\tilde{R}},F_{\tilde{R}}) = (t_{R'},F_{R'}),
    \]
    as desired.

    \sloppy Finally, since $\set{[t_R,F_R]:R\in U}$ is pairwise disjoint and its union is equal to the clopen $U$, it follows from \myref{PROP_PairwiseDisjointPairCollection} and \myref{PROP_spreadout_clopen_eqv} that $\mS_U = \set{(t_R,F_R):R\in U}$ is spread out in $T$.
\end{proof}

Finally, we have every necessary tool at hand to show our combinatorial characterization:

\begin{thm}\label{THM_CombinatorialCharacterization_SpreaoutMap}
    Suppose $T$ and $T'$ are pruned trees. Then $\mR(T)\simeq \mR(T')$ if and only if there is a map $\psi\colon T\to \bbS(T')$ satisfying all of the following properties:
    \begin{enumerate}[label=(\roman*)]
        \item\label{item_CombChar_RootCondition} If $r\in T$ is the root of $T$, then $\psi(r) = \{(r',\emptyset)\}$, where $r'\in T'$ is the root of $T'$.
        \item\label{item_CombChar_OrderMorphismCondition} If $t\le s$ in $T$, then $\psi(t)\le \psi(s)$ in $\bbS(T')$.
        \item\label{item_CombChar_IncompatibleCondition} If $t,s\in T$ are incompatible, then $\psi(t)$ and $\psi(s)$ are ray-disjoint.
        
        \item\label{item_CombChar_CoveringCondition} If $t\in T$, then
        \[\bigvee \psi(t) = \bigcup_{s\in \succs(t)}\bigvee\psi(s)\]
        
        \item\label{item_CombChar_LimitCondition} If $R$ is a ray in $T$, then there is a unique ray $R'$ in $T'$ such that 
        \begin{equation}\label{EQ_CombChar_SpreadoutMapLimitPartition1}
            \bigcap_{t\in R}\left(\bigvee\psi(t)\right)\setminus \bigcup_{s\in \Rtop(R)}\left(\bigvee\psi(s)\right)  = \{R'\}
        \end{equation}
        and, for each $t'\in R'$ and each finite $F'\subseteq \Rtop(R')$, there are $t\in R$ and finite $F\subseteq \Rtop(R)$ such that
        \begin{equation}\label{EQ_CombChar_SpreadoutMapLimitPartition2}
            \bigvee\psi(t)\setminus \bigcup_{s\in F}\left(\bigvee\psi(s)\right) \subseteq [t',F'].
        \end{equation}
    \end{enumerate}
\end{thm}
\begin{proof}
    Suppose $f\colon \mR(T)\to \mR(T')$ is a homeomorphism.
    Let $\psi\colon T\to\bbS(T')$ be such that $\psi(t) = \set{(t_R,F_R):R\in f[[t,\emptyset]]}$ for every $t\in T$, where each pair $(t_R,F_R)$ is as in \myref{LEMMA_BroadestPair} for $R\in f[[t,\emptyset]]$ (so that, by \myref{COR_Clopen_SpreadoutPartition}, $\psi(t)$ is a spread out family in $T'$ for every $t\in T$).
    In particular, note that, for every $t\in T$,
    \begin{equation}\label{EQ_CombChar_SpreadoutMapLimitPartition0}
        \bigvee\psi(t) = f[[t,\emptyset]].
    \end{equation}
    
    We claim that $\psi$ satisfies conditions (i)--(v) in \myref{THM_CombinatorialCharacterization_SpreaoutMap}.

    Indeed, note that $\psi(r) = \{(r',\emptyset)\}$, where $r'\in T'$ is the root of $T'$ (since $f[[r,\emptyset]] = \mR(T')$ and $(r',\emptyset)$ is the longest pair in $\PP(T')$), so $\psi$ satisfies  condition \ref{item_CombChar_RootCondition}.

    On the other hand, minimality of each $(t',F')\in \psi(t)$ with respect to the clopen $f[[t,\emptyset]]$ tells us that $\psi(s)\ge \psi(t)$ whenever $s\ge t$ (since $f[[s,\emptyset]]\subseteq f[[t,\emptyset]]$, in this case). Thus, condition \ref{item_CombChar_OrderMorphismCondition} holds for $\psi$.

    Now suppose $t,s\in T$ are incompatible. The fact that $f[[t,\emptyset]]\cap f[[s,\emptyset]]=\emptyset$ tells us that $\psi(t)$ and $\psi(s)$ are ray-disjoint (in view of \myref{EQ_CombChar_SpreadoutMapLimitPartition0} and  \myref{PROP_RayDisjointCollection}). Therefore, $\psi$ also satisfies condition \ref{item_CombChar_IncompatibleCondition}.

    Moreover, note that
    \[
        \bigvee \psi(t) = f[[t,\emptyset]] = \bigcup_{s\in \succs(t)} f[[s,\emptyset]] = \bigcup_{s\in \succs(t)}\bigvee\psi(s)
    \]
    for every $t\in T$, so it is shown that condition \ref{item_CombChar_CoveringCondition} also holds for $\psi$.
         
    At last, for each ray $R$ in $T$, note that
    \[
        \left(\bigcap_{t\in R}[t,\emptyset]\right) \setminus \bigcup_{s\in \Rtop(R)}[s,\emptyset] = \{R\},
    \]
    so
    \[
        \left(\bigcap_{t\in R}\bigvee\psi(t)\right) \setminus \bigcup_{s\in \Rtop(R)}\left(\bigvee\psi(s)\right) = \left(\bigcap_{t\in R}f[[t,\emptyset]]\right) \setminus \bigcup_{s\in \Rtop(R)}\left(f[[s,\emptyset]]\right) = \{f(R)\}.
    \]
    
    Furthermore, since
    \[
        \set{[t,\emptyset]\setminus \left(\bigcup_{s\in F}[s,\emptyset] \right) = [t,F]:t\in R \text{ and  } F\subseteq\Rtop(R) \text{ is finite } }
    \]
    is a local basis for $R\in \mR(T)$, it follows from the fact that $f$ is a homeomorphism that
    \begin{gather*}
        \set{\bigvee\psi(t)\setminus \left(\bigcup_{s\in F}\left(\bigvee\psi(s)\right) \right) : t\in R \text{ and finite } F\subseteq\Rtop(R)} \\
        =\set{f[[t,\emptyset]]\setminus \left(\bigcup_{s\in F}f[[s,\emptyset]] \right):t\in R \text{ and finite } F\subseteq\Rtop(R)}
    \end{gather*}
     is a local basis for $f(R)\in\mR(T')$. 
     Thus, for all $t'\in f(R)$ and finite $F'\subseteq \Rtop(f(R))$, we can find a $t\in R$ and a finite $F\subseteq \Rtop(R)$ such that
     \[
        \bigvee\psi(t)\setminus \bigcup_{s\in F}\left(\bigvee\psi(s)\right) \subseteq [t',F'].
     \]
    Hence, the defined $\psi$ satisfies conditions (i)--(v).

    Now suppose that $\psi\colon T\to \bbS(T')$ satisfies conditions (i)--(v). 
    We can then let $f\colon \mR(T)\to \mR(T')$ be such that, for each $R\in \mR(T)$, $f(R)=R'\in\mR(T')$ is the unique ray mentioned in condition \ref{item_CombChar_LimitCondition}, \myref{EQ_CombChar_SpreadoutMapLimitPartition1}.

    We claim that $f$ is a homeomorphism. Indeed:
    
    \begin{description}
        \item[$f$ is injective] Suppose $R_0,R_1\in \mR(T)$ are distinct. 
        If $R_0\subsetneq R_1$ (or vice-versa), then condition \ref{item_CombChar_LimitCondition} guarantees that $f(R_0)\neq f(R_1)$. 
        Otherwise,
        we can find the least $\beta$ such that $R_0(\beta) \neq R_1(\beta)$. 
        This means that $R_0(\beta)$ and $R_1(\beta)$ are incompatible in $T$, so it follows from 
        condition \ref{item_CombChar_IncompatibleCondition} that $\bigvee\psi(R_0(\beta))$ and $\bigvee\psi(R_1(\beta))$ are disjoint (see \myref{PROP_RayDisjointCollection}). 
        Since $f(R_0)\in\bigvee\psi(R_0(\beta))$ and $f(R_1) \in \bigvee\psi(R_1(\beta))$ (by definition of $f$, according to condition \ref{item_CombChar_LimitCondition}, \myref{EQ_CombChar_SpreadoutMapLimitPartition1}), we thus obtain that $f(R_0)\neq f(R_1)$, as desired.

        \item[$f$ is surjective] Given $R'\in \mR(T')$, let us show that 
        \begin{equation}\label{EQ_CombChar_SpreadoutMapSurjectivity}
            R' = f(R), \text{ where } R = \set{t\in T: R'\in \bigvee \psi(t)}.
        \end{equation}
        In order to see that $R$ is a ray in $T$, let us first note that $R'\in \bigvee\psi(r) = \mR(T')$, so $r\in R\neq\emptyset$. 
        On the other hand, if $t,s\in T$ are such that $R'\in \bigvee\psi(t)$ and $R'\in \bigvee\psi(s)$, then $\psi(t)$ and $\psi(s)$ are not ray-disjoint (see \myref{PROP_RayDisjointCollection}), so it follows from condition \ref{item_CombChar_IncompatibleCondition} that $t$ and $s$ must be compatible (thus, $R$ is a chain in $T$).
        Now suppose $R'\in \bigvee \psi(t)$ for some $t\in T$. 
        Then, by condition \ref{item_CombChar_CoveringCondition}, $R'\in \bigvee\psi(s)$ for some $s\in\succs(t)$. 
        Hence, $R$ does not have a maximal element and is, therefore, a ray in $T$.

        At last, since $R'\in\bigvee \psi(t)$ for every $t\in R$ and $R'\notin \bigvee \psi(s)$ for every $s\in \Rtop(R)$ (by definition of $R$), we thus conclude (using condition \ref{item_CombChar_LimitCondition}, \myref{EQ_CombChar_SpreadoutMapLimitPartition1}) that $f(R) = R'$.

        \item[$f$ is open] Suppose $R\in \mR(T)$, together with $t\in R$ and a finite $F\subseteq \Rtop(R)$ are given. In view of \myref{EQ_CombChar_SpreadoutMapSurjectivity}, note that $R'\in f[[t,F]]$ if and only if $R'\in \bigvee\psi(t)$ and $R'\notin \bigvee\psi(s)$ for any $s\in F$. Thus,
        \begin{equation}\label{EQ_CombChar_SpreadoutMapBasicImage}
            f[[t,F]] = \bigvee\psi(t)\setminus \bigcup_{s\in F}\left(\bigvee\psi(s)\right).
        \end{equation}

        Finally, since $\bigvee\psi(p)$ is clopen for every $p\in T$ (see \myref{PROP_spreadout_clopen_eqv}), we conclude from \myref{EQ_CombChar_SpreadoutMapBasicImage} that $f[[t,F]]$ is open, as desired.

        \item[$f$ is continuous] Suppose $U\subseteq \mR(T')$ is a nonempty open set. Since $f$ has already been shown to be a bijection, $f^{-1}(U)\neq\emptyset$, so let us fix $R\in f^{-1}(U)$. Since $f(R)\in U$, let us fix $t'\in f(R)$ and a finite $F'\subseteq\Rtop(f(R))$ such that $[t',F']\subseteq U$. Then condition \ref{item_CombChar_LimitCondition}, \myref{EQ_CombChar_SpreadoutMapLimitPartition2}, tells us that there exists a $t\in R$ and a finite $F\subseteq \Rtop(R)$ such that 
        \[
            \bigvee\psi(t)\setminus \bigcup_{s\in F}\left(\bigvee\psi(s)\right) \subseteq [t',F'].
        \]
        Then, by \myref{EQ_CombChar_SpreadoutMapBasicImage},  $f[[t,F]]\subseteq [t',F']\subseteq U$. We conclude that $R\in [t,F]\subseteq f^{-1}(U)$, so $f$ is continuous.
    \end{description} 
\end{proof}

\begin{rmk}
    We should point out that, in the presence of conditions \ref{item_CombChar_OrderMorphismCondition} and \ref{item_CombChar_IncompatibleCondition}, condition \ref{item_CombChar_CoveringCondition} in \myref{THM_CombinatorialCharacterization_SpreaoutMap} is equivalent to 
    \begin{itemize}
        \item[(iv')] If $t\in T$, then $\bigcup_{s\in \succs(t)}\psi(s)$ is a maximal (with respect to inclusion) collection such that
        \[
        \psi(t)\le\left(\bigcup_{s\in \succs(t)}\psi(s)\right)\in \bbS(T').
        \]
    \end{itemize}

    \sloppy Indeed, assume that conditions \ref{item_CombChar_OrderMorphismCondition}, \ref{item_CombChar_IncompatibleCondition} and \ref{item_CombChar_CoveringCondition} hold. Then condition \ref{item_CombChar_OrderMorphismCondition} tells us that $\psi(t)\le\bigcup_{s\in \succs(t)}\psi(s)$, while \ref{item_CombChar_IncompatibleCondition} tells us that $\set{[t,F]:(t,F)\in \bigcup_{s\in \succs(t)}\psi(s)}$ partitions $\bigcup_{s\in \succs(t)}\left(\bigvee\psi(s)\right)$ (see \myref{PROP_RayDisjointCollection}) and thus, since $\bigvee\psi(t) = \bigcup_{s\in \succs(t)}\left(\bigvee\psi(s)\right)$ (by condition \ref{item_CombChar_CoveringCondition}) is clopen (by \myref{PROP_spreadout_clopen_eqv}), we conclude (in view of \myref{COR_RayDisjointSpreadoutFamily}) that $\bigcup_{s\in \succs(t)}\psi(s)\in \bbS(T')$.

    If, in this case, $(t',F')\in \PP(T')$ is such that 
    \[
        \psi(t)\le\left(\{(t',F')\}\cup\bigcup_{s\in \succs(t)}\psi(s)\right)\in \bbS(T'),
    \]
    then, in particular, $[t',F']\subseteq \bigvee\psi(t)$.
    However, $\left(\{(t',F')\}\cup\bigcup_{s\in \succs(t)}\psi(s)\right)\in \bbS(T')$ tells us that $\left(\{(t',F')\}\cup\bigcup_{s\in \succs(t)}\psi(s)\right)$ is pairwise disjoint. Hence, since $\bigvee\psi(t) = \bigcup_{s\in \succs(t)}\left(\bigvee\psi(s)\right)$ (again, by condition \ref{item_CombChar_CoveringCondition}), it follows that $(t',F')\in \bigcup_{s\in \succs(t)}\psi(s)$, so condition (iv') holds.
    
    On the other hand, suppose that condition (iv') holds. 
    In view of \myref{PROP_spreadout_clopen_eqv}, the fact that $\bigcup_{s\in \succs(t)}\psi(s)\in \bbS(T')$ tells us that
    \[
        \bigvee\left( \bigcup_{s\in \succs(t)}\psi(s)\right)= \bigcup_{s\in \succs(t)}\left(\bigvee\psi(s)\right)
    \]
    is clopen, while its maximality tells us that $\overline{\bigcup_{s\in \succs(t)}\left(\bigvee\psi(s)\right)} = \bigvee \psi(t)$, so condition \ref{item_CombChar_CoveringCondition} holds.
\end{rmk}

Note that the partially ordered set $\bbS(T)$ always contains an order isomorphic copy of $T$, namely: $\set{\{(t,\emptyset)\}:t\in T}\subseteq \bbS(T)$.

Thus, it would be tempting to conclude that \myref{THM_CombinatorialCharacterization_SpreaoutMap} establishes that $\bbS(T)$ contains an isomorphic copy of \emph{every} pruned tree $T'$ whose ray space is homeomorphic to the ray space of $T$. 
However, we note that the map $\psi$ in \myref{THM_CombinatorialCharacterization_SpreaoutMap} does not need to be an order embedding -- in fact, the only property that it lacks to this end is injectivity (since it can still be the case that $t<s$ in $T$, but $\psi(t)=\psi(s)$).

With this in mind, it is only natural to ask:

\begin{question}
    Suppose $T$ and $T'$ are pruned  trees such that $\mR(T)\simeq \mR(T')$. Is there a map $\psi\colon T\to \bbS(T')$ satisfying conditions (i)--(v) of \myref{THM_CombinatorialCharacterization_SpreaoutMap} which is, additionally, injective? If not, is there some isomorphic copy of $T$ contained in $\bbS(T')$?
\end{question}

\section{The ray game}\label{SEC_TopologicalCharacterization}

Even though \myref{THM_CombinatorialCharacterization_SpreaoutMap} is placed in the self-contained \myref{SEC_CombinatorialCharacterization}, that combinatorial characterization was originally inspired by the following transfinite topological game, together with the upcoming topological characterization in \myref{COR_TopTreeCharacterization} (we further elaborate on how this ``inspiration'' is realized in \myref{RMK_WinStratSpreadoutMap}). Nevertheless, 
the study of this game, on its own, will be shown to be particularly fruitful in this and the remaining sections of this paper.

In what follows, we say that $\mP$ is a \emph{partition} of a set $X$ if $\mP$ is a pairwise disjoint collection of non-empty sets such that $\bigcup\mP = X$.

\begin{defn}\label{DEF_TreeGame}
    Given a non-empty topological space $X$, the so-called \emph{ray game over $X$}, denoted by $\TreeGame(X)$, is the game played between $\ali$ and $\bob$ which is recursively defined as follows: in the first inning, $\ali$ chooses an open partition $\mP_0$ for $X$, while $\bob$ responds by picking $U_0\in \mP_0$. In the $\alpha$-th inning:
    \begin{description}
        \item[if $\alpha=\beta+1$] then $\bob$ chose a $U_\beta\in \mP_\beta$ in the previous inning, so $\ali$ must now choose an open partition $\mP_{\beta+1}$ for $U_{\beta}$. In this case, $\bob$ must then respond with a $U_{\beta+1}\in\mP_{\beta+1}$.
        \item[if $\alpha$ is a limit ordinal] then we consider two cases:
        \begin{itemize}
            \item If there exists a pairwise disjoint collection $\mP_\alpha$ of open sets such that, for some $x_\alpha\in X$,
            \begin{equation} \label{EQ_RayGameContinues1}
                \bigcap_{\beta<\alpha} U_\beta   = \left(\bigcup \mP_\alpha\right) \sqcup \{x_\alpha\} 
            \end{equation}
            and
            \begin{equation} \label{EQ_RayGameContinues2}
                \set{U_\beta\setminus \bigcup\mF: \beta<\alpha \text{ and finite } \mF \subset \mP_\alpha} \text{ is a local basis for $x_\alpha$,}
            \end{equation}
            then the game proceeds with $\ali$ choosing one such pair $(\mP_\alpha,x_\alpha)$.
            In this case, if $\mP_\alpha\neq\emptyset$, the game continues with $\bob$ choosing a $U_\alpha\in \mP_\alpha$. Otherwise (i.e., if $\mP_\alpha=\emptyset$), this run of the game ends and $\ali$ is declared the winner.

            \item Otherwise (i.e., if no pair $(\mP_\alpha, x_\alpha)$ satisfying \myref{EQ_RayGameContinues1} and \myref{EQ_RayGameContinues2} exists), this run of the game ends and $\bob$ is declared the winner.
        \end{itemize}
    \end{description}
\end{defn}

Suppose $\seq{\mP_0, U_0, \dotsc, (\mP_\gamma,x_\gamma), U_\gamma, \dotsc }$ is a run of the game $\TreeGame(X)$ that finished at the $\alpha$-th inning. Note that $\alpha$ must be a limit ordinal. Furthermore, at every limit ordinal $\gamma<\alpha$, $\ali$ chooses an $x_\gamma\in X$ which is not in $U_\beta$ for any $\gamma<\beta<\alpha$.
Since there are $|X|^+$ limit ordinals in $|X|^+$, this implies that $\alpha<|X|^+$.
In other words:

\begin{prop}\label{PROP_RayGameRunLength}
    For a space $X$, every run of the game $\TreeGame(X)$ has length $<|X|^+$.
\end{prop}

In view of the ray game winning criteria for $\ali$, it is clear that:

\begin{lemma}\label{LEMMA_WRunEndOfRun}
    A run of the game $\TreeGame(X)$ is won by $\ali$ if and only if it ends at a limit ordinal $\alpha$ and $\ali$ chose a valid pair $(\emptyset,x_{\alpha})$.
\end{lemma}

\begin{defn} \label{DEF_StrategyAlice}
    A \emph{strategy $\sigma$ for} $\ali$ \emph{in a game $G$} is a recursively defined map such that $t\in\dom(\sigma)$ if and only if $t$ is a sequence $\seq{B_\beta:\beta<\alpha}$ such that, for every $\beta<\alpha$,
\[
        \seq{\sigma\seq{\,},B_0, \dotsc, \sigma\seq{B_\eta:\eta<\beta}, B_\beta}
\]
is a sequence of moves $\ali$ and $\bob$ can make in $G$ (in particular, note that $\dom(\sigma)$ contains the empty sequence, which we denote by $\seq{\,}$). 
In this case, note that
\[
    \sdom{\sigma}=\{\seq{\,}\}\cup\set{ \seq{B_\beta:\beta<\xi}  \in \dom(\sigma) : \exists \zeta ( \zeta + 1 = \xi) }.
\]
\end{defn}

Essentially, a strategy $\sigma$ for $\ali$ in a game $G$ is a map which determines a unique move for $\ali$ in $G$ in any situation at which she may find herself, if she follows $\sigma$'s instructions.

\begin{defn} \label{DEF_CompatibleStrategy} We say that a run
\[
    \seq{A_0, B_0, \dotsc, A_\beta, B_\beta}
\]
of length $\alpha$ is \emph{compatible} with $\sigma$ if, for every $\gamma<\alpha$,
\[
    A_\gamma = \sigma\seq{B_\beta:\beta<\gamma}.
\]
Furthermore, $\sigma$ is a \emph{winning} strategy for $\ali$ if $\ali$ wins at every finished run which is compatible with $\sigma$.
\end{defn}

\begin{cor}\label{LEMMA_WStratEndOfRun}
    A strategy $\sigma$ for $\ali$ in $\TreeGame(X)$ is a winning strategy if and only if every sequence of $\bob$'s moves $\seq{U_\beta:\beta<\alpha}$ compatible with $\sigma$ is such that $\seq{U_\beta:\beta<\alpha}\in \dom(\sigma)$. In other words, there is no sequence of $\bob$'s responses to $\sigma$ for which $\sigma$ has no viable response.
\end{cor}

\sloppy The following technical lemmas will be useful in determining when a run of $\TreeGame(-)$ is won by $\ali$.

\begin{lemma}\label{LEMMA_WStratClopens}
    \sloppy Suppose $R = \seq{\mP_0, U_0, \dotsc, (\mP_\gamma,x_\gamma), U_\gamma, \dotsc, (\emptyset, x_\alpha) }$ is a run of the game $\TreeGame(X)$ that is finished at the limit ordinal $\alpha$ and is won by $\ali$. Then, for every $\beta<\alpha$ and $U\in \mP_\beta$, $U$ is clopen in $X$.
\end{lemma}
\begin{proof}
    Clearly, for every $\gamma<\alpha$, $U_\gamma\in \mP_\gamma$ is open.
    By induction, assume that every $U\in \mP_\beta$ is closed for every $\beta<\gamma < \alpha$. 
    
    If $\gamma = \beta+1$, then $\mP_\gamma$ is an open partition of the clopen set $U_\beta\in \mP_\beta$, so every $U\in\mP_\gamma$ is closed.
    
    If $\gamma$ is a limit ordinal, fix $W\in \mP_\gamma$. Since 
    \[\set{U_\beta\setminus \bigcup\mF: \beta<\gamma \text{ and finite } \mF \subset \mP_\gamma}\]
    is a local basis for $x_\gamma$, setting $\mF=\{W\}$ it follows that $U_{\beta} \setminus W$ is open in $X$ and $X \setminus (U_{\beta} \setminus W)= (X \setminus U_{\beta}) \cup (U_{\beta} \cap W)$ is closed in $X$. From our induction hypothesis, $U_{\beta}$ is closed in $X$, hence $ U_{\beta} \cap W = U_{\beta} \cap ((X \setminus U_\beta) \cup (U_{\beta} \cap W) )$ is closed in $X$. 
    In particular, this implies that $\bigcap_{\beta<\gamma}(U_{\beta} \cap W)$ is closed in $X$. 
    Since $W\subseteq \bigcap_{\beta<\gamma}U_{\beta}$, we conclude that $W$ is closed in $X$. 
\end{proof}

\begin{lemma}\label{LEMMA_WRunOrdinal}
    Suppose $\seq{\mP_0,U_0,\dotsc,(\mP_\gamma,x_\gamma),U_\gamma,\dotsc,(\emptyset,x_\alpha)}$ is a run of the game $\TreeGame(X)$ finished at the limit ordinal $\alpha$ and won by $\ali$. Then
    \[
        \set{x_\gamma:\gamma\le\alpha\text{ and }\gamma\text{ is a limit ordinal}},
    \]
    as a subspace of $X$, is homeomorphic to
    \[
        \set{\gamma:\gamma\le\alpha\text{ and }\gamma\text{ is a limit ordinal}}
    \]
    as a subspace of the ordered space $\alpha+1$.
\end{lemma}
\begin{proof}
    Let
    \[
        Y=\set{x_\gamma:\gamma\le\alpha\text{ and }\gamma\text{ is a limit ordinal}
        }
    \]
    and let
    \[
        f\colon\set{\gamma:\gamma\le\alpha\text{ and }\gamma\text{ is a limit ordinal}}
        \longrightarrow Y
    \]
    be defined by $f(\gamma)=x_\gamma$. Then $f$ is a bijection.

    For notational convenience, set $U_\alpha=\emptyset$. For all limit ordinals $\eta<\xi\le\alpha$, the rules of $\TreeGame(X)$ imply that $x_\eta\notin U_\xi$ (using this convention when $\xi=\alpha$), while $x_\xi\in U_\eta$. Thus, for all limit ordinals $\beta<\gamma\le\alpha$,
    \begin{equation}\label{EQ_WRunOrdinal}
        f[(\beta,\gamma]]=(U_\beta\setminus U_\gamma)\cap Y.
    \end{equation}
    By \myref{LEMMA_WStratClopens}, the right-hand side is open in $Y$ when $\gamma<\alpha$; when $\gamma=\alpha$, it is $U_\beta\cap Y$, which is also open in $Y$. Hence, $f$ is open.

    Now suppose $W\subseteq Y$ is a non-empty open set, with $W=U\cap Y$ for an open set $U\subseteq X$. We show that every $\gamma\in f^{-1}(W)$ has an open interval neighborhood contained in $f^{-1}(W)$.

    If $\gamma=\alpha$, then $\mP_\alpha=\emptyset$, and \myref{EQ_RayGameContinues2} yields some $\beta<\alpha$ such that $U_\beta\subseteq U$. Therefore,
    \[
        \alpha\in(\beta,\alpha]=f^{-1}[U_\beta\cap Y]\subseteq f^{-1}(W).
    \]

    If $\gamma<\alpha$, then \myref{EQ_RayGameContinues2} yields a finite $\mF\subseteq\mP_\gamma$ and a $\beta<\gamma$ such that
    \[
        x_\gamma\in U_\beta\setminus\bigcup\mF\subseteq U.
    \]
    Since $x_\gamma\notin U_\gamma$, we may enlarge $\mF$ so that $U_\gamma\in\mF$. Moreover, every $V\in\mP_\gamma\setminus\{U_\gamma\}$ is disjoint from $Y$: points $x_\delta$ with $\delta\le\gamma$ lie outside $\bigcup\mP_\gamma$, while points $x_\delta$ with $\delta>\gamma$ lie in $U_\gamma$. Consequently, by \myref{EQ_WRunOrdinal},
    \[
        x_\gamma\in f[(\beta, \gamma]] = (U_\beta\setminus U_\gamma)\cap Y \subseteq (U_\beta\setminus \bigcup\mF)\cap Y\subseteq U\cap Y = W,
    \]
    so $f^{-1}(W)$ is open and $f$ is continuous.
    
\end{proof}

Let us now explore the existence of winning strategies in the ray game for a few examples of topological spaces:

\begin{prop}\label{PROP_RayGameConnectedSpace}
    If $X$ is a connected space with $|X|\ge 2$, then $\TreeGame(X)$ has a single run, and this run is won by $\bob$.
\end{prop}
\begin{proof}
    Since $X$ is a connected space, its only non-empty clopen set is $X$ itself; thus, $\ali$ can only choose the trivial partition $\{X\}$. 
    Hence, $\TreeGame(X)$ has a single run, which ends at the $\omega$-th inning and is won by $\bob$.
\end{proof}

Analogously to the definition for $\ali$, a strategy $\sigma$ for $\bob$ in a game $G$ is a recursively defined map such that $t\in\dom(\sigma)$ if and only if $t$ is a sequence $\seq{A_\beta:\beta\le\alpha}$ such that, for every $\beta\le\alpha$,
\[
        \seq{A_0, \sigma(A_0), \dotsc, A_\beta, \sigma\seq{A_\eta:\eta\le\beta}}
\]
is a sequence of moves $\ali$ and $\bob$ can make in $G$.

We then say that a run
\[
    \seq{A_0, B_0, \dotsc, A_\beta, B_\beta}
\]
of length $\alpha$ is \emph{compatible} with $\sigma$ if, for every $\gamma<\alpha$,
\[
    B_\gamma = \sigma\seq{A_\beta:\beta\le\gamma}.
\]
Furthermore, $\sigma$ is a \emph{winning} strategy for $\bob$ if $\bob$ wins at every finished run which is compatible with $\sigma$.

\begin{ex}
    In view of \myref{PROP_RayGameConnectedSpace}, $\bob$ has a trivial winning strategy in $\TreeGame(\RR)$.
\end{ex}

With the help of the following simple \myref{LEMMA_Omega1Partition}, we prove in \myref{EX_omega1_BobWins} that there is a winning strategy in $\TreeGame(\omega_1)$ for $\bob$:

\begin{lemma}\label{LEMMA_Omega1Partition}
    Suppose $\mP$ is an open partition of $\omega_1$ with the usual order topology. Then there must exist some $V\in\mP$ such that $\omega_1\setminus V$ is countable.
\end{lemma}
\begin{proof}
    Striving for a contradiction, suppose $\omega_1\setminus V$ is uncountable for every $V\in\mP$. Since every member of the open partition $\mP$ is clopen, every unbounded member is a club. There is therefore at most one unbounded member of $\mP$, because two disjoint clubs in $\omega_1$ cannot exist. Let $V^*$ be this unique unbounded member if it exists, and let $V^*=\emptyset$ otherwise.

    We recursively construct a strictly increasing sequence $\seq{x_n:n\in\omega}$ such that $x_i$ and $x_j$ belong to different members of $\mP$ whenever $i\neq j$:
    \begin{itemize}
        \item First, let $x_0=\min(\omega_1\setminus V^*)$.
        \item Assuming that $\seq{x_k:k\le n}$ has already been defined, for each $k\le n$ let $V_k$ be the unique member of $\mP$ containing $x_k$. Each $V_k$ is bounded, and so is $\bigcup_{k\le n}V_k$. Define
        \[
            x_{n+1}=\min\left(\omega_1\setminus\left(V^*\cup\bigcup_{k\le n}V_k\right)\right).
        \]
    \end{itemize}
    Notice that $i<j$ implies $x_i<x_j$. Let $x_\infty=\sup\set{x_n:n\in\omega}\in\omega_1$, and let $V\in\mP$ contain $x_\infty$. Since $V$ is open, it contains all but finitely many members of $\seq{x_n:n\in\omega}$, contradicting the construction.
\end{proof}

\begin{ex}\label{EX_omega1_BobWins}
    Let us recursively define a winning strategy in $\TreeGame(\omega_1)$ for $\bob$ as follows: 
    
    In the first inning, since $\mP_0$ is an open partition of $\omega_1$, we can use \myref{LEMMA_Omega1Partition} to find some $U_0\in \mP_0$ and an $\eta_0<\omega_1$ such that $[\eta_0,\omega_1)\subseteq U_0$. In this case, we let $\sigma\seq{\mP_0} = U_0$.
        
    At the $\gamma$-th inning, with $\gamma<\omega_1$, suppose $\seq{U_\beta:\beta<\gamma}$ is the sequence of open sets chosen by $\sigma$ in the previous innings, and $\seq{\eta_\beta:\beta<\gamma}$ is such that $[\eta_\beta, \omega_1) \subseteq U_\beta$ for every $\beta<\gamma$. 
        
    If $\gamma = \beta+1<\omega_1$ for some $\beta$, then $\mP_\gamma = \mP_{\beta+1}$ is an open partition of $[\eta_\beta, \omega_1) \subseteq U_\beta$. We can apply \myref{LEMMA_Omega1Partition} to $\mP=\set{U \setminus \eta_{\beta} : U \in \mP_{\gamma}}\cup\{[0,\eta_{\beta})\}$ to find some $U_{\gamma}\in \mP_\gamma$ and an $\eta_{\gamma}<\omega_1$ such that $[\eta_\gamma, \omega_1) \subseteq U_\gamma$. We thus set $\sigma\seq{\mP_\delta:\delta\le\gamma} = U_\gamma$.

    If $\gamma<\omega_1$ is a limit ordinal, since $\gamma$ is countable, we can let $\theta := \sup(\set{\eta_\beta:\beta<\gamma}\cup\{x_{\gamma}+1\})$. Then $\mP_\gamma$ induces the open partition $\set{ U \cap [\theta,\omega_1) : U \in \mP_\gamma}$ of $[\theta,\omega_{1})$, since
    \[
        [\theta, \omega_1) \subseteq \bigcap_{\beta<\gamma}U_\beta = \left(\bigcup \mP_\gamma \right) \sqcup \{x_\gamma\}.
    \]
    It follows from \myref{LEMMA_Omega1Partition} applied to $\{[0,\theta)\} \cup \set{ U \cap [\theta,\omega_1) : U \in \mP_\gamma}$ that there is some $U_{\gamma}\in \mP_\gamma$ and an $\eta_{\gamma}<\omega_1$ such that $[\eta_\gamma, \omega_1) \subseteq U_\gamma$. 
    In this case, let $\sigma\seq{\mP_\delta:\delta\le\gamma} = U_\gamma$.

    Otherwise, if $\gamma\ge\omega_1$, we define $\sigma$ arbitrarily.

    We claim that the constructed $\sigma$ is a winning strategy for Bob.
    Indeed, striving for a contradiction, suppose 
    \[
        R = \seq{\mP_0, U_0, \dotsc, (\mP_\gamma,x_\gamma), U_\gamma, \dotsc, (\emptyset, x_\alpha) }
    \] 
    is a run of length $\alpha$ in which $\bob$ played according to $\sigma$, but $\ali$ wins.
    Note that every run played according to $\sigma$ has length $\ge \omega_1$, so $\alpha\ge\omega_1$.
    However, in view of \myref{LEMMA_WRunOrdinal}, this means that 
    \[
         \omega_1+1\simeq \set{\gamma:\gamma\le\omega_1 \text{ is limit}}\subseteq \set{\gamma:\gamma\le\alpha \text{ is limit}}\simeq \set{x_\gamma:\gamma\le\alpha \text{ is limit}}\subseteq \omega_1,
    \]
    which is impossible, since $\omega_1$ is first countable and $\omega_1 \in \omega_1+1$ does not have a countable local basis.
    Hence, $\bob$ must be the winner of $R$. 
\end{ex}

If $(X,\tau)$ is a topological space and $\sigma$ is a strategy for $\ali$ in $\TreeGame(X)$, then $\sdom{\sigma}$ is a pruned subtree of $\tau^{<|X|^+}$ (ordered by inclusion).

\begin{lemma}\label{LEMMA_TreeStratIso}
    If $T$ is a pruned tree, then there is a winning strategy $\sigma$ for $\ali$ in $\TreeGame(\mR(T))$ such that $\sdom{\sigma}\cong T$.
\end{lemma}
\begin{proof}
     We define the desired strategy $\sigma$ as follows: in the first inning, let 
    $$\sigma\seq{\,} = \set{[t, \emptyset]: t\in T \text{ is a successor of the root in } T}.$$

    Assuming the winning strategy has been defined up to the $\alpha$-th inning, with $\seq{[t_\beta,\emptyset]:\beta<\alpha}$ being $\bob$'s choices such that $\set{t_\beta:\beta<\alpha}$ is downward closed in $T$:
    \begin{itemize}
        \item If $\alpha = \delta+1$, let $$\sigma\seq{[t_\beta,\emptyset]:\beta<\alpha} = \set{[t,\emptyset]: t\in T \text{ is a successor of $t_\delta$ in } T}.$$
        \item If $\alpha$ is a limit ordinal, note that $x_\alpha = \set{t_\beta:\beta<\alpha}\in \mR(T)$. Furthermore, if 
        \[
            \mP_\alpha = \set{[t,\emptyset]: t\in \Rtop(x_\alpha)},
        \]
        then it is clear that
        \[
            \set{[t_\beta,\emptyset]\setminus \bigcup\mF: \beta<\alpha \text{ and finite } \mF \subseteq \mP_\alpha} = \set{[t,F]: t\in x_\alpha \text{ and finite } F\subseteq\Rtop(x_\alpha)}
        \]
        is a local basis for $x_\alpha$, so we may set $\sigma\seq{[t_\beta,\emptyset]:\beta<\alpha} = (\mP_\alpha,x_\alpha)$.
    \end{itemize}
    In view of \myref{LEMMA_WStratEndOfRun}, $\sigma$ is a winning strategy.
    Moreover, letting $r$ denote the root of $T$, the map
$
\Phi \colon T \longrightarrow \sdom{\sigma}
$
defined by
\[
\Phi(r)=\seq{\,}
\quad\text{and}\quad
\Phi(t)=\seq{[s,\emptyset] : r<s<t}^{\smallfrown}[t,\emptyset]
\quad\text{for } t\in T\setminus\{r\}
\]
is an order isomorphism. 
\end{proof}

\begin{lemma}\label{LEMMA_StratHomeo}
    If $X$ is a space and $\sigma$ is a winning strategy for $\ali$ in $\TreeGame(X)$, then $\mR(\sdom{\sigma})\simeq X$.
\end{lemma}
\begin{proof}
    By \myref{LEMMA_WStratEndOfRun}, note that every $R = \seq{U_\beta:\beta<\alpha}\in \mR(\sdom{\sigma})$ is such that $R\in \dom(\sigma)$. 
    In this case, let $g(R)$ be such that $\sigma(R) = (\mP,g(R))$.
    We claim that this $g\colon \mR(\sdom{\sigma})\to X$ is a homeomorphism:
    \begin{description}
        \item[$g$ is injective] Suppose $R,R'\in \mR(\sdom{\sigma})$ are distinct. If $R\subsetneq R'$ (or vice versa), then $\TreeGame(X)$'s rules tell us that $g(R)\notin R'(\beta)$ for any $\beta\ge |R|$, so $g(R)\neq g(R')$. Otherwise, there exists a least ordinal $\beta$ such that $R(\beta) \neq R'(\beta)$. 
        This implies that $R(\beta)$ and $R'(\beta)$ are two distinct open sets in the same partition chosen by $\sigma$ at the $\beta$-th inning, and are thus disjoint. Since $g(R)\in R(\beta)$ and $g(R')\in R'(\beta)$, it follows that $g(R)\neq g(R')$.

        \item[$g$ is surjective] Given $x\in X$, let $R = \seq{U_\beta:\beta<\alpha}$ be the maximal sequence of open sets containing $x$ that $\bob$ can play against $\sigma$.
        Then $\sigma(R) = (\mP, x)$ for some $\mP$, so that $g(R) = x$, as desired. 

        \item[$g$ is open] Let
        \[B = \left[\seq{U_\theta:\theta\le \beta}, \set{R^\smallfrown V:V\in \mF}\right]\]
        be a standard basic open set in $\mR(\sdom{\sigma})$, where $\mF \subseteq \mP$ is a finite subset of the partition $\mP$ such that $\sigma(R)=(\mP,a)$ and $R\in \mR(\sdom{\sigma})$ is such that $R \supseteq \seq{U_\theta:\theta\le \beta}$. Note that $g[B] = U_\beta\setminus \bigcup \mF$, which is open (by \myref{LEMMA_WStratClopens}, since $\sigma$ is a winning strategy). Hence, $g$ is an open map.

        \item[$g$ is continuous] Suppose $U\subseteq X$ is a non-empty open set. 
        Since $g$ has already been shown to be a bijection, let us fix $R=\seq{U_\beta:\beta<\alpha}\in \mR(\sdom{\sigma})\cap g^{-1}(U)$. 

        Because $\sigma$ is a winning strategy, it follows from \myref{EQ_RayGameContinues2}  that there must exist some $\beta<\alpha$ and some finite $\mF\subseteq \mP_\alpha$ such that $g(R)\in U_\beta\setminus \bigcup\mF \subseteq U$. 
        The standard basic open set 
        \[B = [\seq{U_\theta:\theta\le \beta}, \set{\seq{U_\theta:\theta < \alpha }^\smallfrown V:V\in \mF}]\]
        then satisfies $g[B]= U_\beta\setminus \bigcup\mF$, which concludes the proof.
    \end{description}
\end{proof}

For a given space $X$, let $\Sigma(X)$ denote the (possibly empty) collection of all winning strategies for $\ali$ in $\TreeGame(X)$.
Then:

\begin{thm}\label{THM_TreeCollection}
    For a given space $X$, the collection 
    \[\set{\sdom{\sigma}: \sigma\in \Sigma(X)}\]
    comprises, up to isomorphism, the class of all pruned trees whose ray space is homeomorphic to $X$.
\end{thm}
\begin{proof}
     \myref{LEMMA_StratHomeo} implies that if $T\in \set{\sdom{\sigma}: \sigma\in \Sigma(X)}$, then $\mR(T)\simeq X$.

    Conversely, suppose $T$ is a pruned tree and $f\colon \mR(T)\to X$ is a homeomorphism.
    By \myref{LEMMA_TreeStratIso}, let $\sigma$ be a winning strategy for $\ali$ in $\TreeGame(\mR(T))$ such that $\sdom{\sigma}\cong T$. 
    Consider the strategy $f[\sigma]$ for $\ali$ in $\TreeGame(X)$ obtained by translating the partitions and points chosen by $\sigma$ via $f$.
    Since $f$ is a homeomorphism, $f[\sigma]\in \Sigma(X)$. 
    Furthermore, the association 
    \[\seq{U_\beta:\beta<\alpha}^\smallfrown U_\alpha\mapsto \seq{f[U_\beta]:\beta<\alpha}^\smallfrown f[U_\alpha]\]
    is an isomorphism between $\sdom{\sigma}$ and $\sdom{f[\sigma]}$; thus, \[
    \sdom{f[\sigma]}\cong
    T. \]\end{proof}

\begin{cor}\label{COR_RaySPaceChar}
    For a non-empty topological space $X$, the following are equivalent:
    \begin{enumerate}[label=(\alph*)]
        \item $X$ is a ray space.
        \item $\ali$ has a winning strategy in $\TreeGame(X)$.
    \end{enumerate}
\end{cor}

We thus obtain the following topological characterization of when two pruned trees have homeomorphic ray spaces:
\begin{cor}\label{COR_TopTreeCharacterization}
    For any pair $(T,T')$ of pruned trees, $\mR(T)\simeq\mR(T')$ if and only if there exists a winning strategy $\sigma$ for $\ali$ in $\TreeGame(\mR(T))$ such that \[\sdom{\sigma} \cong T'.\]
\end{cor}

Likewise, we can use \myref{THM_PitzRepresentation} together with our ray space characterizations to obtain the following characterizations for end spaces of graphs:

\begin{cor}\label{COR_EndSpaceChar}
    A topological space $X$ is an end space if and only if there is a winning strategy $\sigma$ for $\ali$ in $\TreeGame(X)$ such that $\sdom{\sigma}$ is a special tree.
\end{cor}

\begin{cor}
    For any pair $(G,H)$ of graphs, $\Omega(G)\simeq\Omega(H)$ if and only if there exist winning strategies $\sigma_G$ and $\sigma_H$ for $\ali$ in $\TreeGame(\Omega(G))$ and $\TreeGame(\Omega(H))$, respectively, such that $\sdom{\sigma_G} \cong \sdom{\sigma_H}$.
\end{cor}

In view of \myref{COR_EndSpaceChar}, it is natural to ask:

\begin{question}
    Is there a graph $G$ for which there exists a non-special tree $T$ such that $\Omega(G)\simeq \mR(T)$?
\end{question}

Now, for a graph $G$, let us consider $\Omega_E(G)$ as its so-called \emph{edge-end} space, defined just as $\Omega(G)$, with finite sets of edges playing the role in $\Omega_E(G)$ that finite sets of vertices play in $\Omega(G)$ (these spaces were introduced in \cite{edge-ends} and have since been studied in, for instance,  \cite{PauloGustavoDavide, aurichi2024topologicalremarksendedgeend, aurichilucasedgeconectividade, metrizationtheoremforedgeends,real2025subbasepropertydescribingedgeend}).

We are particularly interested here in the following characterization:
\begin{prop}[Proposition 5.1 in \cite{real2025subbasepropertydescribingedgeend}]
    A space $X$ is homeomorphic to the edge-end space of some graph if and only if $X$ is homeomorphic to the ray space of an order tree $T$ whose height is bounded by $\omega\cdot\omega$ and whose non-maximal rays contain only finitely many nodes with more than one successor. 
\end{prop}

With it, we thus obtain:

\begin{cor}
    A topological space $X$ is homeomorphic to the edge-end space of some graph if and only if there is a winning strategy $\sigma$ for $\ali$ in $\TreeGame(X)$ such that:
    \begin{itemize}
        \item[(a)] the height of $\sdom{\sigma}$ is bounded by $\omega\cdot\omega$;
        \item[(b)] every non-maximal ray in $\sdom{\sigma}$ contains only finitely many nodes with more than one successor.
    \end{itemize} 
\end{cor}

It should be clear by now that \myref{THM_TreeCollection} is particularly useful for characterizing spaces which are notorious for being homeomorphic to some kind of ray space. 
For instance, let us recall that an \emph{ultrametric $d$} over a set $X$ is a metric which satisfies the following stronger version of the triangle inequality:
\begin{equation}\label{EQ_ultrametric}
    \forall x,y,z\in X\left( d(x,y)\le \max\{d(x,z),d(z,y)\} \right).
\end{equation}
We then say that a topological space $X$ is \emph{completely ultrametrizable} if there is some ultrametric $d$ which induces its topology and such that $(X,d)$ is complete.
In this case:

\begin{thm}\label{THM_CompletelyUltrametrizableChar}
    A topological space $X$ is completely ultrametrizable if and only if there is some strategy $\sigma$ for $\ali$ in $\TreeGame(X)$ which always wins at the $\omega$th inning.
\end{thm}
\begin{proof}
    If there is some strategy $\sigma$ for $\ali$ in $\TreeGame(X)$ which always wins at the $\omega$th inning, then \myref{THM_TreeCollection} tells us that $X$ is homeomorphic to the ray space of a pruned tree of height $\omega$, which is always completely ultrametrizable.

    On the other hand, suppose $X$ is completely ultrametrizable and fix an ultrametric $d$ over $X$ which induces $X$'s topology and such that $(X,d)$ is complete.

    For each $x\in X$ and $\varepsilon>0$, let $B_\varepsilon(x)$ be the open ball of radius $\varepsilon$ centered at $x$.
    Then \ref{EQ_ultrametric} tells us that, for all $x,y\in X$ and $\varepsilon>0$, 
    \begin{equation}\label{EQ_UltraOpenBall}
        B_\varepsilon(x)\cap B_\varepsilon(y) \neq \emptyset \, \implies \, B_\varepsilon(x) =  B_\varepsilon(y).  
    \end{equation}
    
    In this case, we recursively define a strategy $\sigma$ for $\ali$ in $\TreeGame(X)$ as follows: firstly, let $\sigma\seq{\,} = \set{B_1(x): x\in X}$ (note that \myref{EQ_UltraOpenBall} guarantees that this is a valid move in $\TreeGame(X)$) and assume $\sigma$ has been defined up to the $n$th inning, with $\seq{ B_{2^{-i}}(x_i): i\le n}$ being $\bob$'s choices thus far. 
    Then we let 
    \[
        \sigma\seq{ B_{2^{-i}}(x_i): i\le n} = \set{B_{2^{-(n+1)}}(x):x\in B_{2^{-n}}(x_n)}.
    \]
    (once again,  \myref{EQ_UltraOpenBall} guarantees that this is a valid move in $\TreeGame(X)$).
    
    Hence, since the diameter of $\bob$'s choices throughout a run compatible with $\sigma$ converges to $0$, it follows that every run compatible with $\sigma$ is won by $\ali$ at the $\omega$th inning.
\end{proof}

\begin{rmk}\label{RMK_WinStratSpreadoutMap}
    Suppose $T$ and $T'$ are trees, $\sigma\in \Sigma(\mR(T'))$, and $\varphi\colon T\to\sdom{\sigma}$ is an isomorphism. 
    
    Let $\psi\colon T\to \bbS(T')$ be such that 
    \[
    \psi(t) = \begin{cases}
        \{(r',\emptyset)\}, & \text{if $t$ is the root of $T$ and $r'$ is the root of $T'$}\\
        \set{(t_R,F_R): R\in U_t}, & \text{otherwise,}
    \end{cases}
    \]
    where $U_t$ is such that $\varphi(t) = \seq{U_\beta:\beta<\alpha}^\smallfrown U_t$ and each $(t_R,F_R)$ is as in \myref{LEMMA_BroadestPair} for $R\in U_t$. 
    Then:
    \begin{itemize}
        \item Since $U_t$ is clopen, \myref{COR_Clopen_SpreadoutPartition} and \myref{PROP_spreadout_clopen_eqv} imply that $\psi(t)\in \bbS(T')$ for every $t\in T$.
        
        \item Condition \ref{item_CombChar_RootCondition} of \myref{THM_CombinatorialCharacterization_SpreaoutMap} holds for $\psi$ explicitly.
        
        \item The minimality of each pair $(t_R,F_R)$ (see \myref{LEMMA_BroadestPair}) and the fact that $U_t\supseteq U_s$ if $t\le s$ imply that condition \ref{item_CombChar_OrderMorphismCondition} of \myref{THM_CombinatorialCharacterization_SpreaoutMap} holds for $\psi$. 
        
        \item Since $\sigma$ chooses clopen partitions, \myref{PROP_RayDisjointCollection} ensures that conditions \ref{item_CombChar_IncompatibleCondition} and \ref{item_CombChar_CoveringCondition} of \myref{THM_CombinatorialCharacterization_SpreaoutMap} hold for $\psi$.
        
        \item Finally, the fact that $\sigma$ is a winning strategy for $\ali$ translates, via \myref{PROP_spreadout_clopen_eqv}, to condition \ref{item_CombChar_LimitCondition} holding for $\psi$.
    \end{itemize}

    As previously mentioned, \myref{COR_TopTreeCharacterization} was the first characterization obtained for trees with homeomorphic ray spaces; the construction above thus motivated \myref{THM_CombinatorialCharacterization_SpreaoutMap}.

    We opted to begin the paper with \myref{THM_CombinatorialCharacterization_SpreaoutMap} because it is one of our main results and its proof does not rely on \myref{COR_TopTreeCharacterization}, despite being motivated by it.

    Conversely, the characterization in \myref{COR_TopTreeCharacterization}, as a consequence of \myref{THM_TreeCollection}, will be shown throughout the remainder of this paper to be particularly useful for several applications and generalizations.
\end{rmk}

Now consider the following result from the literature:

\begin{thm}[Particular case of Theorem 4.2 in \cite{carvalho2026coveringpropertiesendray}]\label{THM_ScatteredCharacterization}
    For a ray space $X$, the following are equivalent:
    \begin{itemize}
        \item[(a)] $X$ is scattered (i.e., every non-empty subset of $X$ has an isolated point).
        \item[(b)] $X$ contains no topological copy of the Cantor space $2^\omega$.
    \end{itemize}
\end{thm}

Let $\RR_l$ denote the Sorgenfrey line. Since $\RR_l$ has no isolated points and every compact $K\subset \RR_l$ is countable (Exercise 3.1.B in \cite{engelking}), it follows from \myref{THM_ScatteredCharacterization} that $\RR_l$ cannot be a ray space. 
However, we can further strengthen this observation (in view of \myref{COR_RaySPaceChar}):
    
\begin{ex}\label{EX_SorgenfreyNotRay}
    Consider the following strategy $\sigma$ for $\bob$ in $\TreeGame(\RR_l)$.

    For every inning in which $\ali$'s chosen partition contains no lower-bounded open set, let $\sigma$'s choice be arbitrary. 
    
    Suppose $\alpha_0$ is the first inning such that there exists some lower-bounded $U_{\alpha_0}\in \mP_{\alpha_0}$, in which case we let $\sigma$ pick $U_{\alpha_0}$ and we set $p_{\alpha_0} = \inf(U_{\alpha_0})$. 
    
    For an arbitrary $\alpha$-th inning after $\alpha_0$, assume $\seq{U_\beta:\alpha_0<\beta<\alpha}$ is the sequence of open sets chosen by $\sigma$ and $\seq{p_\beta:\alpha_0<\beta<\alpha}$ is such that $p_\beta= \inf(U_\beta)$.
    \begin{description}
        \item[If $\alpha = \gamma+1$] Provided there is some $U_\alpha\in \mP_\alpha$ with $p_\gamma\notin U_\alpha$, let $\sigma$ pick such $U_\alpha$. Otherwise, let $\sigma$ pick $U_\alpha=U_\gamma\in \mP_\alpha$.

        \item[If $\alpha$ is a limit ordinal] If $\seq{p_\beta:\alpha_0<\beta<\alpha}$ is unbounded, then $\bigcap_{\beta<\alpha} U_\beta = \emptyset$ and the run ends with $\bob$ winning. Otherwise, let $p = \sup(\set{p_\beta:\alpha_0<\beta<\alpha})\in \RR_l$. If there is some $U_\alpha\in \mP_\alpha$ with $p\notin U_\alpha$, let $\sigma$ pick such $U_\alpha$. 
        If no such $U_\alpha$ exists, let $\sigma$ pick the only possible $U_\alpha\in \mP_\alpha$.
    \end{description}

    We claim that $\sigma$ is a winning strategy for $\bob$.
    Striving for a contradiction, suppose 
    \[
        R = \seq{\mP_0, U_0, \dotsc, (\mP_\gamma,x_\gamma), U_\gamma, \dotsc , (\emptyset, x_\alpha)}
    \] 
    is a run of length $\alpha$ in which $\bob$ played according to $\sigma$, but $\ali$ wins.
    By \myref{LEMMA_WStratClopens}, $U_\beta$ is clopen in $\RR_l$ for every $\beta<\alpha$. 
    There must be some $\alpha_0<\alpha$ for which $U_{\alpha_0}$ is lower-bounded by $x_\alpha$ (since $\set{U_\beta:\beta<\alpha}$ is a local basis for $x_\alpha$). 
    Thus, $\seq{p_\beta:\alpha_0<\beta<\alpha}$ is an increasing sequence in $\RR_l$ bounded above by $x_\alpha$.
    Since $p_\beta=p_{\beta+1}$ if and only if $\mP_{\beta+1} = \{U_\beta\}$, it follows that $x_\alpha > p_\beta$ for every $\alpha_0<\beta<\alpha$.
    
    Hence, $\set{U_\beta:\beta<\alpha}$ is not a local basis for $x_\alpha$ (as no $U_\beta$ is contained in $[x_\alpha,+\infty)$), a contradiction.
\end{ex}

The following simple example illustrates how one can explicitly use winning strategies $\sigma\in \Sigma(X)$ to express some different trees whose ray space is homeomorphic to $X$.

\begin{ex}[The convergent sequence]
    Suppose $X$ is a non-trivial convergent sequence (i.e., $X = \set{x_n:n\le\omega}$ where $\seq{x_n:n\in\omega}$ converges to $x_\omega$). 
    We consider a few examples of winning strategies $\sigma$ for $\ali$ in $\TreeGame(X)$ that produce notable representations of $X$ as a ray space.

    \begin{description}
        \item[Comb of rays] In the first inning, let $\sigma\seq{\,} = \{\{x_0\}, \set{x_k:0 < k \le \omega}\}$.
        At the $(n+1)$-th inning, let $\seq{U_i:i\le n}$ be the choices made by $\bob$, where $U_n$ is either a singleton or $\set{x_k:n<k\le\omega}$. We set 
        \[
          \sigma\seq{U_i:i\le n} = \begin{cases}
              \{U_n\}, & \text{if $U_n$ is a singleton,}\\
              \{\{x_{n+1}\}, \set{x_k:n+1<k\le \omega}\}, & \text{otherwise}.
          \end{cases}  
        \]
        Then $\sigma$ is a winning strategy, and $\sdom{\sigma}$ is a "comb of rays" (see \myref{FIG_RayComb}).
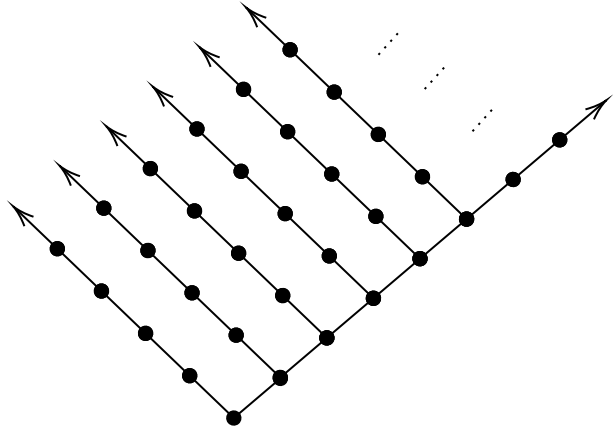
\begin{figure}
    \centering

\tikzset{every picture/.style={line width=0.75pt}} 

\begin{tikzpicture}[x=0.75pt,y=0.75pt,yscale=-1,xscale=1]

\draw    (285.26,226.67) -- (308.63,206.8) ;
\draw [shift={(308.63,206.8)}, rotate = 319.63] [color={rgb, 255:red, 0; green, 0; blue, 0 }  ][fill={rgb, 255:red, 0; green, 0; blue, 0 }  ][line width=0.75]      (0, 0) circle [x radius= 3.35, y radius= 3.35]   ;
\draw [shift={(285.26,226.67)}, rotate = 319.63] [color={rgb, 255:red, 0; green, 0; blue, 0 }  ][fill={rgb, 255:red, 0; green, 0; blue, 0 }  ][line width=0.75]      (0, 0) circle [x radius= 3.35, y radius= 3.35]   ;
\draw    (308.58,206.33) -- (331.95,186.45) ;
\draw [shift={(331.95,186.45)}, rotate = 319.63] [color={rgb, 255:red, 0; green, 0; blue, 0 }  ][fill={rgb, 255:red, 0; green, 0; blue, 0 }  ][line width=0.75]      (0, 0) circle [x radius= 3.35, y radius= 3.35]   ;
\draw [shift={(308.58,206.33)}, rotate = 319.63] [color={rgb, 255:red, 0; green, 0; blue, 0 }  ][fill={rgb, 255:red, 0; green, 0; blue, 0 }  ][line width=0.75]      (0, 0) circle [x radius= 3.35, y radius= 3.35]   ;
\draw    (331.95,186.45) -- (355.33,166.58) ;
\draw [shift={(355.33,166.58)}, rotate = 319.63] [color={rgb, 255:red, 0; green, 0; blue, 0 }  ][fill={rgb, 255:red, 0; green, 0; blue, 0 }  ][line width=0.75]      (0, 0) circle [x radius= 3.35, y radius= 3.35]   ;
\draw [shift={(331.95,186.45)}, rotate = 319.63] [color={rgb, 255:red, 0; green, 0; blue, 0 }  ][fill={rgb, 255:red, 0; green, 0; blue, 0 }  ][line width=0.75]      (0, 0) circle [x radius= 3.35, y radius= 3.35]   ;
\draw    (355.33,166.58) -- (378.7,146.7) ;
\draw [shift={(378.7,146.7)}, rotate = 319.63] [color={rgb, 255:red, 0; green, 0; blue, 0 }  ][fill={rgb, 255:red, 0; green, 0; blue, 0 }  ][line width=0.75]      (0, 0) circle [x radius= 3.35, y radius= 3.35]   ;
\draw [shift={(355.33,166.58)}, rotate = 319.63] [color={rgb, 255:red, 0; green, 0; blue, 0 }  ][fill={rgb, 255:red, 0; green, 0; blue, 0 }  ][line width=0.75]      (0, 0) circle [x radius= 3.35, y radius= 3.35]   ;
\draw    (378.7,146.7) -- (402.08,126.83) ;
\draw [shift={(402.08,126.83)}, rotate = 319.63] [color={rgb, 255:red, 0; green, 0; blue, 0 }  ][fill={rgb, 255:red, 0; green, 0; blue, 0 }  ][line width=0.75]      (0, 0) circle [x radius= 3.35, y radius= 3.35]   ;
\draw [shift={(378.7,146.7)}, rotate = 319.63] [color={rgb, 255:red, 0; green, 0; blue, 0 }  ][fill={rgb, 255:red, 0; green, 0; blue, 0 }  ][line width=0.75]      (0, 0) circle [x radius= 3.35, y radius= 3.35]   ;
\draw    (402.08,126.83) -- (425.45,106.95) ;
\draw [shift={(425.45,106.95)}, rotate = 319.63] [color={rgb, 255:red, 0; green, 0; blue, 0 }  ][fill={rgb, 255:red, 0; green, 0; blue, 0 }  ][line width=0.75]      (0, 0) circle [x radius= 3.35, y radius= 3.35]   ;
\draw [shift={(402.08,126.83)}, rotate = 319.63] [color={rgb, 255:red, 0; green, 0; blue, 0 }  ][fill={rgb, 255:red, 0; green, 0; blue, 0 }  ][line width=0.75]      (0, 0) circle [x radius= 3.35, y radius= 3.35]   ;
\draw    (425.45,106.95) -- (448.83,87.07) ;
\draw [shift={(448.83,87.07)}, rotate = 319.63] [color={rgb, 255:red, 0; green, 0; blue, 0 }  ][fill={rgb, 255:red, 0; green, 0; blue, 0 }  ][line width=0.75]      (0, 0) circle [x radius= 3.35, y radius= 3.35]   ;
\draw [shift={(425.45,106.95)}, rotate = 319.63] [color={rgb, 255:red, 0; green, 0; blue, 0 }  ][fill={rgb, 255:red, 0; green, 0; blue, 0 }  ][line width=0.75]      (0, 0) circle [x radius= 3.35, y radius= 3.35]   ;
\draw    (448.83,87.07) -- (470.68,68.5) ;
\draw [shift={(472.2,67.2)}, rotate = 139.63] [color={rgb, 255:red, 0; green, 0; blue, 0 }  ][line width=0.75]    (10.93,-3.29) .. controls (6.95,-1.4) and (3.31,-0.3) .. (0,0) .. controls (3.31,0.3) and (6.95,1.4) .. (10.93,3.29)   ;
\draw [shift={(448.83,87.07)}, rotate = 319.63] [color={rgb, 255:red, 0; green, 0; blue, 0 }  ][fill={rgb, 255:red, 0; green, 0; blue, 0 }  ][line width=0.75]      (0, 0) circle [x radius= 3.35, y radius= 3.35]   ;
\draw    (263.12,205.43) -- (240.98,184.19) ;
\draw [shift={(240.98,184.19)}, rotate = 223.82] [color={rgb, 255:red, 0; green, 0; blue, 0 }  ][fill={rgb, 255:red, 0; green, 0; blue, 0 }  ][line width=0.75]      (0, 0) circle [x radius= 3.35, y radius= 3.35]   ;
\draw [shift={(263.12,205.43)}, rotate = 223.82] [color={rgb, 255:red, 0; green, 0; blue, 0 }  ][fill={rgb, 255:red, 0; green, 0; blue, 0 }  ][line width=0.75]      (0, 0) circle [x radius= 3.35, y radius= 3.35]   ;
\draw    (240.98,184.19) -- (218.84,162.94) ;
\draw [shift={(218.84,162.94)}, rotate = 223.82] [color={rgb, 255:red, 0; green, 0; blue, 0 }  ][fill={rgb, 255:red, 0; green, 0; blue, 0 }  ][line width=0.75]      (0, 0) circle [x radius= 3.35, y radius= 3.35]   ;
\draw [shift={(240.98,184.19)}, rotate = 223.82] [color={rgb, 255:red, 0; green, 0; blue, 0 }  ][fill={rgb, 255:red, 0; green, 0; blue, 0 }  ][line width=0.75]      (0, 0) circle [x radius= 3.35, y radius= 3.35]   ;
\draw    (218.84,162.94) -- (196.7,141.7) ;
\draw [shift={(196.7,141.7)}, rotate = 223.82] [color={rgb, 255:red, 0; green, 0; blue, 0 }  ][fill={rgb, 255:red, 0; green, 0; blue, 0 }  ][line width=0.75]      (0, 0) circle [x radius= 3.35, y radius= 3.35]   ;
\draw [shift={(218.84,162.94)}, rotate = 223.82] [color={rgb, 255:red, 0; green, 0; blue, 0 }  ][fill={rgb, 255:red, 0; green, 0; blue, 0 }  ][line width=0.75]      (0, 0) circle [x radius= 3.35, y radius= 3.35]   ;
\draw    (196.7,141.7) -- (176,121.84) ;
\draw [shift={(174.56,120.46)}, rotate = 43.82] [color={rgb, 255:red, 0; green, 0; blue, 0 }  ][line width=0.75]    (10.93,-3.29) .. controls (6.95,-1.4) and (3.31,-0.3) .. (0,0) .. controls (3.31,0.3) and (6.95,1.4) .. (10.93,3.29)   ;
\draw [shift={(196.7,141.7)}, rotate = 223.82] [color={rgb, 255:red, 0; green, 0; blue, 0 }  ][fill={rgb, 255:red, 0; green, 0; blue, 0 }  ][line width=0.75]      (0, 0) circle [x radius= 3.35, y radius= 3.35]   ;
\draw    (308.58,206.33) -- (286.44,185.08) ;
\draw [shift={(286.44,185.08)}, rotate = 223.82] [color={rgb, 255:red, 0; green, 0; blue, 0 }  ][fill={rgb, 255:red, 0; green, 0; blue, 0 }  ][line width=0.75]      (0, 0) circle [x radius= 3.35, y radius= 3.35]   ;
\draw [shift={(308.58,206.33)}, rotate = 223.82] [color={rgb, 255:red, 0; green, 0; blue, 0 }  ][fill={rgb, 255:red, 0; green, 0; blue, 0 }  ][line width=0.75]      (0, 0) circle [x radius= 3.35, y radius= 3.35]   ;
\draw    (286.44,185.08) -- (264.3,163.84) ;
\draw [shift={(264.3,163.84)}, rotate = 223.82] [color={rgb, 255:red, 0; green, 0; blue, 0 }  ][fill={rgb, 255:red, 0; green, 0; blue, 0 }  ][line width=0.75]      (0, 0) circle [x radius= 3.35, y radius= 3.35]   ;
\draw [shift={(286.44,185.08)}, rotate = 223.82] [color={rgb, 255:red, 0; green, 0; blue, 0 }  ][fill={rgb, 255:red, 0; green, 0; blue, 0 }  ][line width=0.75]      (0, 0) circle [x radius= 3.35, y radius= 3.35]   ;
\draw    (264.3,163.84) -- (242.16,142.6) ;
\draw [shift={(242.16,142.6)}, rotate = 223.82] [color={rgb, 255:red, 0; green, 0; blue, 0 }  ][fill={rgb, 255:red, 0; green, 0; blue, 0 }  ][line width=0.75]      (0, 0) circle [x radius= 3.35, y radius= 3.35]   ;
\draw [shift={(264.3,163.84)}, rotate = 223.82] [color={rgb, 255:red, 0; green, 0; blue, 0 }  ][fill={rgb, 255:red, 0; green, 0; blue, 0 }  ][line width=0.75]      (0, 0) circle [x radius= 3.35, y radius= 3.35]   ;
\draw    (242.16,142.6) -- (220.02,121.35) ;
\draw [shift={(220.02,121.35)}, rotate = 223.82] [color={rgb, 255:red, 0; green, 0; blue, 0 }  ][fill={rgb, 255:red, 0; green, 0; blue, 0 }  ][line width=0.75]      (0, 0) circle [x radius= 3.35, y radius= 3.35]   ;
\draw [shift={(242.16,142.6)}, rotate = 223.82] [color={rgb, 255:red, 0; green, 0; blue, 0 }  ][fill={rgb, 255:red, 0; green, 0; blue, 0 }  ][line width=0.75]      (0, 0) circle [x radius= 3.35, y radius= 3.35]   ;
\draw    (220.02,121.35) -- (199.32,101.49) ;
\draw [shift={(197.88,100.11)}, rotate = 43.82] [color={rgb, 255:red, 0; green, 0; blue, 0 }  ][line width=0.75]    (10.93,-3.29) .. controls (6.95,-1.4) and (3.31,-0.3) .. (0,0) .. controls (3.31,0.3) and (6.95,1.4) .. (10.93,3.29)   ;
\draw [shift={(220.02,121.35)}, rotate = 223.82] [color={rgb, 255:red, 0; green, 0; blue, 0 }  ][fill={rgb, 255:red, 0; green, 0; blue, 0 }  ][line width=0.75]      (0, 0) circle [x radius= 3.35, y radius= 3.35]   ;
\draw    (331.95,186.45) -- (309.81,165.21) ;
\draw [shift={(309.81,165.21)}, rotate = 223.82] [color={rgb, 255:red, 0; green, 0; blue, 0 }  ][fill={rgb, 255:red, 0; green, 0; blue, 0 }  ][line width=0.75]      (0, 0) circle [x radius= 3.35, y radius= 3.35]   ;
\draw [shift={(331.95,186.45)}, rotate = 223.82] [color={rgb, 255:red, 0; green, 0; blue, 0 }  ][fill={rgb, 255:red, 0; green, 0; blue, 0 }  ][line width=0.75]      (0, 0) circle [x radius= 3.35, y radius= 3.35]   ;
\draw    (309.81,165.21) -- (287.67,143.96) ;
\draw [shift={(287.67,143.96)}, rotate = 223.82] [color={rgb, 255:red, 0; green, 0; blue, 0 }  ][fill={rgb, 255:red, 0; green, 0; blue, 0 }  ][line width=0.75]      (0, 0) circle [x radius= 3.35, y radius= 3.35]   ;
\draw [shift={(309.81,165.21)}, rotate = 223.82] [color={rgb, 255:red, 0; green, 0; blue, 0 }  ][fill={rgb, 255:red, 0; green, 0; blue, 0 }  ][line width=0.75]      (0, 0) circle [x radius= 3.35, y radius= 3.35]   ;
\draw    (287.67,143.96) -- (265.53,122.72) ;
\draw [shift={(265.53,122.72)}, rotate = 223.82] [color={rgb, 255:red, 0; green, 0; blue, 0 }  ][fill={rgb, 255:red, 0; green, 0; blue, 0 }  ][line width=0.75]      (0, 0) circle [x radius= 3.35, y radius= 3.35]   ;
\draw [shift={(287.67,143.96)}, rotate = 223.82] [color={rgb, 255:red, 0; green, 0; blue, 0 }  ][fill={rgb, 255:red, 0; green, 0; blue, 0 }  ][line width=0.75]      (0, 0) circle [x radius= 3.35, y radius= 3.35]   ;
\draw    (265.53,122.72) -- (243.39,101.48) ;
\draw [shift={(243.39,101.48)}, rotate = 223.82] [color={rgb, 255:red, 0; green, 0; blue, 0 }  ][fill={rgb, 255:red, 0; green, 0; blue, 0 }  ][line width=0.75]      (0, 0) circle [x radius= 3.35, y radius= 3.35]   ;
\draw [shift={(265.53,122.72)}, rotate = 223.82] [color={rgb, 255:red, 0; green, 0; blue, 0 }  ][fill={rgb, 255:red, 0; green, 0; blue, 0 }  ][line width=0.75]      (0, 0) circle [x radius= 3.35, y radius= 3.35]   ;
\draw    (355.33,166.58) -- (333.19,145.33) ;
\draw [shift={(333.19,145.33)}, rotate = 223.82] [color={rgb, 255:red, 0; green, 0; blue, 0 }  ][fill={rgb, 255:red, 0; green, 0; blue, 0 }  ][line width=0.75]      (0, 0) circle [x radius= 3.35, y radius= 3.35]   ;
\draw [shift={(355.33,166.58)}, rotate = 223.82] [color={rgb, 255:red, 0; green, 0; blue, 0 }  ][fill={rgb, 255:red, 0; green, 0; blue, 0 }  ][line width=0.75]      (0, 0) circle [x radius= 3.35, y radius= 3.35]   ;
\draw    (333.19,145.33) -- (311.05,124.09) ;
\draw [shift={(311.05,124.09)}, rotate = 223.82] [color={rgb, 255:red, 0; green, 0; blue, 0 }  ][fill={rgb, 255:red, 0; green, 0; blue, 0 }  ][line width=0.75]      (0, 0) circle [x radius= 3.35, y radius= 3.35]   ;
\draw [shift={(333.19,145.33)}, rotate = 223.82] [color={rgb, 255:red, 0; green, 0; blue, 0 }  ][fill={rgb, 255:red, 0; green, 0; blue, 0 }  ][line width=0.75]      (0, 0) circle [x radius= 3.35, y radius= 3.35]   ;
\draw    (311.05,124.09) -- (288.91,102.85) ;
\draw [shift={(288.91,102.85)}, rotate = 223.82] [color={rgb, 255:red, 0; green, 0; blue, 0 }  ][fill={rgb, 255:red, 0; green, 0; blue, 0 }  ][line width=0.75]      (0, 0) circle [x radius= 3.35, y radius= 3.35]   ;
\draw [shift={(311.05,124.09)}, rotate = 223.82] [color={rgb, 255:red, 0; green, 0; blue, 0 }  ][fill={rgb, 255:red, 0; green, 0; blue, 0 }  ][line width=0.75]      (0, 0) circle [x radius= 3.35, y radius= 3.35]   ;
\draw    (288.91,102.85) -- (266.77,81.6) ;
\draw [shift={(266.77,81.6)}, rotate = 223.82] [color={rgb, 255:red, 0; green, 0; blue, 0 }  ][fill={rgb, 255:red, 0; green, 0; blue, 0 }  ][line width=0.75]      (0, 0) circle [x radius= 3.35, y radius= 3.35]   ;
\draw [shift={(288.91,102.85)}, rotate = 223.82] [color={rgb, 255:red, 0; green, 0; blue, 0 }  ][fill={rgb, 255:red, 0; green, 0; blue, 0 }  ][line width=0.75]      (0, 0) circle [x radius= 3.35, y radius= 3.35]   ;
\draw    (378.7,146.7) -- (356.56,125.46) ;
\draw [shift={(356.56,125.46)}, rotate = 223.82] [color={rgb, 255:red, 0; green, 0; blue, 0 }  ][fill={rgb, 255:red, 0; green, 0; blue, 0 }  ][line width=0.75]      (0, 0) circle [x radius= 3.35, y radius= 3.35]   ;
\draw [shift={(378.7,146.7)}, rotate = 223.82] [color={rgb, 255:red, 0; green, 0; blue, 0 }  ][fill={rgb, 255:red, 0; green, 0; blue, 0 }  ][line width=0.75]      (0, 0) circle [x radius= 3.35, y radius= 3.35]   ;
\draw    (356.56,125.46) -- (334.42,104.21) ;
\draw [shift={(334.42,104.21)}, rotate = 223.82] [color={rgb, 255:red, 0; green, 0; blue, 0 }  ][fill={rgb, 255:red, 0; green, 0; blue, 0 }  ][line width=0.75]      (0, 0) circle [x radius= 3.35, y radius= 3.35]   ;
\draw [shift={(356.56,125.46)}, rotate = 223.82] [color={rgb, 255:red, 0; green, 0; blue, 0 }  ][fill={rgb, 255:red, 0; green, 0; blue, 0 }  ][line width=0.75]      (0, 0) circle [x radius= 3.35, y radius= 3.35]   ;
\draw    (334.42,104.21) -- (312.28,82.97) ;
\draw [shift={(312.28,82.97)}, rotate = 223.82] [color={rgb, 255:red, 0; green, 0; blue, 0 }  ][fill={rgb, 255:red, 0; green, 0; blue, 0 }  ][line width=0.75]      (0, 0) circle [x radius= 3.35, y radius= 3.35]   ;
\draw [shift={(334.42,104.21)}, rotate = 223.82] [color={rgb, 255:red, 0; green, 0; blue, 0 }  ][fill={rgb, 255:red, 0; green, 0; blue, 0 }  ][line width=0.75]      (0, 0) circle [x radius= 3.35, y radius= 3.35]   ;
\draw    (312.28,82.97) -- (290.14,61.73) ;
\draw [shift={(290.14,61.73)}, rotate = 223.82] [color={rgb, 255:red, 0; green, 0; blue, 0 }  ][fill={rgb, 255:red, 0; green, 0; blue, 0 }  ][line width=0.75]      (0, 0) circle [x radius= 3.35, y radius= 3.35]   ;
\draw [shift={(312.28,82.97)}, rotate = 223.82] [color={rgb, 255:red, 0; green, 0; blue, 0 }  ][fill={rgb, 255:red, 0; green, 0; blue, 0 }  ][line width=0.75]      (0, 0) circle [x radius= 3.35, y radius= 3.35]   ;
\draw    (402.08,126.83) -- (379.94,105.58) ;
\draw [shift={(379.94,105.58)}, rotate = 223.82] [color={rgb, 255:red, 0; green, 0; blue, 0 }  ][fill={rgb, 255:red, 0; green, 0; blue, 0 }  ][line width=0.75]      (0, 0) circle [x radius= 3.35, y radius= 3.35]   ;
\draw [shift={(402.08,126.83)}, rotate = 223.82] [color={rgb, 255:red, 0; green, 0; blue, 0 }  ][fill={rgb, 255:red, 0; green, 0; blue, 0 }  ][line width=0.75]      (0, 0) circle [x radius= 3.35, y radius= 3.35]   ;
\draw    (379.94,105.58) -- (357.8,84.34) ;
\draw [shift={(357.8,84.34)}, rotate = 223.82] [color={rgb, 255:red, 0; green, 0; blue, 0 }  ][fill={rgb, 255:red, 0; green, 0; blue, 0 }  ][line width=0.75]      (0, 0) circle [x radius= 3.35, y radius= 3.35]   ;
\draw [shift={(379.94,105.58)}, rotate = 223.82] [color={rgb, 255:red, 0; green, 0; blue, 0 }  ][fill={rgb, 255:red, 0; green, 0; blue, 0 }  ][line width=0.75]      (0, 0) circle [x radius= 3.35, y radius= 3.35]   ;
\draw    (357.8,84.34) -- (335.66,63.1) ;
\draw [shift={(335.66,63.1)}, rotate = 223.82] [color={rgb, 255:red, 0; green, 0; blue, 0 }  ][fill={rgb, 255:red, 0; green, 0; blue, 0 }  ][line width=0.75]      (0, 0) circle [x radius= 3.35, y radius= 3.35]   ;
\draw [shift={(357.8,84.34)}, rotate = 223.82] [color={rgb, 255:red, 0; green, 0; blue, 0 }  ][fill={rgb, 255:red, 0; green, 0; blue, 0 }  ][line width=0.75]      (0, 0) circle [x radius= 3.35, y radius= 3.35]   ;
\draw    (335.66,63.1) -- (313.52,41.85) ;
\draw [shift={(313.52,41.85)}, rotate = 223.82] [color={rgb, 255:red, 0; green, 0; blue, 0 }  ][fill={rgb, 255:red, 0; green, 0; blue, 0 }  ][line width=0.75]      (0, 0) circle [x radius= 3.35, y radius= 3.35]   ;
\draw [shift={(335.66,63.1)}, rotate = 223.82] [color={rgb, 255:red, 0; green, 0; blue, 0 }  ][fill={rgb, 255:red, 0; green, 0; blue, 0 }  ][line width=0.75]      (0, 0) circle [x radius= 3.35, y radius= 3.35]   ;
\draw  [dash pattern={on 0.84pt off 2.51pt}]  (414.8,72) -- (403.4,84.6) ;
\draw  [dash pattern={on 0.84pt off 2.51pt}]  (390.8,50.8) -- (379.4,63.4) ;
\draw  [dash pattern={on 0.84pt off 2.51pt}]  (367.6,33.6) -- (356.2,46.2) ;
\draw    (243.39,101.48) -- (222.7,81.62) ;
\draw [shift={(221.25,80.24)}, rotate = 43.82] [color={rgb, 255:red, 0; green, 0; blue, 0 }  ][line width=0.75]    (10.93,-3.29) .. controls (6.95,-1.4) and (3.31,-0.3) .. (0,0) .. controls (3.31,0.3) and (6.95,1.4) .. (10.93,3.29)   ;
\draw [shift={(243.39,101.48)}, rotate = 223.82] [color={rgb, 255:red, 0; green, 0; blue, 0 }  ][fill={rgb, 255:red, 0; green, 0; blue, 0 }  ][line width=0.75]      (0, 0) circle [x radius= 3.35, y radius= 3.35]   ;
\draw    (266.77,81.6) -- (246.07,61.74) ;
\draw [shift={(244.63,60.36)}, rotate = 43.82] [color={rgb, 255:red, 0; green, 0; blue, 0 }  ][line width=0.75]    (10.93,-3.29) .. controls (6.95,-1.4) and (3.31,-0.3) .. (0,0) .. controls (3.31,0.3) and (6.95,1.4) .. (10.93,3.29)   ;
\draw [shift={(266.77,81.6)}, rotate = 223.82] [color={rgb, 255:red, 0; green, 0; blue, 0 }  ][fill={rgb, 255:red, 0; green, 0; blue, 0 }  ][line width=0.75]      (0, 0) circle [x radius= 3.35, y radius= 3.35]   ;
\draw    (290.14,61.73) -- (269.45,41.87) ;
\draw [shift={(268,40.49)}, rotate = 43.82] [color={rgb, 255:red, 0; green, 0; blue, 0 }  ][line width=0.75]    (10.93,-3.29) .. controls (6.95,-1.4) and (3.31,-0.3) .. (0,0) .. controls (3.31,0.3) and (6.95,1.4) .. (10.93,3.29)   ;
\draw [shift={(290.14,61.73)}, rotate = 223.82] [color={rgb, 255:red, 0; green, 0; blue, 0 }  ][fill={rgb, 255:red, 0; green, 0; blue, 0 }  ][line width=0.75]      (0, 0) circle [x radius= 3.35, y radius= 3.35]   ;
\draw    (313.52,41.85) -- (292.82,21.99) ;
\draw [shift={(291.38,20.61)}, rotate = 43.82] [color={rgb, 255:red, 0; green, 0; blue, 0 }  ][line width=0.75]    (10.93,-3.29) .. controls (6.95,-1.4) and (3.31,-0.3) .. (0,0) .. controls (3.31,0.3) and (6.95,1.4) .. (10.93,3.29)   ;
\draw [shift={(313.52,41.85)}, rotate = 223.82] [color={rgb, 255:red, 0; green, 0; blue, 0 }  ][fill={rgb, 255:red, 0; green, 0; blue, 0 }  ][line width=0.75]      (0, 0) circle [x radius= 3.35, y radius= 3.35]   ;
\draw    (285.26,226.67) -- (263.12,205.43) ;
\draw [shift={(263.12,205.43)}, rotate = 223.82] [color={rgb, 255:red, 0; green, 0; blue, 0 }  ][fill={rgb, 255:red, 0; green, 0; blue, 0 }  ][line width=0.75]      (0, 0) circle [x radius= 3.35, y radius= 3.35]   ;


\end{tikzpicture}
    \caption{Representation of the comb of rays, whose ray space is homeomorphic to the convergent sequence.}
    \label{FIG_RayComb}
\end{figure}
        
        \item[$\omega\cdot\omega$] For every inning $n\in\omega$, let $\sigma$ pick the trivial partition $\{X\}$. At each nonzero limit inning $(n+1)\cdot\omega$, for $n\in\omega$, set
        \[
            \sigma\seq{U_\beta:\beta<(n+1)\cdot\omega}
            = (\{\set{x_k:n<k\le\omega}\},x_n).
        \]
        For every successor ordinal $\alpha=\beta+1<\omega\cdot\omega$, let $\sigma$ pick the same trivial partition as in the previous inning. At the terminal inning $\omega\cdot\omega$, set
        \[
            \sigma\seq{U_\beta:\beta<\omega\cdot\omega}=(\emptyset,x_\omega).
        \]
        This $\sigma$ is a winning strategy with a single run ending at $\omega\cdot\omega$. $\sdom{\sigma}$ is illustrated in \myref{FIG_omega.omega}.
\begin{figure}
    \centering

\tikzset{every picture/.style={line width=0.75pt}} 

\begin{tikzpicture}[x=0.75pt,y=0.75pt,yscale=-1,xscale=1]

\draw  [dash pattern={on 0.84pt off 2.51pt}]  (305,14) -- (305.3,33.1) ;
\draw    (305.31,35.1) -- (305.6,78.2) ;
\draw [shift={(305.6,78.2)}, rotate = 89.62] [color={rgb, 255:red, 0; green, 0; blue, 0 }  ][fill={rgb, 255:red, 0; green, 0; blue, 0 }  ][line width=0.75]      (0, 0) circle [x radius= 3.35, y radius= 3.35]   ;
\draw [shift={(305.3,33.1)}, rotate = 89.62] [color={rgb, 255:red, 0; green, 0; blue, 0 }  ][line width=0.75]    (10.93,-4.9) .. controls (6.95,-2.3) and (3.31,-0.67) .. (0,0) .. controls (3.31,0.67) and (6.95,2.3) .. (10.93,4.9)   ;
\draw    (305.61,84.7) -- (305.9,127.8) ;
\draw [shift={(305.9,127.8)}, rotate = 89.62] [color={rgb, 255:red, 0; green, 0; blue, 0 }  ][fill={rgb, 255:red, 0; green, 0; blue, 0 }  ][line width=0.75]      (0, 0) circle [x radius= 3.35, y radius= 3.35]   ;
\draw [shift={(305.6,82.7)}, rotate = 89.62] [color={rgb, 255:red, 0; green, 0; blue, 0 }  ][line width=0.75]    (10.93,-4.9) .. controls (6.95,-2.3) and (3.31,-0.67) .. (0,0) .. controls (3.31,0.67) and (6.95,2.3) .. (10.93,4.9)   ;
\draw    (305.91,134.8) -- (306.2,177.9) ;
\draw [shift={(306.2,177.9)}, rotate = 89.62] [color={rgb, 255:red, 0; green, 0; blue, 0 }  ][fill={rgb, 255:red, 0; green, 0; blue, 0 }  ][line width=0.75]      (0, 0) circle [x radius= 3.35, y radius= 3.35]   ;
\draw [shift={(305.9,132.8)}, rotate = 89.62] [color={rgb, 255:red, 0; green, 0; blue, 0 }  ][line width=0.75]    (10.93,-4.9) .. controls (6.95,-2.3) and (3.31,-0.67) .. (0,0) .. controls (3.31,0.67) and (6.95,2.3) .. (10.93,4.9)   ;
\draw    (306.21,184.9) -- (306.5,228) ;
\draw [shift={(306.5,228)}, rotate = 89.62] [color={rgb, 255:red, 0; green, 0; blue, 0 }  ][fill={rgb, 255:red, 0; green, 0; blue, 0 }  ][line width=0.75]      (0, 0) circle [x radius= 3.35, y radius= 3.35]   ;
\draw [shift={(306.2,182.9)}, rotate = 89.62] [color={rgb, 255:red, 0; green, 0; blue, 0 }  ][line width=0.75]    (10.93,-4.9) .. controls (6.95,-2.3) and (3.31,-0.67) .. (0,0) .. controls (3.31,0.67) and (6.95,2.3) .. (10.93,4.9)   ;
\draw    (313,14) .. controls (386,6.5) and (325,117) .. (370,123.5) ;
\draw    (370,123.5) .. controls (322.5,130.5) and (385.5,238) .. (315,230.5) ;

\draw (380,111.9) node [anchor=north west][inner sep=0.75pt]    {$\omega \ \text{times}$};

\end{tikzpicture}
    \caption{Representation of the ordinal $\omega\cdot\omega$, whose ray space is homeomorphic to the convergent sequence.}
    \label{FIG_omega.omega}
\end{figure}
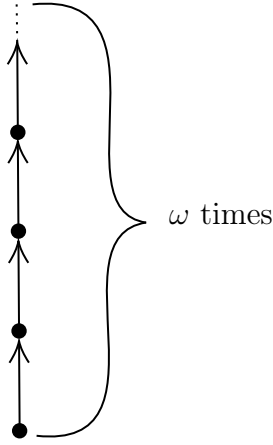

        \item[A ray with countably infinite tops] For every inning $n\in\omega$, let $\sigma$ pick $\{X\}$. At the $\omega$-th inning, set
        \[
            \sigma\seq{\{X\}:n\in\omega}=(\set{\{x_k\}:k<\omega},x_\omega).
        \]
        After $\bob$ chooses $\{x_k\}$, $\sigma$ plays the trivial partition $\{\{x_k\}\}$ at every successor inning $\omega+n+1$, for $n\in\omega$, and at the terminal inning $\omega+\omega$ it plays $(\emptyset,x_k)$. The game therefore ends at the $(\omega+\omega)$-th inning with $\ali$ winning. See \myref{FIG_RayWithCountableTops}.
\begin{figure}
    \centering

\tikzset{every picture/.style={line width=0.75pt}} 

\begin{tikzpicture}[x=0.75pt,y=0.75pt,yscale=-1,xscale=1]

\draw    (325,70) -- (325,164) ;
\draw [shift={(325,164)}, rotate = 90] [color={rgb, 255:red, 0; green, 0; blue, 0 }  ][fill={rgb, 255:red, 0; green, 0; blue, 0 }  ][line width=0.75]      (0, 0) circle [x radius= 3.35, y radius= 3.35]   ;
\draw [shift={(325,68)}, rotate = 90] [color={rgb, 255:red, 0; green, 0; blue, 0 }  ][line width=0.75]    (10.93,-4.9) .. controls (6.95,-2.3) and (3.31,-0.67) .. (0,0) .. controls (3.31,0.67) and (6.95,2.3) .. (10.93,4.9)   ;
\draw    (325,8.5) -- (325,62) ;
\draw [shift={(325,62)}, rotate = 90] [color={rgb, 255:red, 0; green, 0; blue, 0 }  ][fill={rgb, 255:red, 0; green, 0; blue, 0 }  ][line width=0.75]      (0, 0) circle [x radius= 3.35, y radius= 3.35]   ;
\draw [shift={(325,6.5)}, rotate = 90] [color={rgb, 255:red, 0; green, 0; blue, 0 }  ][line width=0.75]    (10.93,-4.9) .. controls (6.95,-2.3) and (3.31,-0.67) .. (0,0) .. controls (3.31,0.67) and (6.95,2.3) .. (10.93,4.9)   ;
\draw    (345,8) -- (345,61.5) ;
\draw [shift={(345,61.5)}, rotate = 90] [color={rgb, 255:red, 0; green, 0; blue, 0 }  ][fill={rgb, 255:red, 0; green, 0; blue, 0 }  ][line width=0.75]      (0, 0) circle [x radius= 3.35, y radius= 3.35]   ;
\draw [shift={(345,6)}, rotate = 90] [color={rgb, 255:red, 0; green, 0; blue, 0 }  ][line width=0.75]    (10.93,-4.9) .. controls (6.95,-2.3) and (3.31,-0.67) .. (0,0) .. controls (3.31,0.67) and (6.95,2.3) .. (10.93,4.9)   ;
\draw    (284,8.5) -- (284,62) ;
\draw [shift={(284,62)}, rotate = 90] [color={rgb, 255:red, 0; green, 0; blue, 0 }  ][fill={rgb, 255:red, 0; green, 0; blue, 0 }  ][line width=0.75]      (0, 0) circle [x radius= 3.35, y radius= 3.35]   ;
\draw [shift={(284,6.5)}, rotate = 90] [color={rgb, 255:red, 0; green, 0; blue, 0 }  ][line width=0.75]    (10.93,-4.9) .. controls (6.95,-2.3) and (3.31,-0.67) .. (0,0) .. controls (3.31,0.67) and (6.95,2.3) .. (10.93,4.9)   ;
\draw    (304,7.5) -- (304,61) ;
\draw [shift={(304,61)}, rotate = 90] [color={rgb, 255:red, 0; green, 0; blue, 0 }  ][fill={rgb, 255:red, 0; green, 0; blue, 0 }  ][line width=0.75]      (0, 0) circle [x radius= 3.35, y radius= 3.35]   ;
\draw [shift={(304,5.5)}, rotate = 90] [color={rgb, 255:red, 0; green, 0; blue, 0 }  ][line width=0.75]    (10.93,-4.9) .. controls (6.95,-2.3) and (3.31,-0.67) .. (0,0) .. controls (3.31,0.67) and (6.95,2.3) .. (10.93,4.9)   ;

\draw (350,25.4) node [anchor=north west][inner sep=0.75pt]    {$\cdots $};

\end{tikzpicture}
    \caption{Representation of a ray with countably infinite tops, whose ray space is homeomorphic to the convergent sequence.}
    \label{FIG_RayWithCountableTops}
\end{figure}
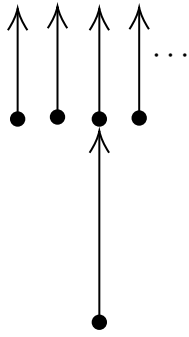
    \end{description}
\end{ex}
\section{Quasi-ray spaces}\label{SEC_QuasiRay}
The fact that $\TreeGame(X)$ characterizes through the existence of a winning strategy for $\ali$ the property of $X$ being a ray space (as seen in  \myref{COR_RaySPaceChar}) can motivate us to weaken the winning criteria for $\ali$ in order to generalize the ray space property.
To be precise:

\begin{defn}\label{DEF_QuasiTreeGame}
    Given a topological space $X$, the so-called \emph{quasi-ray game over $X$}, denoted by $\PTreeGame(X)$, is the game played between $\ali$ and $\bob$ which is recursively defined as follows: in the first inning, $\ali$ chooses an open partition $\mP_0$ for $X$, while $\bob$ responds by picking $U_0\in \mP_0$. In the $\alpha$-th inning:
    \begin{description}
        \item[if $\alpha=\beta+1$] then $\bob$ chose a $U_\beta\in \mP_\beta$ in the previous inning, so $\ali$ must now choose a clopen partition $\mP_{\beta+1}$ for $U_{\beta}$. In this case, $\bob$ must then respond with a $U_{\beta+1}\in\mP_{\beta+1}$.
        \item[if $\alpha$ is a limit ordinal] then we consider two cases:
        \begin{itemize}
            \item If there exists a pairwise disjoint collection $\mP_\alpha$ of clopen sets such that, for some $x_\alpha\in X$,
            \begin{equation} \label{EQ_QRayGameContinues1}
                \bigcap_{\beta<\alpha} U_\beta   = \left(\bigcup \mP_\alpha\right) \sqcup \{x_\alpha\},
            \end{equation}
            then the game proceeds with $\ali$ choosing one such pair $(\mP_\alpha,x_\alpha)$.
            In this case, if $\mP_\alpha\neq\emptyset$, the game continues with $\bob$ choosing a $U_\alpha\in \mP_\alpha$. Otherwise (i.e., if $\mP_\alpha=\emptyset$), this run of the game ends and $\ali$ is declared the winner.

            \item Otherwise (i.e., if no pair $(\mP_\alpha, x_\alpha)$ satisfying \myref{EQ_QRayGameContinues1} exists), this run of the game ends and $\bob$ is declared the winner.
        \end{itemize}
    \end{description}
\end{defn}

Naturally, the same arguments provided for \myref{PROP_RayGameRunLength} yield:

\begin{prop}\label{PROP_QuasiGameRunLength}
    For a space $X$, every run of the game $\PTreeGame(X)$ has length $<|X|^+$.
\end{prop}

Furthermore, we obtain the following result as a direct consequence of \myref{LEMMA_WStratClopens}:
\begin{cor}\label{COR_RayImpliesQuasi}
    If $\sigma$ is a winning strategy for $\ali$ in $\TreeGame(X)$, then $\sigma$ is a winning strategy for $\ali$ in $\PTreeGame(X)$.
\end{cor}

As with $\TreeGame(-)$, note that if $\sigma$ is a strategy for $\ali$ in $\PTreeGame(X,\tau)$, then $\sdom{\sigma}$ is a pruned subtree of $\tau^{<|X|^+}$ (ordered by inclusion).
Consequently, from \myref{LEMMA_TreeStratIso} and \myref{COR_RayImpliesQuasi}, we obtain:
\begin{cor}\label{COR_QuasiTreeStratIso}
    If $T$ is a pruned tree, then there is a winning strategy $\sigma$ for $\ali$ in $\PTreeGame(\mR(T))$ such that $\succs(\dom(\sigma))\cong T$.
\end{cor}

\begin{lemma}\label{LEMMA_QuasiStratHomeo}
    If $X$ is a space and $\sigma$ is a winning strategy for $\ali$ in $\PTreeGame(X)$, then there is a continuous bijection $f\colon X\to \mR(\sdom{\sigma})$.
\end{lemma}
\begin{proof}
    Let $g\colon \mR(\sdom{\sigma})\to X$ be defined as in the proof of \myref{LEMMA_StratHomeo}.
    It is straightforward to verify that the arguments showing that $g$ is an open bijection also apply here.
    Thus, $f=g^{-1}$ is the desired continuous bijection. 
\end{proof}

Letting $\Sigma_{\mathrm{q}}(X)$ denote the (possibly empty) collection of all winning strategies for $\ali$ in $\PTreeGame(X)$, we have:

\begin{thm}\label{THM_TreeQuasiCollection}
    For a given space $X$, the collection 
    \[\set{\sdom{\sigma}: \sigma\in \Sigma_{\mathrm{q}}(X)}\]
    comprises, up to isomorphism, the class of all pruned trees whose ray space is the continuous image of a bijection from $X$.
\end{thm}

\myref{THM_TreeQuasiCollection} motivates the following definition:

\begin{defn}\label{DEF_QuasiRay}
    We say that a topological space $(X,\tau)$ is a \emph{quasi-ray space} if there is some topology $\tau'\subseteq \tau$ such that $(X,\tau')$ is a ray space (or, equivalently, if there exists a tree $T$ and a continuous bijection $f\colon X\to \mR(T)$).
\end{defn}

From this, we obtain:

\begin{cor}\label{COR_QuasiRaySPaceChar}
    For a non-empty topological space $X$, the following are equivalent:
    \begin{itemize}
        \item[(a)] $X$ is a quasi-ray space.
        \item[(b)] $\ali$ has a winning strategy in $\PTreeGame(X)$.
    \end{itemize}
\end{cor}

Now let us apply \myref{COR_QuasiRaySPaceChar} to some examples.
We start by noting that the proof of \myref{PROP_RayGameConnectedSpace} also shows that:

\begin{prop}
    If $X$ is a connected space with $|X|\ge 2$, then $\PTreeGame(X)$ has a single run, which is won by $\bob$. In particular, such an $X$ is not a quasi-ray space.
\end{prop}

However, in contrast to \myref{EX_omega1_BobWins}:
\begin{ex}\label{EX_OrdinalQuasi}
    Given an ordinal $\gamma$ and a non-empty $X\subseteq\gamma$, consider the strategy $\sigma$ for $\ali$ in $\PTreeGame(X)$ (with the subspace topology induced by the order topology on $\gamma$), defined recursively as follows:

    First, fix $p_0=\min(X)$ and let $\sigma\seq{\,}=\{X\}$. In the $\alpha$-th inning:
    \begin{description}
        \item[If $\alpha=\eta+1$] Let $\sigma\seq{U_\beta:\beta<\alpha}=\{U_\eta\}$.
        \item[If $\alpha$ is a limit ordinal] Let $I_\alpha=\bigcap_{\beta<\alpha}U_\beta$. If $I_\alpha\setminus\{p_0\}\neq\emptyset$, set
        \[
            x_\alpha=\min(I_\alpha\setminus\{p_0\}),\qquad
            \mP_\alpha=\{I_\alpha\setminus\{x_\alpha\}\},\qquad
            \sigma\seq{U_\beta:\beta<\alpha}=(\mP_\alpha,x_\alpha).
        \]
        If $I_\alpha\setminus\{p_0\}=\emptyset$, set
        \[
            \sigma\seq{U_\beta:\beta<\alpha}=(\emptyset,p_0).
        \]
    \end{description}
    
    Since every partition played by $\sigma$ is a singleton, there is a unique run $R$ compatible with $\sigma$. Furthermore, $p_0$ is the unique element in the intersection of $\bob$'s moves in $R$, meaning that $\sigma$ is indeed a winning strategy for $\ali$.  

    Hence, in view of \myref{COR_QuasiRaySPaceChar}, we conclude that any subset $X$ of an ordinal $\gamma$, equipped with the subspace topology induced by the order topology on $\gamma$, is a quasi-ray space.
\end{ex}

\begin{ex}\label{EX_SorgenfreyQuasiRay}
    Let $\RR_l$ denote the Sorgenfrey line and consider the following strategy $\sigma$ for $\ali$ in $\PTreeGame(\RR_l)$: in the first inning, we let 
    \[\sigma\seq{\,} = \mP_0 = \set{[k,k+1):k\in \ZZ}.\]
    In the $(n+1)$-th inning, provided that $\seq{[a_i,b_i):i\le n}$ is the sequence of intervals chosen by $\bob$ against $\sigma$, we let
    \begin{align*}
        \sigma\seq{[a_i,b_i):i\le n} &= \mP_{n+1} \\
        &=  \set{\left[a_n, a_n + \frac{b_n-a_n}{2}\right), \left[a_n + \frac{b_n-a_n}{2}, a_n + \frac{3(b_n-a_n)}{4}\right), \dots}\\
        &= \set{\left[a_n+\sum_{i=1}^k\frac{b_n-a_n}{2^i}, a_n+\sum_{i=1}^{k+1}\frac{b_n-a_n}{2^i}\right):k\in\omega}.
    \end{align*}
     At the $\omega$-th inning,  suppose $\seq{[a_n,b_n):n\in\omega}$ is a sequence of intervals chosen by $\bob$ in a run compatible with $\sigma$. By the construction of $\sigma$, $[a_{n+1},b_{n+1}] \subset [a_n,b_n)$, so 
    \[
        \bigcap_{n\in\omega}[a_n,b_n) = \bigcap_{n\in\omega}[a_n,b_n] \neq \emptyset.
    \] and $\diam([a_n,b_n))$ converges to $0$. Then, for some $x$ we have that $\{x\} =\bigcap_{n\in\omega}[a_n,b_n)$. We set
    \[
        \sigma\seq{[a_n,b_n):n\in\omega}=(\emptyset,x).
    \]

    It follows that $\sigma$ is a winning strategy for $\ali$. 
    \myref{COR_QuasiRaySPaceChar} then implies that $\RR_l$ is a quasi-ray space, despite not being a ray space (see \myref{EX_SorgenfreyNotRay}). By \myref{THM_TreeQuasiCollection}, there exists a topology $\tau$ on $\RR$ weaker than that of $\RR_l$ such that 
    \[(\RR,\tau)\simeq \mR(\sdom{\sigma}) \simeq \omega^\omega.\]
    
    A closer examination of \myref{LEMMA_QuasiStratHomeo} reveals that $\tau$ is the topology generated by the collection of all intervals chosen by $\sigma$. Since the diameter of these intervals containing any given $x\in \RR$ converges to $0$, $\tau$ is strictly stronger than the usual Euclidean topology on $\RR$.
    
\end{ex}

From \myref{THM_ScatteredCharacterization}, we can identify a class of spaces that are not quasi-ray spaces:

\begin{ex}\label{EX_NoIsolatedLessThanContinuum}
    Suppose $(X,\tau)$ is a non-empty space with $|X|<2^{\aleph_0}$ and no isolated points (e.g., a countable dense subspace of $2^\omega$). For every $\tau'\subseteq \tau$, $(X,\tau')$ also has no isolated points. Thus, \myref{THM_ScatteredCharacterization} implies that $X$ cannot be a quasi-ray space.
\end{ex}

For the countable case, though, we can strengthen the conclusion of \myref{EX_NoIsolatedLessThanContinuum}:

\begin{prop}\label{PROP_CountableWithNoIsolated}
    If $X$ is a countable space with no isolated points, then $\bob$ has a winning strategy in $\PTreeGame(X)$.
\end{prop}
\begin{proof}
    Consider the following strategy for $\bob$ in $\PTreeGame(X)$: fix an enumeration $X = \set{x_n:n\in\omega}$. In the first inning, let $\sigma$ pick a clopen set $U_0\in \mP_0$ such that $x_0 \notin U_0$, if such a set exists; otherwise, $\mP_0$ must be a singleton, and $\sigma$ chooses $X$.

    Suppose $\seq{U_\beta:\beta<\alpha}$ is the sequence of sets chosen by $\sigma$ up to the $\alpha$-th inning.
    
    \begin{description}
        \item[If $\alpha=\gamma+1$] Let $n_\alpha\in\omega$ be the least index such that $x_{n_\alpha}\in U_\gamma$. Let $\sigma$ pick $U_\alpha\in \mP_\alpha$ such that $x_{n_\alpha} \notin U_\alpha$, if possible; otherwise, $\sigma$ chooses the only available set.

        \item[If $\alpha$ is a limit ordinal] Provided $\ali$ has just chosen $(\mP_\alpha, p_\alpha)$, with  $p_\alpha \in \bigcap_{\beta<\alpha}U_\beta$ and $\mP_\alpha\neq\emptyset$, let $n_\alpha\in\omega$ be the least index such that 
        \[
            p_\alpha \neq x_{n_\alpha} \in \bigcap_{\beta<\alpha}U_\beta.
        \]
        Let $\sigma$ pick $U_\alpha\in \mP_\alpha$ such that $x_{n_\alpha} \notin U_\alpha$, if possible (or the unique $U_\alpha\in \mP_\alpha$, otherwise).
    \end{description}
    
    We claim that $\sigma$ is a winning strategy for $\bob$. Striving for a contradiction, suppose 
    \[
        R = \seq{\mP_0, U_0, \dotsc, (\mP_\gamma, p_\gamma), U_\gamma, \dotsc , (\emptyset, p_\alpha)}
    \] 
    is a run played according to $\sigma$ that ends at the (limit) $\alpha$-th inning, so that
    \[
        \bigcap_{\beta<\alpha}U_\beta = \{p_\alpha\} = \{x_N\}
    \]
    for some $N\in\omega$. 
    Let $\gamma<\alpha$ be such that $N = n_\gamma$. 
    This implies $n_\gamma = n_{\gamma+k}$ for all $k \in \omega$, which only occurs if $U_\gamma = U_{\gamma+k}$ (by construction of $\sigma)$. 
    In this case, 
    \[
        \bigcap_{\beta<\gamma+\omega}U_\beta = U_\gamma.
    \]
    However, since $X$ has no isolated points, $U_\gamma\setminus \{p_{\gamma+\omega}\}$ is not clopen. This implies that $\alpha\neq\gamma+\omega$ (because $U_\gamma$ cannot be a singleton) and $\mP_{\gamma+\omega} \neq \{U_\gamma \setminus \{p_{\gamma+\omega}\}\}$, allowing $\sigma$ to have chosen $U_{\gamma+\omega}$ such that $x_{n_\gamma} \notin U_{\gamma+\omega}$, a contradiction.
\end{proof}
\section{A brief study of ray and end subspaces}\label{SEC_RaySubspaces}

We shall now apply the characterizations for ray and quasi-ray spaces obtained in Corollaries \ref{COR_RaySPaceChar}, \ref{COR_EndSpaceChar} and \ref{COR_QuasiRaySPaceChar} to determine whether a few classes of subspaces of ray (or end) spaces are themselves ray (or end) spaces. 
Let us begin by noting that not every subspace of a ray space is even a quasi-ray space:
\begin{ex}\label{EX_CountableDenseCantor}
    Consider a countable $X\subset 2^{\omega}$ that is dense in $2^\omega$. In view of \myref{EX_NoIsolatedLessThanContinuum}, $X$ is not a quasi-ray space (in fact, \myref{PROP_CountableWithNoIsolated} further implies that $\bob$ has a winning strategy in $\PTreeGame(X)$).
\end{ex}

On the other hand, although it is known \cite{approximating} that open subspaces of end spaces comprise themselves end spaces, the same cannot be said for ray spaces. Indeed, consider the open subspace $\omega_1 \subset \omega_1+1$. Then:

\begin{cor}
    $\omega_1$ is not a ray space, despite being a quasi-ray space.
\end{cor}
\begin{proof}
    While \myref{EX_OrdinalQuasi} and \myref{COR_QuasiRaySPaceChar} show that $\omega_1$ is a quasi-ray space, \myref{EX_omega1_BobWins} and \myref{COR_RaySPaceChar} imply that $\omega_1$ is not a ray space.
\end{proof}

Thus, it is only natural to ask:

\begin{question}\label{QUE_AreOpenSubspaceofRayQuasi?}
    If $X$ is an open subspace of a ray space, must it be a quasi-ray space?
\end{question}

Regarding closed subspaces, it has been previously shown in the literature that:

\begin{prop}[Corollary 3.7 in \cite{pitz2023characterisingpathraybranch}]\label{PROP_ClosedRay}
    If $X$ is a closed subspace of a ray space, then $X$ is a ray space.
\end{prop}

One might wonder, in light of \myref{QUE_AreOpenSubspaceofRayQuasi?} and \myref{PROP_ClosedRay}, whether open or closed subspaces of quasi-ray spaces are themselves quasi-ray spaces.
However, the following example answers this question negatively, even for clopen subspaces:

\begin{ex}
    Let $X\subseteq 2^\omega$ be countable and dense, as in \myref{EX_CountableDenseCantor}. Let $\tau$ be the topology generated by the usual topology of $2^{\omega}$ together with $X$ and $2^\omega\setminus X$ as basic open sets (so that the usual topology over $2^\omega$ attests that $(2^\omega,\tau)$ is a quasi-ray space). 
    
    Note that $X$ is clopen in $(2^\omega,\tau)$. 
    On the other hand, the topology on $X$ as a subspace of $(2^\omega,\tau)$ coincides with the subspace topology induced by the usual topology of $2^\omega$.
    Hence, by \myref{EX_CountableDenseCantor}, as a subspace of $(2^\omega,\tau)$, $X$ is a clopen subspace of a quasi-ray space that is not itself a quasi-ray space.
\end{ex}

The branch space of a tree $T$ (i.e., its set of maximal rays imbued with the topology generated by the basis $\set{[t]:t\in T}$), denoted by $\mB(T)$, is also a notable subspace of its ray space. In this context:

\begin{prop}\label{PROP_BranchQuasiRay}
    If $X$ is a branch space (i.e., there exists a tree $T$ with $X\simeq \mB(T)$), then $X$ is a quasi-ray space.
\end{prop}
\begin{proof}
    Without loss of generality, let $T$ be a pruned tree such that $X = \mB(T)$ (where $\mB(T)$ denotes the branch space of $T$).

    We recursively define a winning strategy for $\ali$ in $\PTreeGame(X)$ as follows: in the first inning, let 
    \[
        \sigma\seq{\,} = \set{[t]: t\in T \text{ is a successor of the root } r\in T}.
    \]

    Suppose the game has been played up to the $\alpha$-th inning, where $\seq{[t_\beta]:\beta<\alpha}$ are the clopen sets chosen by $\bob$ such that $\seq{t_\beta:\beta<\alpha}$ is strictly increasing in $T$.

    \begin{description}
        \item[If $\alpha = \gamma+1$] We let
        \[
            \sigma\seq{[t_\beta]:\beta < \alpha} = \sigma\seq{[t_\beta]:\beta \le \gamma} = \set{[t]: t\in T \text{ is a successor of } t_\gamma\in T}.
        \]

        \item[If $\alpha$ is a limit ordinal] If $R = \{r\}\cup\set{t_\beta:\beta<\alpha}\in \mB(T)$, the run is finished and 
        \[
            \bigcap_{\beta<\alpha}[t_\beta] = \{R\},
        \] 
        so $\ali$ wins.
        Otherwise, fix any branch $B\supseteq R$ and let 
        \[
            A = \{t\in T\setminus B: \text{$t$ is above $R$ and no $s<t$ is in $T\setminus B$}\} 
        \]
        (note that $t\in A$ if and only if $t\in T\setminus B$ is above $R$ and is either a top of a ray contained in $B$ or a successor of an element of $B$, as illustrated in \myref{FIG_BranchQuasiRay}).
\begin{figure}
    \centering

\tikzset{every picture/.style={line width=0.75pt}} 

\begin{tikzpicture}[x=0.75pt,y=0.75pt,yscale=-1,xscale=1]

\draw  [dash pattern={on 0.84pt off 2.51pt}] (321.51,277) -- (181.27,17) -- (460.92,16.57) -- cycle ;
\draw [color={rgb, 255:red, 0; green, 0; blue, 0 }  ,draw opacity=1 ]   (281.19,132.71) .. controls (287.91,220.42) and (332.89,143.35) .. (321.51,277) ;
\draw [shift={(281,130)}, rotate = 86.27] [color={rgb, 255:red, 0; green, 0; blue, 0 }  ,draw opacity=1 ][line width=0.75]    (10.93,-4.9) .. controls (6.95,-2.3) and (3.31,-0.67) .. (0,0) .. controls (3.31,0.67) and (6.95,2.3) .. (10.93,4.9)   ;
\draw [color={rgb, 255:red, 208; green, 2; blue, 27 }  ,draw opacity=1 ]   (255.55,72.8) -- (263,88) ;
\draw [shift={(254.67,71)}, rotate = 63.89] [color={rgb, 255:red, 208; green, 2; blue, 27 }  ,draw opacity=1 ][line width=0.75]    (8.74,-3.92) .. controls (5.56,-1.84) and (2.65,-0.53) .. (0,0) .. controls (2.65,0.53) and (5.56,1.84) .. (8.74,3.92)   ;
\draw [color={rgb, 255:red, 208; green, 2; blue, 27 }  ,draw opacity=1 ]   (263,88) -- (267.5,97.5) ;
\draw [color={rgb, 255:red, 208; green, 2; blue, 27 }  ,draw opacity=1 ]   (267.5,97.5) -- (272,107) ;
\draw [shift={(272,107)}, rotate = 64.65] [color={rgb, 255:red, 208; green, 2; blue, 27 }  ,draw opacity=1 ][fill={rgb, 255:red, 208; green, 2; blue, 27 }  ,fill opacity=1 ][line width=0.75]      (0, 0) circle [x radius= 2.68, y radius= 2.68]   ;
\draw [color={rgb, 255:red, 208; green, 2; blue, 27 }  ,draw opacity=1 ]   (272,107) -- (276.5,116.5) ;
\draw [shift={(276.5,116.5)}, rotate = 64.65] [color={rgb, 255:red, 208; green, 2; blue, 27 }  ,draw opacity=1 ][fill={rgb, 255:red, 208; green, 2; blue, 27 }  ,fill opacity=1 ][line width=0.75]      (0, 0) circle [x radius= 2.68, y radius= 2.68]   ;
\draw [color={rgb, 255:red, 208; green, 2; blue, 27 }  ,draw opacity=1 ]   (276.5,116.5) -- (281,126) ;
\draw [shift={(281,126)}, rotate = 64.65] [color={rgb, 255:red, 208; green, 2; blue, 27 }  ,draw opacity=1 ][fill={rgb, 255:red, 208; green, 2; blue, 27 }  ,fill opacity=1 ][line width=0.75]      (0, 0) circle [x radius= 2.68, y radius= 2.68]   ;
\draw [color={rgb, 255:red, 208; green, 2; blue, 27 }  ,draw opacity=1 ]   (214.97,22.18) -- (223.85,32.45) ;
\draw [shift={(213.67,20.67)}, rotate = 49.15] [color={rgb, 255:red, 208; green, 2; blue, 27 }  ,draw opacity=1 ][line width=0.75]    (8.74,-3.92) .. controls (5.56,-1.84) and (2.65,-0.53) .. (0,0) .. controls (2.65,0.53) and (5.56,1.84) .. (8.74,3.92)   ;
\draw [color={rgb, 255:red, 208; green, 2; blue, 27 }  ,draw opacity=1 ]   (223.85,32.45) -- (230.56,40.54) ;
\draw [color={rgb, 255:red, 208; green, 2; blue, 27 }  ,draw opacity=1 ]   (230.56,40.54) -- (237.27,48.63) ;
\draw [color={rgb, 255:red, 208; green, 2; blue, 27 }  ,draw opacity=1 ]   (237.27,48.63) -- (243.98,56.72) ;
\draw [shift={(243.98,56.72)}, rotate = 50.33] [color={rgb, 255:red, 208; green, 2; blue, 27 }  ,draw opacity=1 ][fill={rgb, 255:red, 208; green, 2; blue, 27 }  ,fill opacity=1 ][line width=0.75]      (0, 0) circle [x radius= 2.68, y radius= 2.68]   ;
\draw [color={rgb, 255:red, 208; green, 2; blue, 27 }  ,draw opacity=1 ]   (243.98,56.72) -- (250.69,64.81) ;
\draw [shift={(250.69,64.81)}, rotate = 50.33] [color={rgb, 255:red, 208; green, 2; blue, 27 }  ,draw opacity=1 ][fill={rgb, 255:red, 208; green, 2; blue, 27 }  ,fill opacity=1 ][line width=0.75]      (0, 0) circle [x radius= 2.68, y radius= 2.68]   ;
\draw [color={rgb, 255:red, 0; green, 0; blue, 255 }  ,draw opacity=1 ]   (281,126) -- (296,116) ;
\draw [shift={(296,116)}, rotate = 326.31] [color={rgb, 255:red, 0; green, 0; blue, 255 }  ,draw opacity=1 ][fill={rgb, 255:red, 0; green, 0; blue, 255 }  ,fill opacity=1 ][line width=0.75]      (0, 0) circle [x radius= 2.68, y radius= 2.68]   ;
\draw [color={rgb, 255:red, 0; green, 0; blue, 255 }  ,draw opacity=1 ]   (276.5,116.5) -- (289,102.5) ;
\draw [shift={(289,102.5)}, rotate = 311.76] [color={rgb, 255:red, 0; green, 0; blue, 255 }  ,draw opacity=1 ][fill={rgb, 255:red, 0; green, 0; blue, 255 }  ,fill opacity=1 ][line width=0.75]      (0, 0) circle [x radius= 2.68, y radius= 2.68]   ;
\draw [color={rgb, 255:red, 0; green, 0; blue, 255 }  ,draw opacity=1 ]   (272,107) -- (276.5,87.5) ;
\draw [shift={(276.5,87.5)}, rotate = 282.99] [color={rgb, 255:red, 0; green, 0; blue, 255 }  ,draw opacity=1 ][fill={rgb, 255:red, 0; green, 0; blue, 255 }  ,fill opacity=1 ][line width=0.75]      (0, 0) circle [x radius= 2.68, y radius= 2.68]   ;
\draw [color={rgb, 255:red, 0; green, 0; blue, 255 }  ,draw opacity=1 ]   (272,107) -- (285,92.5) ;
\draw [shift={(285,92.5)}, rotate = 311.88] [color={rgb, 255:red, 0; green, 0; blue, 255 }  ,draw opacity=1 ][fill={rgb, 255:red, 0; green, 0; blue, 255 }  ,fill opacity=1 ][line width=0.75]      (0, 0) circle [x radius= 2.68, y radius= 2.68]   ;
\draw [color={rgb, 255:red, 0; green, 0; blue, 255 }  ,draw opacity=1 ]   (290.5,126) ;
\draw [shift={(290.5,126)}, rotate = 0] [color={rgb, 255:red, 0; green, 0; blue, 255 }  ,draw opacity=1 ][fill={rgb, 255:red, 0; green, 0; blue, 255 }  ,fill opacity=1 ][line width=0.75]      (0, 0) circle [x radius= 2.68, y radius= 2.68]   ;
\draw [color={rgb, 255:red, 0; green, 0; blue, 255 }  ,draw opacity=1 ]   (300.5,129) -- (300,126) ;
\draw [shift={(300,126)}, rotate = 260.54] [color={rgb, 255:red, 0; green, 0; blue, 255 }  ,draw opacity=1 ][fill={rgb, 255:red, 0; green, 0; blue, 255 }  ,fill opacity=1 ][line width=0.75]      (0, 0) circle [x radius= 2.68, y radius= 2.68]   ;
\draw [color={rgb, 255:red, 0; green, 0; blue, 255 }  ,draw opacity=1 ]   (261.5,64) ;
\draw [shift={(261.5,64)}, rotate = 0] [color={rgb, 255:red, 0; green, 0; blue, 255 }  ,draw opacity=1 ][fill={rgb, 255:red, 0; green, 0; blue, 255 }  ,fill opacity=1 ][line width=0.75]      (0, 0) circle [x radius= 2.68, y radius= 2.68]   ;
\draw [color={rgb, 255:red, 0; green, 0; blue, 255 }  ,draw opacity=1 ]   (271.17,66.67) -- (270.67,63.67) ;
\draw [shift={(270.67,63.67)}, rotate = 260.54] [color={rgb, 255:red, 0; green, 0; blue, 255 }  ,draw opacity=1 ][fill={rgb, 255:red, 0; green, 0; blue, 255 }  ,fill opacity=1 ][line width=0.75]      (0, 0) circle [x radius= 2.68, y radius= 2.68]   ;
\draw [color={rgb, 255:red, 0; green, 0; blue, 255 }  ,draw opacity=1 ]   (243.98,56.72) -- (256.48,42.72) ;
\draw [shift={(256.48,42.72)}, rotate = 311.76] [color={rgb, 255:red, 0; green, 0; blue, 255 }  ,draw opacity=1 ][fill={rgb, 255:red, 0; green, 0; blue, 255 }  ,fill opacity=1 ][line width=0.75]      (0, 0) circle [x radius= 2.68, y radius= 2.68]   ;
\draw [color={rgb, 255:red, 0; green, 0; blue, 255 }  ,draw opacity=1 ]   (243.98,56.72) -- (247,41) ;
\draw [shift={(247,41)}, rotate = 280.87] [color={rgb, 255:red, 0; green, 0; blue, 255 }  ,draw opacity=1 ][fill={rgb, 255:red, 0; green, 0; blue, 255 }  ,fill opacity=1 ][line width=0.75]      (0, 0) circle [x radius= 2.68, y radius= 2.68]   ;
\draw  [color={rgb, 255:red, 0; green, 0; blue, 255 }  ,draw opacity=1 ][dash pattern={on 0.84pt off 2.51pt}] (243.67,46) .. controls (224,19) and (264.33,32.67) .. (269,38.33) .. controls (273.67,44) and (282.33,71.67) .. (302.33,101.67) .. controls (322.33,131.67) and (285.67,144) .. (286.33,133) .. controls (287,122) and (263.33,73) .. (243.67,46) -- cycle ;

\draw (426,27.51) node [anchor=north west][inner sep=0.75pt]    {$T$};
\draw (199.67,25.07) node [anchor=north west][inner sep=0.75pt]  [color={rgb, 255:red, 208; green, 2; blue, 27 }  ,opacity=1 ]  {$B$};
\draw (302,51.07) node [anchor=north west][inner sep=0.75pt]  [color={rgb, 255:red, 0; green, 0; blue, 255 }  ,opacity=1 ]  {$A$};
\draw (298,136.73) node [anchor=north west][inner sep=0.75pt]    {$R$};

\end{tikzpicture}
    \caption{After fixing a branch $B$ above the ray $R$ played thus far (in the proof of \myref{PROP_BranchQuasiRay}), we consider the clopen sets of the form $[t]$ for each $t$ highlighted in blue, so that we obtain a clopen partition for the set of branches above $R$ which are not $B$.}
    \label{FIG_BranchQuasiRay}
\end{figure}
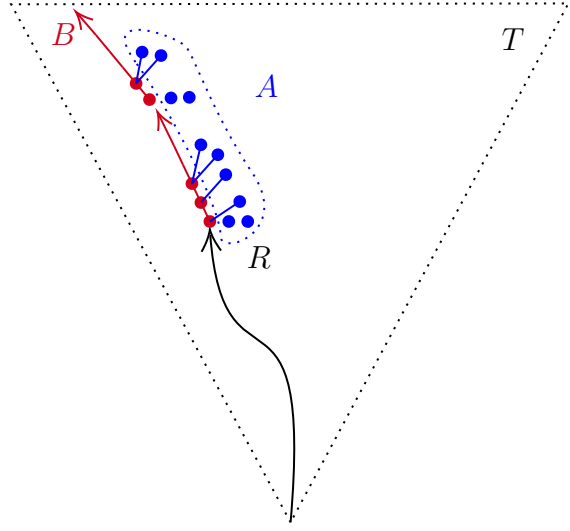
        
         Then 
         \[
            \bigcap_{\beta<\alpha}[t_\beta] = \{B\}\sqcup\bigcup_{t\in A}[t],
        \]
        so we may let  
        \[  
            \sigma\seq{[t_\beta]:\beta<\alpha} = (\set{[t]:t\in A}, B).
        \]
    \end{description}
    It is then clear that $\sigma$ is a winning strategy, since it has been constructed in a way that $\ali$ is able to keep playing for as long as she has not won yet.
\end{proof}

On the other hand:
\begin{prop}\label{PROP_2omega1NotRay}
    If $X$ is a space with no $G_\delta$ points (i.e., no singleton $\{x\}$ is the intersection of countably many open sets in $X$) and there is a run of the game $\TreeGame(X)$ which is won by $\ali$, then $X$ contains a non-trivial convergent sequence.
\end{prop}
\begin{proof}
    Suppose 
    \[R = \seq{\mP_0, U_0, \dotsc, (\mP_\omega, x_\omega), U_\omega, \dotsc, (\emptyset,x_\alpha) }\]
    is a run of $\TreeGame(X)$ won by $\ali$. 

    Since $X$ does not contain any $G_\delta$ points, it follows from the fact that
    \[
        \mB = \set{U_n\setminus \bigcup\mF: n<\omega \text{ and finite } \mF \subset \mP_\omega}
    \]
    is a local basis for $x_\omega$ (see \myref{EQ_RayGameContinues2}) that $\mP_\omega$ must be uncountable.

    In this case, if we fix an infinite subset $\set{W_n:n\in\omega} \subset \mP_\omega$ and points $x_n \in W_n$ for each $n\in\omega$, then the fact that $\mB$ is a local basis for $x_\omega$ implies that $\seq{x_n:n\in\omega}$ is a non-trivial sequence converging to $x_\omega$.
\end{proof}

Thus:
\begin{cor}
    Let $2^{<\omega_1}$ denote the complete binary tree of height $\omega_1$. Then every run of the game $\TreeGame(\mB(2^{<\omega_1}))$ is won by $\bob$.
    
    In particular, $\mB(2^{<\omega_1})$ is not a ray space, despite being a quasi-ray space.
\end{cor}
\begin{proof}
    This is a direct consequence of \myref{PROP_2omega1NotRay}, as $2^{\omega_1}$ contains no non-trivial convergent sequence and, for every point of $\mB(2^{<\omega_1})$, the intersection of countably many open neighborhoods of that point contains an uncountable basic open neighborhood.
\end{proof}

Recall that the \emph{Cantor-Bendixson derivative} of a space $X$ is its (closed) subspace
\[
    \partial(X) = X\setminus\set{x\in X: x \text{ is isolated}}.
\]
We can then recursively define, for an arbitrary ordinal $\alpha$, the (closed) subspace
\[
    \partial^\alpha(X) = \begin{cases}
        \partial(\partial^\beta(X)), &\text{ if $\alpha = \beta+1$},\\
        \bigcap_{\beta<\alpha}\partial^\beta(X)&\text{ if $\alpha$ is a limit ordinal}.
    \end{cases}
\]

The so-called \emph{Cantor-Bendixson rank of $X$} is then the least ordinal $\alpha$ for which $\partial^\alpha(X) = \partial^{\alpha+1}(X)$ (i.e., for which $\partial^\alpha(X)$ has no isolated point).

Note that if $X$ is scattered, then its Cantor-Bendixson rank $\alpha$ is such that $\alpha\ge 1$ and $\partial^\alpha(X) = \emptyset$.

Before we move on to the main result of this section, consider the following result from the literature:

\begin{prop}[Corollary 5.2 in \cite{approximating}]\label{PROP_EndSpacesHeredUltrapara}
    For every $X\subseteq \Omega(G)$ and every open cover $\mU$ of $X$, there exists an open partition $\mP$ of $X$ that refines $\mU$.
\end{prop}

We now have the necessary tools to show:

\begin{thm}\label{THM_ScatteredIsEnd}
    Let $T$ be a special tree and suppose $X\subseteq \mR(T)$ is scattered. Then $X$ is an end space. 
\end{thm}
\begin{proof}
    We proceed by induction on the Cantor-Bendixson rank of $X$.
    First, if the rank of $X$ is $1$, then $X$ is discrete and the result is immediate.
    Assume the result holds whenever the rank of $X$ is $\beta$, for all $\beta < \alpha$.
    \begin{description}
        \item[If $\alpha = \gamma+1$] Let $X_\gamma = X\setminus \partial^\gamma(X)$. 

        Since
        \[
            \partial^{\gamma+1}(X) = \partial(\partial^\gamma(X)) = \emptyset,
        \]
        $\partial^{\gamma}(X)$ is discrete. 
        So for each $x\in \partial^{\gamma}(X)$ there is an open set $U_x$ such that $U_x \cap \partial^{\gamma}(X) = \{x\}$.
        Moreover, as $\partial^{\gamma}(X)$ is closed, $\mU = \set{U_x : x\in \partial^{\gamma}(X)} \cup \{X_\gamma\}$ is an open cover for $X$.
        By \myref{THM_PitzRepresentation} and \myref{PROP_EndSpacesHeredUltrapara}, there exists an open partition $\mP$ refining $\mU$.

        We define a strategy $\sigma$ for $\ali$ in $\TreeGame(X)$ as follows: Let $\sigma\seq{\,} = \mP$. 
        
        If $\bob$ responds with $U\in \mP$ such that $U \cap \partial^{\gamma}(X) = \emptyset$, then $U \subseteq X_\gamma$. 
        By the induction hypothesis and \myref{COR_EndSpaceChar}, there exists a winning strategy $\sigma_U$ for $\ali$ in $\TreeGame(U)$ such that $\sdom{\sigma_U}$ is special. 
        We then let $\sigma$ follow $\sigma_U$ from the second inning onwards. 

        If $\bob$ chooses $U\in \mP$ such that $U \cap \partial^{\gamma}(X) \neq \emptyset$, then by the definition of $\mP$, there is a unique $x \in \partial^{\gamma}(X)$ such that $U \cap \partial^{\gamma}(X) = \{x\}$.
        In this case, let $\sigma$ respond with $\{U\}$ in every inning $n < \omega$, so that, at inning $\omega$, the intersection of $\bob$'s choices is $U$ itself. 

        Since $x \in \mR(T)$ is a countable ray, let us fix a strictly increasing sequence $\set{t_n:n\in\omega}$ in $T$ which is cofinal in $x$ and such that $t_0$ is the root of $T$. 
        Now consider
        \[
            \mP_x = \left(\set{[t_{n},\{t_{n+1}\}]\cap U: n\in\omega}\cup\{[t,\emptyset]\cap U: t\in \Rtop(x)\}\right)\setminus\{\emptyset\}.
        \]
        Then $U = \left(\bigcup\mP_x\right)\sqcup\{x\}$.
        Furthermore, if $t\in x$ and $F\subseteq\Rtop(x)$ is finite, then we can find an $N\in\omega$ such that $t_n\ge t$ for every $n\ge N$.
        In this case,
        \[
            x\in U\setminus \left(\left(\bigcup_{n\le N}[t_{n},\{t_{n+1}\}]\right)\cup\left(\bigcup_{t\in F}[t,\emptyset]\right)\right)\subseteq [t,F],
        \]
        so 
        \[
            \set{U\setminus \bigcup\mF: \text{finite }\mF\subseteq \mP_x }
        \]
        is a local basis at $x$ and we may let $\sigma\seq{U : n < \omega} = (\mP_x, x)$.
        
        For every $V \in \mP_x$, we have $V \subset X_\gamma$ (since $V\subset U$ and $U\cap\partial^\gamma(X) = \{x\}$). 
        Then the induction hypothesis and \myref{COR_EndSpaceChar} tell us that there is a winning strategy $\sigma_V$ for $\ali$ in $\TreeGame(V)$ with special $\sdom{\sigma_V}$. 
        Let $\sigma$ follow $\sigma_V$ after $\bob$'s response $V$ at this $\omega$th inning.
        Then it is clear that the defined $\sigma$ is a winning strategy for $\ali$ in $\TreeGame(X)$ with special $\sdom{\sigma}$.  We thus reach the desired conclusion with \myref{COR_EndSpaceChar}.
        
        \item[If $\alpha$ is a limit ordinal] Then $ \bigcap_{\beta < \alpha} \partial^{\beta}(X)=\emptyset$. 
        For each $\gamma < \alpha$, let $X_\gamma = X\setminus  \partial^{\gamma}(X)$. 
        Since $\mU = \set{X_\gamma : \gamma < \alpha}$ is an open cover for $X$, we can use \myref{THM_PitzRepresentation} and \myref{PROP_EndSpacesHeredUltrapara} to once again find an open partition $\mP$ refining $\mU$.

        Define $\sigma$ with $\sigma\seq{\,} = \mP$. 
        If $\bob$ responds with $U \in \mP$, then $U \subseteq X_\gamma = X\setminus \partial^\gamma(X)$ for some $\gamma < \alpha$.
        By the induction hypothesis, there is a winning strategy $\sigma_U$ for $\ali$ in $\TreeGame(U)$ with special $\sdom{\sigma_U}$. Letting $\sigma$ follow $\sigma_U$ from this point onward defines a winning strategy for $\ali$ in $\TreeGame(X)$ such that $\sdom{\sigma}$ is special, so the conclusion follows from \myref{COR_EndSpaceChar}.
    \end{description}
\end{proof}

\begin{cor}\label{THM_CountSpecRaySubSpaceChar}
    Let $G$ be a graph and suppose $X \subseteq \Omega(G)$ satisfies $|X| < 2^{\aleph_0}$. Then the following are equivalent:
    \begin{enumerate}[label=(\alph*)]
        \item $X$ is scattered.
        \item $X$ is an end space.
    \end{enumerate}
\end{cor}
\begin{proof}
    Theorems \ref{THM_PitzRepresentation} and \ref{THM_ScatteredIsEnd} show that (a) implies (b). The reverse implication follows from Theorems \ref{THM_PitzRepresentation} and \ref{THM_ScatteredCharacterization}.
\end{proof}

It was shown in \cite{endspacesandtreedecompositions} that if $G$ is a graph and $X \subseteq \Omega(G)$ is closed, then there exists an induced subgraph $H \subseteq G$ such that $X \simeq \Omega(H)$. In light of \myref{THM_ScatteredIsEnd}, it is natural to ask:

\begin{question}
    Suppose $G$ is a graph and $X \subseteq \Omega(G)$ is scattered. Is there an induced subgraph $H \subseteq G$ such that $X \simeq \Omega(H)$? If not, is there an $H \subset G$ (not necessarily induced) satisfying this property?
\end{question}
\section{A brief study of products of ray and end spaces}\label{SEC_RayProducts}
An example of a pair of end spaces whose product is not itself an end space was shown in \cite{PauloGustavoDavide} with the help of a game-theoretic characterization.
We shall now generalize that example with the help of our characterizations in Corollaries \ref{COR_RaySPaceChar} and \ref{COR_QuasiRaySPaceChar} and, with these generalizations, we conjecture a possible characterization for the classes of ray and end spaces which are stable under finite products (see \myref{QUESTION_FirsCountRayProd0}). 

Firstly, let $A_{\kappa}$ denote the 1-point compactification of the discrete space of cardinality $\kappa$. Then:

\begin{thm}\label{THM_ConvSeq&UncountTopsNotProd}
    Suppose that $X$ contains a non-trivial convergent sequence and that $Y$ contains a copy of $A_{\aleph_1}$. Then $\bob$ has a winning strategy in $\PTreeGame(X\times Y)$.
\end{thm}
\begin{proof}
    Let us first fix an $x_\infty\in X$ which is the limit of a non-trivial convergent sequence $\seq{x_n:n\in\omega}$ in $X$, as well as an uncountable $D\subset Y$ with a $y_\infty\in Y$ for which every open set $V\subset Y$ with $y_\infty\in V$ is such that $D\setminus V$ is finite.  
    
    We define $\bob$'s winning strategy $\sigma$ as follows: from the beginning of the game, let $\sigma$ choose the open sets in $\ali$'s partitions containing $(x_\infty,y_\infty)$, up to the (limit) inning $\alpha_\infty$, at which no further open set in $\ali$'s partitions contains it (either because the game ended at this inning $\alpha_\infty$, or because $\ali$ chose $p_{\alpha_\infty} = (x_\infty,y_\infty)$).  Afterwards, if the game is not finished at the inning $\alpha_\infty$, we may let $\sigma$ pick arbitrary clopen sets.
    
    In order to show that this $\sigma$ is a winning strategy, suppose 
    \[
        R = \seq{\mP_0, W_0, \dotsc, (\mP_\gamma, p_\gamma), W_\gamma,  \dotsc},
    \]
    is a run compatible with $\sigma$ (so that its length $\alpha$ is greater than or equal to $\alpha_\infty$). 
    Let $W = \bigcap_{\beta<\alpha_\infty}W_\beta$.
    
    Note that, for each $\beta<\alpha_\infty$, $(\{x_\infty\}\times D)\setminus W_\beta = F_\beta$ is finite. 
    Then $\seq{F_\beta:\beta<\alpha_\infty}$ is an increasing sequence of finite subsets of $\{x_\infty\}\times D$, so that $\bigcup_{\beta<\alpha_\infty}F_\beta$ is, at most, countable. 
    Since $D$ is uncountable, it follows that 
    \begin{equation}\label{EQ_ConvSeq&UncountTopsNotProd}
        (\{x_\infty\}\times D)\cap W = (\{x_\infty\}\times D)\cap \left(\bigcap_{\beta<\alpha_\infty}W_\beta\right) = (\{x_\infty\}\times D) \setminus \left(\bigcup_{\beta<\alpha_\infty}F_\beta\right)
    \end{equation}
    is uncountable. In particular, it follows that $\bob$ wins in $R$ if it ends at $\alpha=\alpha_\infty$.

    We claim that, indeed, $R$ must end at $\alpha=\alpha_\infty$.
    Striving for a contradiction, suppose not -- so that $\ali$ will be able to choose a non-empty open partition $\mP_{\alpha_\infty}$  with $p_{\alpha_\infty} = (x_\infty,y_\infty)$ (by definition of $\sigma$ and $\alpha_\infty$) at this inning.

    In this case, for each $(x_\infty, d)\in (\{x_\infty\}\times D)\cap W$, there must be a (unique) $U_d\in \mP_{\alpha_\infty}$ such that $(x_\infty, d)\in U_d$ (since $(x_\infty, d)\neq p_{\alpha_\infty}$). 
    Then let $n_d\in \omega$ be the least natural number for which $(x_k,d)\in U_d$ for every $k\ge n_d$ (using the fact that $\seq{x_n:n\in\omega}$ converges to $x_\infty$). 
    Thus, note that there must be an $N\in\omega$ such that $n_d = N$ for uncountably many $d\in D$, because $(\{x_\infty\}\times D)\cap W$ is uncountable.

    We claim that $(x_N, y_\infty)\in W$. 
    Indeed, suppose not: then, since $W$ is closed (being the intersection of clopen sets), there must be some open $V\subset Y$ containing $y_\infty$ such that $\{x_N\}\times V$ is in the complement of $W$. However, as $D\setminus V$ is finite, it would follow that $(\{x_N\}\times D) \cap W$ is finite, contradicting the definition of $N$.

    We can now use the fact that $(x_N, y_\infty)\in W$ and $p_{\alpha_\infty}\neq (x_N, y_\infty)$ to find some $V\subset Y$ containing $y_\infty$ such that $\{x_N\}\times V\subseteq U$ for some $U\in \mP_{\alpha_\infty}$. But, since $D\setminus V$ is finite, there must then be uncountably many $d\in D$ such that
    \[(x_N, d)\in (\{x_N\}\times V) \cap U_d\subseteq U\cap U_d.\]
    Thus, for uncountably many $d\in D$, $U_d = U$ (as $\mP_{\alpha_\infty}$ is a partition). 
    However, this means that $p_{\alpha_\infty} = (x_\infty, y_\infty)\in \overline{U}$, contradicting the fact that every element of $\mP_{\alpha_\infty}$ is clopen (since $p_{\alpha_\infty}\notin \bigcup\mP_{\alpha_\infty}$). 

    Hence, every run compatible with $\sigma$ must end at the $\alpha_\infty$-th inning and it is won by $\bob$. 
    So we have, indeed, defined a winning strategy for $\bob$.
\end{proof}

Now consider the following simple results:

\begin{prop}\label{PROP_DiscreteProduct}
    If $T'$ is a tree such that $\mR(T')$ is discrete, then $\mR(T)\times\mR(T')$ is a ray space for every tree $T$. 
\end{prop}
\begin{proof}
    Suppose $|\mR(T')|=\kappa$. Consider $\tilde{T}$ as the tree obtained by taking the disjoint union of $\kappa$-many copies of $T$ and then identifying all of their roots as the root of $\tilde{T}$ (see \myref{FIG_DiscreteProduct}).
    Then it follows that $\mR(T)\times\mR(T') \simeq \mR(\tilde{T})$, so $\mR(T)\times\mR(T')$ is a ray space.
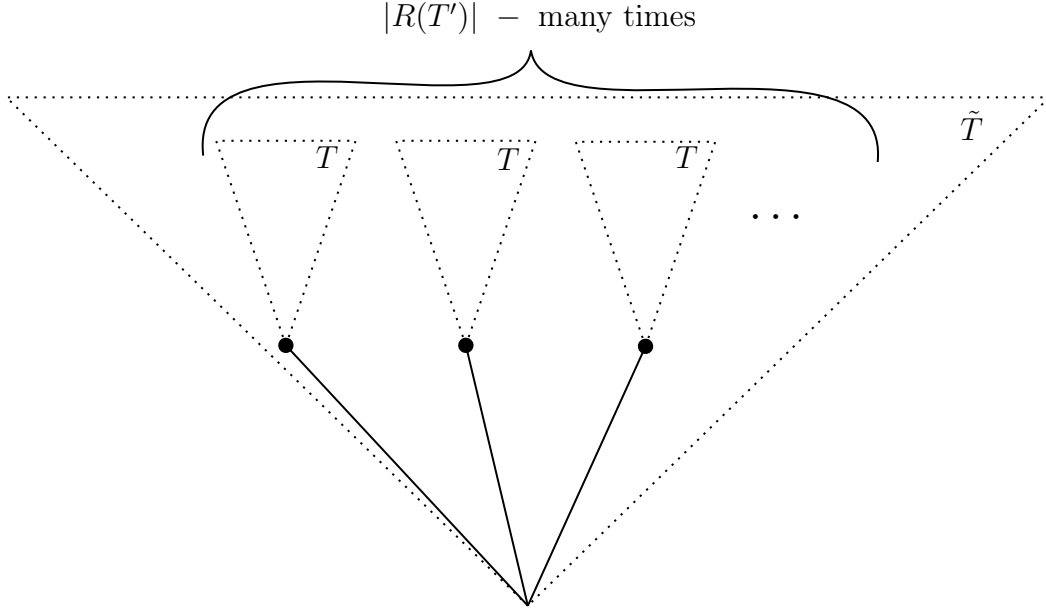
\begin{figure}
    \centering

\tikzset{every picture/.style={line width=0.75pt}} 

\begin{tikzpicture}[x=0.75pt,y=0.75pt,yscale=-1,xscale=1]

\draw  [dash pattern={on 0.84pt off 2.51pt}] (330.44,320.68) -- (68.73,65.55) -- (592.14,65.55) -- cycle ;
\draw    (209,190) -- (330.44,320.68) ;
\draw [shift={(209,190)}, rotate = 47.1] [color={rgb, 255:red, 0; green, 0; blue, 0 }  ][fill={rgb, 255:red, 0; green, 0; blue, 0 }  ][line width=0.75]      (0, 0) circle [x radius= 3.35, y radius= 3.35]   ;
\draw    (299.33,190) -- (330.44,320.68) ;
\draw [shift={(299.33,190)}, rotate = 76.61] [color={rgb, 255:red, 0; green, 0; blue, 0 }  ][fill={rgb, 255:red, 0; green, 0; blue, 0 }  ][line width=0.75]      (0, 0) circle [x radius= 3.35, y radius= 3.35]   ;
\draw  [dash pattern={on 0.84pt off 2.51pt}] (209,190) -- (174,87.33) -- (244,87.33) -- cycle ;
\draw    (389.5,190.5) -- (330.44,320.68) ;
\draw [shift={(389.5,190.5)}, rotate = 114.4] [color={rgb, 255:red, 0; green, 0; blue, 0 }  ][fill={rgb, 255:red, 0; green, 0; blue, 0 }  ][line width=0.75]      (0, 0) circle [x radius= 3.35, y radius= 3.35]   ;
\draw  [dash pattern={on 0.84pt off 2.51pt}] (299.33,190) -- (264.33,87.33) -- (334.33,87.33) -- cycle ;
\draw  [dash pattern={on 0.84pt off 2.51pt}] (389.5,190.5) -- (354.5,87.83) -- (424.5,87.83) -- cycle ;
\draw    (167.52,94.79) .. controls (159.32,21.87) and (325.93,87.06) .. (332,42) ;
\draw    (332,42) .. controls (339.46,89.43) and (512.82,27.43) .. (506,98) ;

\draw (440,119.4) node [anchor=north west][inner sep=0.75pt]  [font=\Large]  {$\cdots $};
\draw (222.5,88.4) node [anchor=north west][inner sep=0.75pt]    {$T$};
\draw (313.5,89.4) node [anchor=north west][inner sep=0.75pt]    {$T$};
\draw (403,88.4) node [anchor=north west][inner sep=0.75pt]    {$T$};
\draw (546,70.9) node [anchor=north west][inner sep=0.75pt]    {$\tilde{T}$};
\draw (256,16.4) node [anchor=north west][inner sep=0.75pt]    {$| R( T') |\ -\ \text{many times}$};

\end{tikzpicture}
    \caption{The tree $\tilde{T}$ whose ray space is homeomorphic to $\mR(T)\times\mR(T')$.}
    \label{FIG_DiscreteProduct}
\end{figure}
\end{proof}

\begin{prop}\label{PROP_DiscreteConvSequenceChara}
    For a tree $T$, $\mR(T)$ is discrete if and only if $\mR(T)$ contains no non-trivial convergent sequence.
\end{prop}
\begin{proof}
    It is clear that $\mR(T)$ being discrete implies that it contains  no non-trivial convergent sequence.

    Now suppose $\mR(T)$ contains no non-trivial convergent sequence and fix $R\in \mR(T)$ (and, without loss of generality, let us assume that $T$ is pruned).
    For each $t\in \Rtop(R)$, fix an $R_t$ containing $t$.
    Note that if $\set{R_t:t\in \Rtop(R)}$ were infinite, then any countable subset of it would determine a non-trivial sequence convergent to $R$.
    Thus, $\Rtop(R) = F_R$ is finite.

    We claim that we can find a $t_R\in R$ such that $[t_R,F_R] = \{R\}$, which concludes the proof.
    Indeed, striving for a contradiction, suppose not. Choose $t_0\in R$. Recursively, after $t_n$ has been chosen, select
    \[
        R_n\in[t_n,F_R]\setminus\{R\},
    \]
    and then choose $t_{n+1}\in R\setminus R_n$ above $t_n$. Thus $R_n\in[t_n,\{t_{n+1}\}]$ for every $n\in\omega$. The sequence $(R_n)_{n\in\omega}$ converges to
    \[
        R_\infty = \bigcup_{n\in\omega}\lceil t_n\rceil \subseteq R
    \]
    (since $\seq{t_n:n\in\omega}$ is cofinal in $R_\infty$),
    contradicting our hypothesis.
\end{proof}

\begin{cor}\label{COR_ConvSeq&UncountTopsNotProd}
    Suppose $T$ is a pruned tree with a ray containing uncountably many tops. Then the following are equivalent for a tree $T'$: 
    \begin{enumerate}[label=(\alph*)]
        \item\label{item_T'Discrete} $\mR(T')$ is discrete.
        \item\label{item_ProdisRaySpace} $\mR(T)\times \mR(T')$ is a ray space.
    \end{enumerate}
\end{cor}
\begin{proof}
     In view of \myref{PROP_DiscreteProduct}, it is clear that \ref{item_T'Discrete} implies \ref{item_ProdisRaySpace}.

    On the other hand, suppose $\mR(T)\times \mR(T')$ is a ray space. 
    If a ray $R_\infty\in \mR(T)$ has uncountably many tops in $T$, then fix, for each $t\in \Rtop(R_\infty)$, a ray $R_t$ containing $t$.
    In this case, it is clear that
    \[
        \mR(T)\supseteq\{R_\infty\} \cup \set{R_t:t\in\Rtop(R_\infty)}\simeq A_{|\Rtop(R_\infty)|}\supseteq A_{\aleph_1}.
    \]

    Thus, the conclusion follows directly from \myref{THM_ConvSeq&UncountTopsNotProd} and \myref{PROP_DiscreteConvSequenceChara}.
\end{proof}

\begin{cor}\label{COR_Aaleph1^2NotRay}
    If $X$ is a space containing a copy of $A_{\aleph_1}$, then $X^2$ is not a quasi-ray space. 
    In particular, $X^2$ is not a ray space.
\end{cor}

\begin{ex}\label{EX_ProdNotRay}
    Let $G$ be the graph obtained by taking the disjoint union of $\aleph_1$-many rays and then adding an edge between all pairs of starting vertices of such rays. Then $\Omega(G)\simeq A_{\aleph_1}$. 
    However, in view of Corollaries \ref{COR_QuasiRaySPaceChar} and \ref{COR_Aaleph1^2NotRay}, $(\Omega(G))^2 = (A_{\aleph_1})^2$ is not a quasi-ray space (let alone a ray or an end space).
\end{ex}

\begin{thm}\label{THM_ConvSeq&Omega1NotProd}
    Suppose that $X$ contains a non-trivial convergent sequence and that $Y$ contains a copy of $\omega_1+1$. Then $\bob$ has a winning strategy in $\PTreeGame(X\times Y)$.
\end{thm}
\begin{proof}
    Let us first fix an $x_\infty\in X$ which is the limit of a non-trivial convergent sequence $\seq{x_n:n\in\omega}$ in $X$, as well as a copy $\set{y_\alpha:\alpha\le \omega_1}\subseteq Y$ of $\omega_1+1$ (with the map $\alpha\mapsto y_\alpha$ defining a homeomorphism).

    We define $\bob$'s winning strategy $\sigma$ as follows: from the beginning of the game, let $\sigma$ choose the open sets in $\ali$'s partitions containing $(x_\infty,y_{\omega_1})$, up to the (limit) inning $\alpha_\infty$, at which no further open set in $\ali$'s partitions contains it (either because the game ended at this inning $\alpha_\infty$, or because $\ali$ chose $p_{\alpha_\infty} = (x_\infty,y_{\omega_1})$).  Afterwards, if the game is not finished at the inning $\alpha_\infty$, we may let $\bob$ pick arbitrary open sets.

    Striving for a contradiction, suppose $\ali$ wins in a finished run 
    \[R = \seq{\mP_0, W_0, \dotsc, (\mP_\gamma, p_\gamma),W_\gamma, \dotsc, (\emptyset, p_\alpha)},\]
    compatible with $\sigma$ (in particular, note that the length $\alpha$ of $R$ must be such that $\alpha\ge \alpha_\infty$ and $p_{\alpha_\infty} = (x_\infty, y_{\omega_1})$). 
    Let $W = \bigcap_{\beta<\alpha_\infty}W_\beta$.

    \begin{claim}\label{CLAIM_omega1Basis}
        $A = (\{x_\infty\}\times \set{y_\eta:\eta<\omega_1})\cap W$ is countable.
    \end{claim}
    \begin{proof}
        Since $R$ is assumed to be won by $\ali$, we have two cases to check:
        \begin{itemize}
            \item If $\alpha = \alpha_\infty$, then $W = \{(x_\infty,y_{\omega_1})\}$ and the claim is trivially true.
            \item Otherwise, striving for a contradiction, suppose $A$ is uncountable. 
            Then $A\simeq \omega_1$, as a subspace of $X\times Y$ (since $A$ is an uncountable closed subspace of $(\{x_\infty\}\times \set{y_\eta:\eta<\omega_1})\simeq\omega_1 $). 
            The collection
            \[
                \set{V\cap A:V\in\mP_{\alpha_\infty}\text{ and }V\cap A\neq\emptyset}
            \]
            is an open partition of $A$. By \myref{LEMMA_Omega1Partition}, there is a corresponding $V\in\mP_{\alpha_\infty}$ such that $A\setminus V$ is countable.
            However, since $V$ is closed, this means that $(x_\infty, y_{\omega_1})\in V\in \mP_{\alpha_\infty}$, contradicting our definition of $\alpha_\infty$.
        \end{itemize}
    \end{proof}

    Now, for every $\beta<\alpha_\infty$, the fact that $(x_\infty,y_{\omega_1})\in W_\beta$ (by definition of $\alpha_\infty$) implies that  $(\{x_\infty\}\times \set{y_\eta:\eta<\omega_1})\cap W_\beta$ is co-countable.
    In particular, in view of \myref{CLAIM_omega1Basis}, this tells us that $\alpha_\infty$ must have cofinality $\omega_1$.
    Furthermore, we can also find a $\gamma<\omega_1$ such that 
    \[(\{x_\infty\}\times \set{y_\eta:\gamma< \eta<\omega_1})\cap W = \emptyset.\]

    For each $\beta<\alpha_\infty$, let $n_\beta\in \omega$ be the least for which $(x_k, y_{\omega_1})\in W_\beta$ for every $k\ge n_\beta$ (which can be found, since $\seq{x_n:n\in\omega}$ converges to $x_\infty$ in $X$ and $(x_\infty, y_{\omega_1})\in W_\beta$ for every $\beta<\alpha_\infty$).
    Note that $\seq{n_\beta:\beta<\alpha_\infty}$ is an increasing sequence of natural numbers. 
    Thus, since $\alpha_\infty$'s cofinality is $\omega_1$, it must follow that there is some $\xi<\alpha_\infty$  such that $n_\beta=n_\xi=N$ for every $\beta>\xi$.
    Hence,
    \[\set{x_k:k\ge N}\times \{y_{\omega_1}\} \subseteq W. \]
    In particular, since we are assuming $R$ is a run won by $\ali$, this means that $\alpha>\alpha_\infty$ and thus $\mP_{\alpha_\infty}\neq\emptyset$.
    
    Now, since $p_{\alpha_\infty}\neq (x_k, y_{\omega_1})\in W$ for each $k\ge N$, there must be a $\gamma<\theta_k<\omega_1$ and a $V_k\in \mP_{\alpha_\infty}$ such that
    \[(\{x_k\}\times \set{y_\eta:\theta_k< \eta\le\omega_1}) \subseteq V_k.\]
    In this case, let $\theta = \sup\set{\theta_k+1:k\ge N}<\omega_1$. 
    Note that, by construction, $\theta\ge \gamma$. 
    In particular, $(x_k,y_\theta)\in W$ for every $k\ge N$.
    However, since $W$ is closed (being the intersection of clopen sets), this would entail that $(x_\infty, y_\theta)\in W$, contradicting the fact that $\theta\ge\gamma$ and the definition of $\gamma$.

    Hence, no such run $R$ won by $\ali$ can exist.
\end{proof}

\begin{cor}\label{COR_ConvSeq&Omega1NotProd}
    Suppose $T$ is a pruned tree with an uncountable ray. Then the following are equivalent for a tree $T'$: 
    \begin{enumerate}[label=(\alph*)]
        \item\label{item_T'Discrete_omega1} $\mR(T')$ is discrete.
        \item\label{item_ProdisRaySpace_omega1} $\mR(T)\times \mR(T')$ is a ray space.
    \end{enumerate}
\end{cor}
\begin{proof}
    In view of \myref{PROP_DiscreteProduct}, it is clear that \ref{item_T'Discrete_omega1} implies \ref{item_ProdisRaySpace_omega1}.

    Now suppose $\mR(T)\times \mR(T')$ is a ray space. If a ray $R\in \mR(T)$ is uncountable, then we find a copy of $\omega_1+1$ in $\mR(T)$ by considering $\mR(R)\subseteq \mR(T)$, if $R\subseteq T$ has order type $\omega_1$, or, otherwise, $\mR(\lceil t \rceil)\subseteq \mR(T)$ for some $t\in R$ such that the set of strict predecessors $\lceil\mathring{t}\rceil$ has order type $\omega_1$ (and hence $\lceil t\rceil$ has order type $\omega_1+1$).

    Thus, the conclusion follows directly from \myref{THM_ConvSeq&Omega1NotProd} and \myref{PROP_DiscreteConvSequenceChara}.
\end{proof}

\begin{cor}
    If $X$ is a space containing a copy of $\omega_1+1$, then $X^2$ is not a quasi-ray space.
    In particular, $X^2$ is not a ray space.
\end{cor}

\begin{ex}
    Despite $\omega_1+1$ being a ray space, $(\omega_1+1)^2$ is not a quasi-ray space (let alone a ray space).
\end{ex}

In view of Corollaries \ref{COR_ConvSeq&UncountTopsNotProd} and \ref{COR_ConvSeq&Omega1NotProd}, it is only natural to ask whether the complement of these presented classes of trees is stable under finite products. 
In other words: 

\begin{question}\label{QUESTION_FirsCountRayProd0}
    Suppose $T$ and $T'$ are trees in which
    \begin{enumerate}[label=(\roman*)]
        \item every branch is countable,
        \item every ray has at most countably many tops.
    \end{enumerate}
    Is $\mR(T)\times \mR(T')$, then, a ray space?
\end{question}

Or, equivalently:

\begin{question}\label{QUESTION_FirsCountRayProd}
    Suppose $X$ and $Y$ are first-countable ray spaces. Is $X\times Y$, then, a ray space? 
\end{question}

In the case of end spaces, it was shown in Theorem 3.1 of \cite{pitz2023characterisingpathraybranch} that every metrizable end space is homeomorphic to the ray space of a tree of height $\omega$, thus:

\begin{prop}
    If $X$ and $Y$ are metrizable end spaces, then $X \times Y$ is a metrizable end space.
\end{prop}

Hence, when restricted to end spaces, the investigation of \myref{QUESTION_FirsCountRayProd} reduces to non-metrizable end spaces, such as (see Corollary 2.18 in \cite{pitz2023characterisingpathraybranch}):

\begin{question}
    If $T$ is a special Aronszajn tree, is $\mR(T)^2$ an end (or a ray) space?
\end{question}
\section{Final remarks}\label{SEC_FinalRmk}

We note that, throughout this paper, there was only one example of a non-ray space (and a non-quasi-ray space) for which no winning strategy for $\bob$ in $\TreeGame(X)$ (or $\PTreeGame(X)$) was provided: \myref{EX_NoIsolatedLessThanContinuum}. Thus, in view of the characterizations presented in \myref{COR_RaySPaceChar} and \myref{COR_QuasiRaySPaceChar}, it is natural to ask whether these games are always \emph{determined}. That is:

\begin{question}
    Is there a space $X$ such that neither $\ali$ nor $\bob$ has a winning strategy in $\TreeGame(X)$? Furthermore, what is the status of determinacy for $\PTreeGame(X)$?
\end{question}


\begin{thebibliography}{10}

\bibitem{maxflowmincut}
Ron Aharoni, Eli Berger, Agelos Georgakopoulos, Amitai Perlstein, and Philipp
  Spr{\"u}ssel.
\newblock The {Max}-{Flow} {Min}-{Cut} theorem for countable networks.
\newblock {\em J. Comb. Theory, Ser. B}, 101(1):1--17, 2011.

\bibitem{PauloGustavoDavide}
Leandro Aurichi, Gustavo Boska, Davide Giacopello, and Paulo~Magalhães
  Júnior.
\newblock A topological characterization of end space of infinite graphs via
  games, subspaces and products.
\newblock \href{https://arxiv.org/abs/2604.17164}{arXiv:2604.17164}, 2026.

\bibitem{aurichi2024topologicalremarksendedgeend}
Leandro Aurichi, Paulo Magalhães-Júnior, and Lucas Real.
\newblock Topological remarks on end and edge-end spaces.
\newblock \href{https://arxiv.org/abs/2404.17116}{arXiv:2404.17116}, 2024.

\bibitem{aurichilucasedgeconectividade}
Leandro Aurichi and Lucas Real.
\newblock Edge-connectivity between edge-ends of infinite graphs.
\newblock {\em J. Graph Theory}, 109(4):454--465, 2025.

\bibitem{cycle-cocycle}
Henning Bruhn, Reinhard Diestel, and Maya Stein.
\newblock Cycle-cocycle partitions and faithful cycle covers for locally finite
  graphs.
\newblock {\em J. Graph Theory}, 50(2):150--161, 2005.

\bibitem{carvalho2026coveringpropertiesendray}
Rodrigo Carvalho, Matheus Duzi, and Vinicius Rodrigues.
\newblock On covering properties of end and ray spaces.
\newblock \href{https://arxiv.org/abs/2510.10825}{arXiv:2510.10825}, 2026.

\bibitem{Diestelquestion}
Reinhard Diestel.
\newblock The end structure of a graph: {Recent} results and open problems.
\newblock {\em Discrete Math.}, 100(1-3):313--327, 1992.

\bibitem{engelking}
Ryszard Engelking.
\newblock {\em General Topology}, volume~6 of {\em Sigma Ser. Pure Math.}
\newblock Berlin: Heldermann Verlag, rev. and compl. ed. edition, 1989.

\bibitem{edge-ends}
Ge{\v{n}}a Hahn, Fran{\c{c}}ois Laviolette, and Jozef {\v{S}}ir{\'a}{\v{n}}.
\newblock Edge-ends in countable graphs.
\newblock {\em J. Comb. Theory, Ser. B}, 70(2):225--244, 1997.

\bibitem{halin}
Rudolf Halin.
\newblock \"{U}ber unendliche {W}ege in {G}raphen.
\newblock {\em Mathematische Annalen}, 157:125--137, 1964/65.

\bibitem{endspacesandtreedecompositions}
Marcel Koloschin, Thilo Krill, and Max Pitz.
\newblock End spaces and tree-decompositions.
\newblock {\em J. Comb. Theory, Ser. B}, 161:147--179, 2023.

\bibitem{approximating}
Jan Kurkofka, Ruben Melcher, and Max Pitz.
\newblock Approximating infinite graphs by normal trees.
\newblock {\em J. Comb. Theory, Ser. B}, 148:173--183, 2021.

\bibitem{representationtheorem}
Jan Kurkofka and Max Pitz.
\newblock A representation theorem for end spaces of infinite graphs.
\newblock {\em Mathematical Proceedings of the Cambridge Philosophical
  Society}, pages 1--27, 2026.

\bibitem{pitz2023characterisingpathraybranch}
Max Pitz.
\newblock Characterising path-, ray- and branch spaces of order trees, and end
  spaces of infinite graphs.
\newblock \href{https://arxiv.org/abs/2303.00547}{arXiv:2303.00547}, 2023.

\bibitem{metrizationtheoremforedgeends}
Max Pitz.
\newblock A metrization theorem for edge-end spaces of infinite graphs.
\newblock {\em Proc. Am. Math. Soc.}, 154(5):2209--2219, 2026.

\bibitem{real2025subbasepropertydescribingedgeend}
Lucas Real.
\newblock A subbase property for describing edge-end spaces.
\newblock \href{https://arxiv.org/abs/2508.17424}{arXiv:2508.17424}, 2025.

\bibitem{arboricityandtree-packing}
Maya Stein.
\newblock Arboricity and tree-packing in locally finite graphs.
\newblock {\em J. Comb. Theory, Ser. B}, 96(2):302--312, 2006.

\end{thebibliography}
\end{document}